\documentclass[12pt,
]{amsart}

\usepackage{graphicx} 
\usepackage[margin=1in, marginparsep = 0pt, marginparwidth = 72pt
]{geometry}
\usepackage{subcaption} 
\usepackage[mathscr]{eucal} 
\usepackage{amssymb} 
	\usepackage{layout} 
	\usepackage{comment} 
\usepackage{hyperref} 
\usepackage{tikz-cd} 

\newtheorem{theorem}{Theorem}
\newtheorem{proposition}[theorem]{Proposition}
\newtheorem{lemma}[theorem]{Lemma}
\newtheorem{claim}[theorem]{Claim}
\newtheorem{corollary}[theorem]{Corollary}
\newtheorem{mainlemma}[theorem]{Main Lemma}
\newtheorem{fact}[theorem]{Fact}

\newtheorem{innercustomlem}{Lemma}
\newenvironment{customlem}[1]
  {\renewcommand\theinnercustomlem{#1}\innercustomlem}
  {\endinnercustomlem}

\theoremstyle{definition}
\newtheorem{definition}[theorem]{Definition}
\newtheorem{remark}[theorem]{Remark}
\newtheorem{question}[theorem]{Question}

\renewcommand\subset{\subseteq}
\renewcommand\emptyset{\varnothing}
\renewcommand{\geq}{\geqslant}
\renewcommand{\leq}{\leqslant}

\newcommand{\Z}{\mathbb{Z}}
\newcommand{\Zpos}{\Z_{\geq 0}}
\newcommand{\Q}{\mathbb{Q}}

\newcommand{\surf}{\mathfrak{S}}
\newcommand{\surfbord}{\widehat{\mathfrak{S}}}
\newcommand{\idealtriang}{\lambda}
\newcommand{\splitidealtriang}{\widehat{\lambda}}
\newcommand{\biang}{\mathfrak{B}}
\newcommand{\triang}{\mathfrak{T}}
\newcommand{\poly}{\mathfrak{D}}
\newcommand{\E}{\mathscr{E}}

\newcommand{\themap}[1]{\Phi_{#1}}
\newcommand{\invmap}[1]{\Psi_{#1}}
\newcommand{\invmaptilde}[1]{\widetilde{\Psi}_{#1}}
\newcommand{\Cone}{\mathscr{C}}
\newcommand{\KTcone}[1]{\Cone^+_{#1}}
\newcommand{\W}{\mathscr{W}}
\newcommand{\webbasis}[1]{\W_{#1}}
\newcommand{\FG}{\mathrm{FG}}

\title[Tropical coordinates for webs]{Tropical Fock-Goncharov coordinates \\ for $\mathrm{SL}_3$-webs on surfaces I: construction}
\author[D. C. Douglas]{Daniel C. Douglas}
\address{Department of Mathematics, Virginia Tech, 225 Stanger Street, Blacksburg, VA 24061, USA}
\email{dcdouglas@vt.edu}
\author[Z. Sun]{Zhe Sun}
\address{Key Laboratory of Wu Wen-Tsun Mathematics, Chinese Academy of Sciences;
School of Mathematical Sciences, University of Science and Technology of China, 96 Jinzhai Road, 230026 Hefei, Anhui, China}
\email{sunz@ustc.edu.cn}
\date{\today}
\thanks{
This work was partially supported by the U.S. National Science Foundation grants DMS-1107452, 1107263, 1107367 ``RNMS: GEometric structures And Representation varieties'' (the GEAR Network).  The first author was also partially supported by the U.S. National Science Foundation grants DMS-1406559 and 1711297, and the second author by the China Postdoctoral Science Foundation grant 2018T110084, the FNR AFR Bilateral grant COALAS 11802479-2, and the Huawei Young Talents Program at IHES}

\begin{document}
\begin{abstract}
	For a finite-type surface $\surf$, we study a preferred basis for the commutative algebra $\mathbb{C}[\mathscr{R}_{\mathrm{SL}_3(\mathbb{C})}(\surf)]$ of regular functions on the $\mathrm{SL}_3(\mathbb{C})$-character variety, introduced by  Sikora-Westbury.  These basis elements come from the trace functions associated to certain tri-valent graphs embedded in the surface $\surf$.  We show that this basis can be naturally indexed by non-negative integer coordinates, defined by Knutson-Tao rhombus inequalities and modulo 3 congruence conditions.  These coordinates are related, by the geometric theory of Fock and Goncharov, to the tropical points at infinity of the dual version of the character variety.  
\end{abstract}
\maketitle

\section{Introduction}
For a finitely generated group $\Gamma$ and a suitable Lie group $G$, a primary object of study in low-dimensional geometry and topology is the character variety
\begin{equation*}
	\mathscr{R}_G(\Gamma) = \left\{  \rho : \Gamma \longrightarrow G  \right\} /\!\!/ \hspace{2.5pt} G
\end{equation*}
consisting of group homomorphisms $\rho$ from $\Gamma$ to $G$, considered up to conjugation.  Here, the quotient is taken in the algebraic geometric sense of Geometric Invariant Theory \cite{Mumford94}.  Character varieties can be explored using a wide variety of mathematical skill sets.  Some examples include the Higgs bundle approach of Hitchin \cite{HitchinTopology92}, the dynamics approach of Labourie \cite{LabourieInvent06}, and the representation theory approach of Fock-Goncharov \cite{FockIHES06}.

We are interested in the case where the group $G$ is the special linear group $\mathrm{SL}_n(\mathbb{C})$.  Adopting the viewpoint of algebraic geometry, one can study the $\mathrm{SL}_n(\mathbb{C})$-character variety $\mathscr{R}_{\mathrm{SL}_n(\mathbb{C})}(\Gamma)$ by means of its commutative algebra of regular functions $\mathbb{C}[\mathscr{R}_{\mathrm{SL}_n(\mathbb{C})}(\Gamma)]$.  An example of a regular function is the trace function $\mathrm{Tr}_\gamma : \mathscr{R}_{\mathrm{SL}_n(\mathbb{C})}(\Gamma) \to \mathbb{C}$ associated to an element $\gamma \in \Gamma$, sending a representation $\rho$ to the trace $\mathrm{Tr}(\rho(\gamma)) \in \mathbb{C}$ of the matrix $\rho(\gamma) \in \mathrm{SL}_n(\mathbb{C})$.  A theorem of Procesi \cite{ProcesiAdvMath76} implies that the trace functions $\mathrm{Tr}_\gamma$ generate the algebra of functions $\mathbb{C}[\mathscr{R}_{\mathrm{SL}_n(\mathbb{C})}(\Gamma)]$ as an algebra, and also identifies all of the relations.

Sikora \cite{SikoraTrans01} provided a more refined description of Procesi's result in the case where $\Gamma = \pi_1(\mathfrak{X})$ is the fundamental group of a topological space $\mathfrak{X}$; see also the earlier work of Bullock \cite{BullockCommentMathHelv97} for the case $G=\mathrm{SL}_2(\mathbb{C})$.  Sikora  extended the notion of a trace function to include functions $\mathrm{Tr}_W \in \mathbb{C}[\mathscr{R}_{\mathrm{SL}_n(\mathbb{C})}(\mathfrak{X})]$ on the character variety $\mathscr{R}_{\mathrm{SL}_n(\mathbb{C})}(\mathfrak{X}) := \mathscr{R}_{\mathrm{SL}_n(\mathbb{C})}(\pi_1(\mathfrak{X}))$ that are associated to homotopy classes of certain (ciliated) oriented $n$-valent graphs $W$, called webs, in the space $\mathfrak{X}$.  The trace functions $\mathrm{Tr}_W$ span the algebra of functions $\mathbb{C}[\mathscr{R}_{\mathrm{SL}_n(\mathbb{C})}(\mathfrak{X})]$ as a vector space, and the relations are described pictorially in terms of the associated graphs.  

In this article, we restrict attention to the case where the Lie group is $\mathrm{SL}_3(\mathbb{C})$ and the space $\mathfrak{X}=\surf$ is a punctured finite-type surface.  Sikora-Westbury \cite{SikoraAlgGeomTop07} proved that the collection of trace functions $\mathrm{Tr}_W$ associated to non-elliptic webs $W$, which are certain webs embedded in the surface $\surf $, forms a linear basis for the algebra of functions $\mathbb{C}[\mathscr{R}_{\mathrm{SL}_3(\mathbb{C})}(\surf)]$.  

An analogous result \cite{HosteKnotTheoryRam93} in the case of $\mathrm{SL}_2(\mathbb{C})$ says that the collection of trace functions $\mathrm{Tr}_\gamma$ associated to essential multi-curves $\gamma$ embedded in the surface $\surf $ forms a linear basis for the algebra of functions $\mathbb{C}[\mathscr{R}_{\mathrm{SL}_2(\mathbb{C})}(\surf)]$.  A well-known topological-combinatorial fact says that if the punctured surface $\surf$ is equipped with an ideal triangulation $\idealtriang$, then the geometric intersection numbers $\iota(\gamma, E)$ of a curve $\gamma$ with the edges $E$ of $\idealtriang$ furnish an explicit system of non-negative integer coordinates on the collection of essential multi-curves $\gamma$.  These coordinates can be characterized by finitely many triangle inequalities and parity conditions.  

The present work is part of a series of two papers, whose goal is to generalize these $\mathrm{SL}_2$-properties to the case $n=3$.  The main result of the current paper is the following.  

\begin{theorem}
\label{thm:first-theorem-intro}
	For a punctured finite-type surface $\surf$ equipped with an ideal triangulation $\idealtriang$, the Sikora-Westbury $\mathrm{SL}_3$-web basis for the algebra of functions $\mathbb{C}[\mathscr{R}_{\mathrm{SL}_3(\mathbb{C})}(\surf)]$ admits an explicit system of non-negative integer coordinates, which can be characterized by finitely many Knutson-Tao rhombus inequalities {\upshape\cite{KnutsonJAmerMathsoc99}} and modulo $3$ congruence conditions.  
\end{theorem}

In the companion article \cite{DouglasArxiv20b}, we prove that the web coordinates from Theorem \ref{thm:first-theorem-intro} are natural with respect to the action of the mapping class group of the surface $\surf$.

\begin{theorem}[{\cite{DouglasArxiv20b}}]
\label{thm:second-theorem-intro}
	  If another ideal triangulation $\idealtriang^\prime$ of $\surf$ is chosen, then the induced coordinate transformation takes the form of a tropicalized \hbox{$\mathcal{A}$-coordinate} cluster transformation (as opposed to $\mathcal{X}$-coordinate), in the language of Fock-Goncharov {\upshape\cite{FockIHES06, FominJAmerMathSoc02}}.
\end{theorem}

Strictly speaking, Theorems \ref{thm:first-theorem-intro} and \ref{thm:second-theorem-intro} have been stated assuming that the punctured surface $\surf$ has empty boundary.    In \S \ref{sec:webs-on-surfaces-with-boundary}, we give two different, but related, generalizations (Theorems \ref{thm:main-theorem-v1} and \ref{thm:main-theorem-2}) of Theorem \ref{thm:first-theorem-intro} valid in the boundary setting, $\partial \surf \neq \emptyset$; see also \cite{KimArxiv20}.  There we also provide applications to the geometry and topology of $\mathrm{SL}_3(\mathbb{C})$-character varieties, as well as to the representation theory of the Lie group $\mathrm{SL}_3(\mathbb{C})$.  In the companion article \cite{DouglasArxiv20b}, we likewise provide a version of Theorem \ref{thm:second-theorem-intro} valid in the boundary setting.  

This work drew much inspiration from papers of Xie \cite{XieArxiv13}, Kuperberg \cite{KuperbergCommMathPhys96}, and Goncharov-Shen \cite{GoncharovInvent15}. 

At its heart, Theorem \ref{thm:first-theorem-intro} simply describes how to assign tuples of numbers to pictures.  We have motivated these web pictures $W$ by their association with trace functions $\mathrm{Tr}_W$. As such, it is desirable to tie directly the coordinates to the trace functions.  Such a relationship is well-known for $\mathrm{SL}_2(\mathbb{C})$; see  \cite{Fock07}, for instance.  In that case, the trace functions $\mathrm{Tr}_\gamma$ for curves $\gamma$ can be expressed as Laurent polynomials $\mathrm{Tr}_\gamma = \mathrm{Tr}_\gamma(X_i)$ in variables $X_i$ where there is one variable per coordinate (that is, per edge $E_i$ of $\idealtriang$).  Moreover, the coordinates of a curve $\gamma$ can be read off as the exponents of the highest term of the trace polynomial $\mathrm{Tr}_\gamma(X_i)$, demonstrating the tropical geometric nature of these coordinates.  

There is a similar story for $\mathrm{SL}_3(\mathbb{C})$, and conjecturally for $\mathrm{SL}_n(\mathbb{C})$. The Fock-Goncharov theory tells us how to express the trace functions $\mathrm{Tr}_W$ for webs $W$ as Laurent polynomials $\mathrm{Tr}_W(X_i)$.  Here, the number of variables $X_i$ (called Fock-Goncharov coordinates) increases with $n$. In the case $n=3$, Kim \cite{KimArxiv20}, building on \cite{DouglasThesis20}, showed that the tropical coordinates of Theorem \ref{thm:first-theorem-intro} appear as the exponents of the highest term of the Fock-Goncharov trace polynomial $\mathrm{Tr}_W(X_i)$.  This idea was Xie's \cite{XieArxiv13} point of departure, and these coordinates were constructed following his lead.  

Kuperberg's landmark paper \cite{KuperbergCommMathPhys96} influenced \cite{SikoraAlgGeomTop07} and laid the topological foundation for the present work as well.  He proved that a certain collection of web pictures drawn on an ideal polygon $\poly_k$ indexes a linear basis for the sub-space  of $\mathrm{SL}_3(\mathbb{C})$-invariant tensors in a $k$-fold tensor product of finite-dimensional irreducible representations of $\mathrm{SL}_3(\mathbb{C})$.  Along the way, he showed how the pictures for the ideal polygon $\poly_k$ can be obtained by gluing together the more basic pictures for an ideal triangle $\poly_3$.  We apply Kuperberg's local pictorial ideas in order to analyze global web pictures drawn on a triangulated surface $(\surf, \idealtriang)$.  

Motivated by the Fock-Goncharov Duality Conjecture \cite{FockIHES06} (see also \cite{GrossJAmerMathSoc18, GoncharovAdvMath18}), Goncharov-Shen \cite{GoncharovInvent15} developed a theory by which bases of algebras of functions on moduli spaces, defined abstractly via the geometric Satake correspondence, can be indexed by positive integral tropical points, namely the preimage points mapping to $\Zpos$ under a  tropicalized potential function.  They showed that, for an ideal triangle $\poly_3$ equipped with a general linear symmetry group, the positive integral tropical points correspond to solutions of the Knutson-Tao rhombus inequalities.  In the $\mathrm{SL}_3$-setting of this article, Theorem \ref{thm:first-theorem-intro} also makes use of these inequalities in order to assign positive integer coordinates to webs.  We think of this result as a manifestation of Goncharov-Shen's ideas about duality; see \cite{DouglasArxiv20b} for a further discussion.  (For another geometric application of the Goncharov-Shen potential function, see \cite{MR4595283huangsun}.)

Frohman-Sikora \cite{MR4359515frohmansikora} independently constructed coordinates for the same $\mathrm{SL}_3$-web basis as that appearing in Theorem \ref{thm:first-theorem-intro}.  While their topological strategy is the same,  their coordinates are different from ours.  They do not characterize by inequalities the values taken by their coordinates, and they do not address the question of naturality under changing the triangulation. Their proof is algebraic, as it uses the Sikora-Westbury theorem (discussed above)  saying that the non-elliptic webs are linearly independent, which ultimately relies on the Diamond Lemma from non-commutative algebra.  On the other hand, we give a purely topological-combinatorial proof of Theorem \ref{thm:first-theorem-intro}, which does not require using this linear independence.  Moreover, we give an alternative geometric proof of this Sikora-Westbury theorem, by using Theorem \ref{thm:first-theorem-intro} together with the $\mathrm{SL}_3$-quantum trace map \cite{DouglasThesis20, KimArxiv20}.  (Ishibashi-Kano \cite{Ishibashi22} mimicked the construction and proof strategy of Theorem \ref{thm:first-theorem-intro} to define shearing coordinates for unbounded $\mathrm{SL}_3$-laminations.)

As another application, Kim \cite{KimArxiv20,kim2021mutation} used Theorems \ref{thm:first-theorem-intro} and \ref{thm:second-theorem-intro} to study a classical and quantum $\mathrm{SL}_3$-version of Fock-Goncharov duality, generalizing the $\mathrm{SL}_2$-case \hbox{\cite{FockIHES06, AllegrettiAdvMath17}}.  

For the underlying $\mathrm{SL}_3$-geometry, see \cite{FockAdvMath07, CasellaMathSocJapMem20}.  

We are also interested in comparing our methods to other approaches to studying webs and related objects,  falling under the umbrella of so-called ``higher laminations''.  In addition to webs  \cite{SikoraAlgGeomTop05, FontaineAdvMath2012, CautisMathAnn14}, this includes  cluster algebras \cite{FominAdvMath16,MR4552137ishibashiyuasa}, buildings \cite{FontaineCompositio13, LeGeomTop16, MartoneMathZ19}, and spectral networks \cite{GaiottoAnnHenriPoincare13,neitzke2022quantum}.

\section*{Acknowledgements}
This research would not have been possible without the help and support of many people, whose time and patience often seem limitless.  In particular, we are profoundly grateful to Dylan Allegretti for his involvement during the early stages of this project; Charlie Frohman for his guidance from the very beginning; Francis Bonahon and Viktor Kleen for helping us refine our ideas and for technical assistance; Tommaso Cremaschi for his invaluable feedback after reading way too many drafts; as well as Vijay Higgins, Hyun Kyu Kim,  Linhui Shen, and Adam Sikora for helpful conversations.  Much of this work was completed during very enjoyable visits to Tsinghua University in Beijing (supported by a GEAR graduate internship grant) and the University of Southern California in Los Angeles.  We would like to take this opportunity to extend our enormous gratitude to these institutions for their warm hospitality and many tasty dinners (the first author was especially fond of the s\={o}ngsh\v{u} gu\`{i}y\'{u}).  We also thank the referee for helping us to improve the paper.

		\section{Global webs}
		\label{sec:global-webs}

We introduce the primary topological objects of study.

		\subsection{Topological setting}
		\label{ssec:topological-setting}

Let $\surf$ be an oriented \textit{punctured surface} of finite topological type, namely $\surf$ is diffeomorphic to the space obtained by removing a finite subset $P$, called the set of \textit{punctures}, from a closed oriented surface $\overline{\surf}$.   In particular, note that $\surf$ has empty boundary, $\partial \surf = \emptyset$.  We require that there is at least one puncture, and that the Euler characteristic $\chi(\surf)$ of the punctured surface $\surf$ is strictly less than zero, $\chi(\surf) < 0$.  These topological conditions guarantee the existence of an \textit{ideal triangulation} $\idealtriang$ of the punctured surface $\surf$, namely a triangulation $\overline{\idealtriang}$ of the closed surface $\overline{\surf}$ whose vertex set is equal to the set of punctures $P$.  See Figure \ref{fig:example-ideal-triangulations} for some examples of ideal triangulations.

To simplify the exposition, we always assume that $\idealtriang$ does not contain any \textit{self-folded triangles}, meaning that each triangle $\triang$ of $\idealtriang$ has three distinct edges.  Such a $\idealtriang$ always exists.  Our results should generalize, essentially without change, to allow for self-folded triangles.  

\begin{figure}[htb]
     \centering
     \begin{subfigure}{0.35\textwidth}
         \centering
         \includegraphics[width=.76\textwidth]{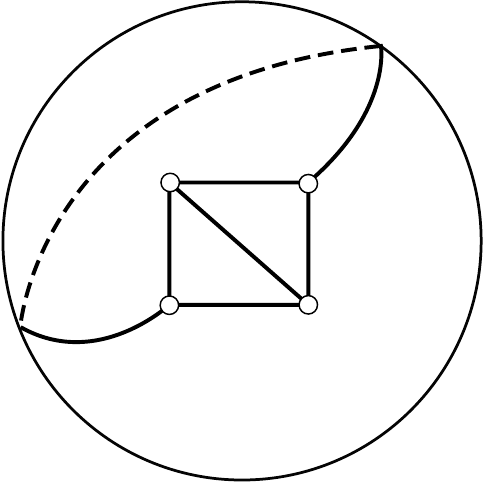}
         \caption{Four times punctured sphere}
         \label{fig:four-times-punctured-sphere-triang}
     \end{subfigure}     
\hfill
     \begin{subfigure}{0.35\textwidth}
         \centering
         \includegraphics[width=\textwidth]{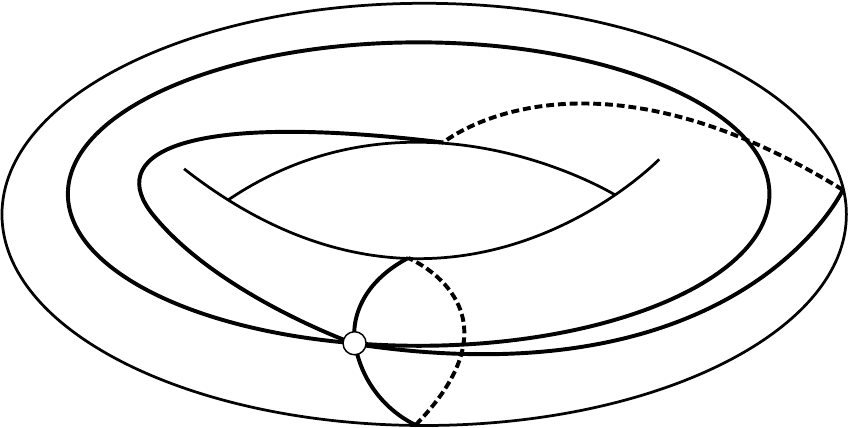}
         \caption{Once punctured torus}
         \label{fig:once-punctured-torus}
     \end{subfigure}
        \caption{Ideal triangulations}
        \label{fig:example-ideal-triangulations}
\end{figure}

		\subsection{Webs}
		\label{ssec:webs}

\begin{definition}
\label{def:curve}
	An \textit{immersed curve}, or just \textit{curve}, $\gamma$ in any surface (possibly with boundary) $\surfbord$ is an immersion into $\surfbord$ of the circle $S^1$ or the compact interval $I$.  In other words, a curve is either an oriented \textit{loop} (that is, a closed curve) or an oriented \textit{arc}, possibly with self-intersections.  
\end{definition}

We will often be working with \textit{embedded curves}, where there are no self-intersections.
	
\begin{definition}
\label{def:web-on-a-surface}
	An \textit{embedded global web}, or just \textit{global web} or \textit{web}, $W = \{ w_i \}_i$ on the surface $\surf$ is a finite collection of closed connected oriented tri-valent (finite) graphs or closed curves $w_i$ embedded in $\surf$, such that the (images of the) components $w_i$ are mutually disjoint, and such that each vertex of $w_i$ is either a \textit{source} or a \textit{sink}, namely the orientations either go all in or all out, respectively.  Note that the web $W$ has empty boundary, $\partial W = \emptyset$.  
\end{definition}

	For an example, in Figure \ref{fig:example-web} we show a web on the once punctured torus, which has four components consisting of two tri-valent graphs and two curves.  

\begin{figure}[t]
	\centering
	\includegraphics[scale=1.07]{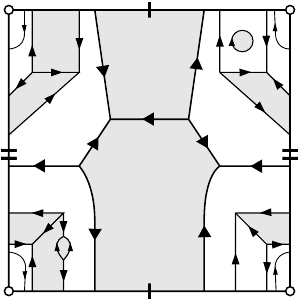}
	\caption{Web}
	\label{fig:example-web}
\end{figure}

\begin{definition}
\label{def:parallel-equivalent-webs}
	Two webs $W$ and $W^\prime$ on the surface $\surf$ are \textit{parallel-equivalent} if  $W$ can be taken to $W^\prime$, preserving orientation, by a sequence of moves of the following two types:
\begin{enumerate}
	\item an \textit{isotopy} of the web, namely a smoothly varying family of webs;
	\item a \textit{global parallel-move}, exchanging two loops that together form the boundary of an embedded annulus $A$ in the surface $\surf$; see Figure \ref{fig:parallel-move}.  
\end{enumerate}
In this case, we say that $W$ and $W^\prime$ belong to the same \textit{parallel-equivalence class} $[W] = [W^\prime]$. 
\end{definition}

Intuitively, we think of parallel-equivalent as meaning homotopic on the surface. 

\begin{figure}[htb]
	\centering
	\includegraphics[scale=1.11]{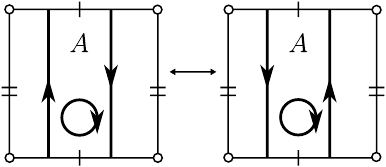}
	\caption{Global parallel-move}
	\label{fig:parallel-move}
\end{figure}

		\subsection{Faces}
		\label{ssec:faces}

\begin{definition}
\label{def:web-face-closed-surface}
	A \textit{face} $D$ of a web $W$ on the surface $\surf$ is a contractible component of the complement $W^c \subset \surf$ of the web.  A \textit{$n$-face} $D_n$ is a face with $n$-sides, counted with multiplicity.  An alternative name for a $0$-face $D_0$, $2$-face $D_2$, $4$-face $D_4$, and $6$-face $D_6$ is a \textit{disk-}, \textit{bigon-}, \textit{square-}, and \textit{hexagon-face}, respectively.   
\end{definition}

For an example, the web shown in Figure \ref{fig:example-web} above has one disk-face, one bigon-face, two square-faces, and two hexagon-faces; these faces are shaded in the figure.  Notice that one of the hexagon-faces consists of five edges of the web, one edge being  counted twice.  

 By orientation considerations, faces must have an even number of sides.  

Bigon- and square-faces always consist of exactly two and four edges, respectively, of $W$.  See Figure~\ref{fig:prohibited-4-face}.

 In figures, we often omit the web orientations, as in Figure \ref{fig:local-webs-examples}.

\begin{figure}[htb]
	\centering
	\includegraphics[scale=.35]{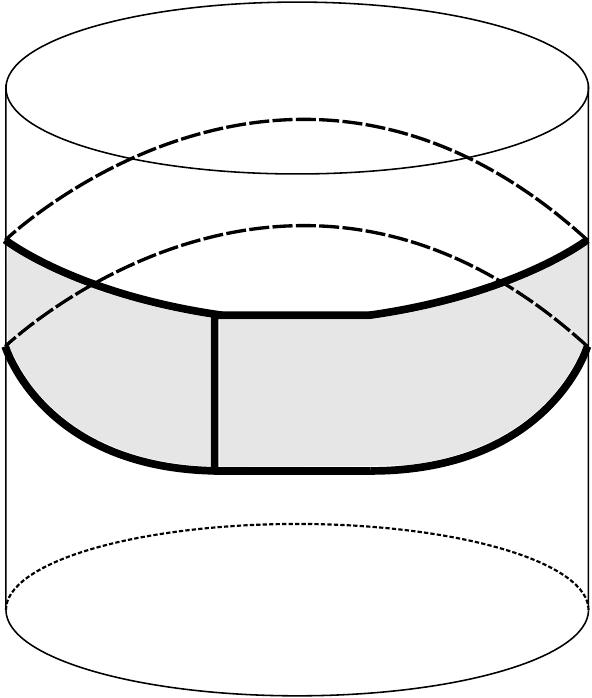}
	\caption{Prohibited square-face}
	\label{fig:prohibited-4-face}
\end{figure}

		\subsection{Non-elliptic webs}
		\label{ssec:essential-webs}

\begin{definition}
\label{def:essential-global-webs}
	A web $W$ on the surface $\surf$ is called \textit{non-elliptic} if it has no disk-, bigon-, or square-faces.  Otherwise, $W$ is called \textit{elliptic}.  
\end{definition}

If $W$ is non-elliptic, and if $W^\prime$ is parallel-equivalent to $W$, then $W^\prime$ is non-elliptic.  Denote the set of non-elliptic webs by $\webbasis{\surf}$, and the set of parallel-equivalence classes of non-elliptic webs by $[\webbasis{\surf}]$.  The \textit{empty web} $W = \emptyset$ represents a class with one element in $[\webbasis{\surf}]$.

		\section{Local webs}
		\label{sec:local-webs}
As a technical device, we study webs-with-boundary in the disk.

		\subsection{Ideal polygons}
		\label{ssec:ideal-polygons}

For a non-negative integer  $k \geq 0$, an \textit{ideal $k$-polygon} $\poly_k$ is the surface $\poly_0 - P$ obtained by removing $k$ punctures $P \subset \partial \poly_0$ from the boundary of the closed disk~$\poly_0$.  

Observe that, when $k > 0$, the boundary $\partial \poly_k$ of the ideal polygon consists of $k$ ideal arcs.

		\subsection{Local webs}	
		\label{ssec:local-webs}

Recall the notion of a curve (Definition \ref{def:curve}).  
	
\begin{definition}
\label{def:local-web-on-a-surface}
	An \textit{embedded local web}, or just \textit{local web}, $W = \{ w_i \}_i$ in an ideal polygon $\poly_k$ is a finite collection of connected oriented tri-valent graphs or curves $w_i$ embedded in $\poly_k$, such that the components $w_i$ are mutually disjoint, and such that each vertex of $w_i$ is either a source or sink.  Note that the local web $W$ may have boundary, in which case we require $\partial W = W \cap \partial \poly_k$ and we consider each point $v \in \partial W$ to be a mono-valent vertex.  
\end{definition}

For some examples of local webs, see Figure \ref{fig:local-webs-examples}.  There, $k=4$.  

\begin{figure}[htb]
	\centering
	\begin{subfigure}{.49\textwidth}
		\centering
		\includegraphics[scale=1.19]{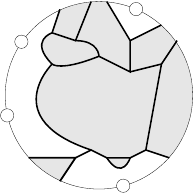}
		\caption{Connected}
		\label{web-with-boundary2}
	\end{subfigure}
	\begin{subfigure}{.49\textwidth}
		\centering
		\includegraphics[scale=1.19]{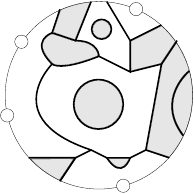}
		\caption{Not connected}
		\label{web-with-boundary1}
	\end{subfigure}
	\caption{Local webs}
	\label{fig:local-webs-examples}
\end{figure}

		\subsection{External faces}
		\label{ssec:external-faces}

\begin{definition}
\label{def:external-web-face-closed-surface}
A \textit{face} $D$ of a local web $W$ in an ideal polygon $\poly_k$ ($k \geq 0$) is a contractible component of the complement $W^c \subset \poly_k$ of $W$ that is \textit{puncture-free}, meaning that $D$ does not limit to any punctures $p \in P$.  A \textit{$n$-face} $D_n$ is a face with $n$-sides.  Here, a maximal segment $\alpha \subset (\partial \poly_k) \cap D_n$ of the boundary $\partial \poly_k$ contained in the face $D_n$ is counted as a side, called a \textit{boundary side}.  An \textit{external face} $D^\mathrm{ext}$ (resp. \textit{internal face} $D^\mathrm{int}$) of the local web $W$ is a face having at least one (resp. no) boundary side.  

In contrast to internal faces, external faces can have an odd number of sides.  An alternative name for an external $2$-face $D^\mathrm{ext}_2$, $3$-face $D^\mathrm{ext}_3$, $4$-face $D^\mathrm{ext}_4$ with one boundary side, and $5$-face $D^\mathrm{ext}_5$ with one boundary side is a \textit{cap-}, \textit{fork-}, \textit{H-}, and \textit{half-hexagon-face}, respectively; see Figure \ref{fig:external-small-faces}.  Also, as for global webs (see Definition \ref{def:web-face-closed-surface}), an alternative name for an internal $0$-face $D^\mathrm{int}_0$, $2$-face $D^\mathrm{int}_2$, $4$-face $D^\mathrm{int}_4$, and $6$-face $D^\mathrm{int}_2$ is a \textit{disk-}, \textit{bigon-}, \textit{square-}, and \textit{hexagon-face}.  
\end{definition}

For example, the connected local web in Figure \ref{web-with-boundary2} has one fork-face, two H-faces, one half-hexagon-face, one external 6-face, one bigon-face, one square-face, and one internal 8-face.  Also, the disconnected local web in Figure \ref{web-with-boundary1} has one cap-face, one fork-face, one H-face, one half-hexagon-face, one external 6-face, two disk-faces, one bigon-face, and one square-face.  

\begin{figure}[htb]
	\centering
	\includegraphics[scale=.74]{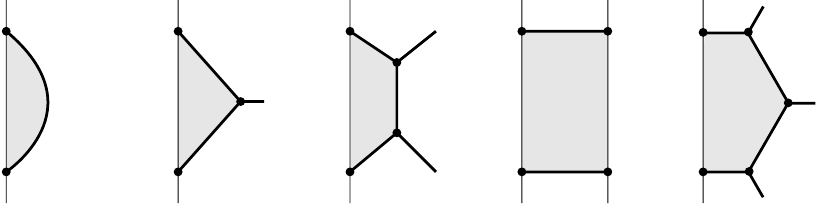}
	\caption{Cap-, fork-, H-, external 4-, and half-hexagon-face}
	\label{fig:external-small-faces}
\end{figure}

		\subsection{Combinatorial identity}
		\label{ssec:singular-flat-metric-associated-to-a-local-web}

\begin{proposition}[{compare \cite[\S 6.1]{KuperbergCommMathPhys96}} ]
\label{fact:GB}	
	Let $W$ be a connected local web in the closed disk $\poly_0$ with non-empty boundary $\partial W \neq \emptyset$.  Then,
\begin{equation*}	2 \pi = 
	\sum_{\textnormal{internal faces } D^\mathrm{int}_n} \left( 2\pi - \frac{\pi}{3} n
	\right)
\quad+\quad
	\sum_{\textnormal{external faces } D^\mathrm{ext}_n} \left( \pi - \frac{\pi}{3} \left( n-2 \right) \right). 
\end{equation*}
\end{proposition}

\begin{proof}
Since $W$ is connected, its complement $W^c \subset \poly_0$ contains at most one annulus, which faces the boundary $\partial \poly_0$.  Such an annulus does not exist, since $\partial W \neq \emptyset$.  Thus, every component $D$ of $W^c$ is contractible, and of course puncture-free, so $D$ is a face.  

It follows that the closed disk $\poly_0$ can be tiled by the dual graph of $W$.  More precisely, the vertices of the dual graph are the faces of $W$, and the complement of the dual graph consists of triangles.    In Figure \ref{fig:example-web-tiling}, we demonstrate this tiling procedure for the local web $W$ that we saw in Figure \ref{web-with-boundary2} above (after forgetting the punctures).  

This triangular tiling gives rise to a flat Riemannian metric with conical singularities and piecewise-geodesic boundary on the closed disk $\poly_0$, by requiring that each triangle is Euclidean equilateral.  Apply the Gauss-Bonnet theorem to this singular flat surface.  
\end{proof}

\begin{figure}[htb]
	\centering
	\includegraphics[scale=1.10]{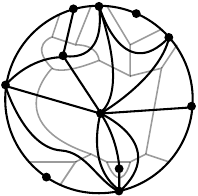}
	\caption{Tiling the closed disk with the dual graph of a local web}
	\label{fig:example-web-tiling}
\end{figure}

		\subsection{Non-ellipticity}
		\label{ssec:non-ellipticity}

\begin{definition}
\label{def:non-elliptic-web-with-boundary}
	As for global webs, a local web $W$ in an ideal polygon $\poly_k$ is \textit{non-elliptic} if $W$ has no disk-, bigon-, or square-faces.  Otherwise, $W$ is called \textit{elliptic}; see Figure \ref{fig:small-webs-with-boundary-in-the-disk}.
\end{definition}

\begin{lemma}
\label{fact:boundary-non-elliptic-web-regularity}
	Let $W$ be a non-elliptic local web in the closed disk $\poly_0$ such that $W$ is connected, has non-empty boundary $\partial W \neq \emptyset$, and has at least one tri-valent vertex.  Then $W$ has at least three fork- and/or H-faces.  
\end{lemma}

\begin{proof}
	We apply the formula of Proposition \ref{fact:GB}.  For each internal face $D_n^\mathrm{int}$ of $W$, the internal angle $2\pi - (\pi/3) n \leq 0$ is non-positive, since $n \geq 6$ by non-ellipticity.  For each external face $D_n^\mathrm{ext}$, necessarily $n \geq 2$, and the external angle $\pi - (\pi/3)(n-2)$ is $\leq 0$ is non-positive if and only if $n \geq 5$.  By hypothesis, $W$ has no cap-faces (else $W$ would be an arc).  So, those external faces $D^\mathrm{ext}_n$ with a positive contribution satisfy $n =3,4$.  The result follows since fork- and $H$-faces contribute $2 \pi/3$ and $\pi/3$, respectively, in the formula.  
\end{proof}

\begin{lemma}
\label{fact:no-closed-non-elliptic webs}
	Non-elliptic local webs $W$ ($\neq \emptyset$) in an ideal polygon $\poly_k$ ($k \geq 0$) having empty boundary $\partial W = \emptyset$ do not exist.
\end{lemma}

\begin{proof}
	Suppose otherwise.  We may assume $W$ is connected.  Since $W$ is non-elliptic, $W$ is not a loop (this uses that $\poly_k$ is contractible).  Then, the outer rim of $W$ forms the boundary of a smaller closed disk $\poly_0^\prime \subset \poly_k$ containing a sub-web $W^\prime \subset W$ that has non-empty boundary $\partial W^\prime \neq \emptyset$.  By non-ellipticity, $W^\prime$ does not have a cap-face, so $W^\prime$ has a tri-valent vertex.  Applying Lemma \ref{fact:boundary-non-elliptic-web-regularity} to connected components of $W^\prime$, an analysis of inner-most components leads to the fact that $W^\prime$ has at least one fork- or H-face.  By non-ellipticity, $W^\prime$ does not have an H-face, and it does not have a fork-face by orientation considerations applied to~$W$.  
\end{proof}

Lemma \ref{fact:no-closed-non-elliptic webs} plus a small argument allows us to relax the hypotheses of Lemma \ref{fact:boundary-non-elliptic-web-regularity} as follows.  

\begin{proposition}  
\label{fact:strengthened-boundary-non-elliptic-web-regularity}
	Let $W$ be a non-elliptic local web in the closed disk $\poly_0$ such that $W$ is connected and has at least one tri-valent vertex.  Then $W$ has at least three fork- and/or H-faces.   If, in addition, $W$ is assumed not to have any cap-faces, then the connectedness hypothesis above is superfluous.  \qed
\end{proposition}

\begin{figure}[htb]
     \centering
     \begin{subfigure}{0.25\textwidth}
         \centering
         \includegraphics[scale=.25]{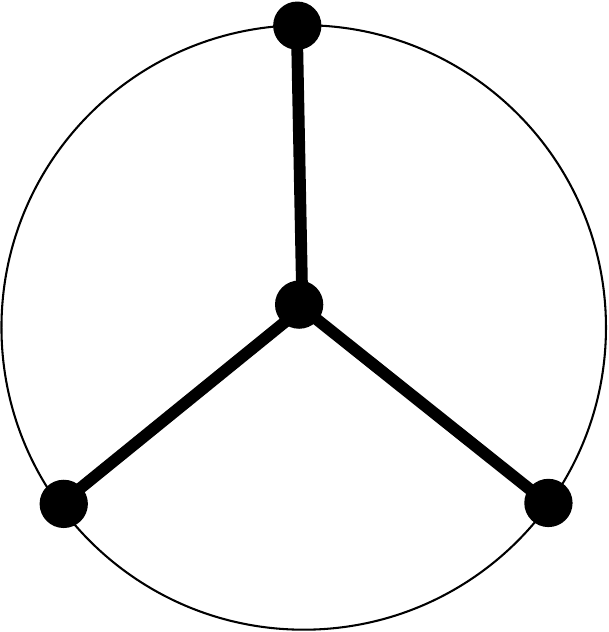}
         \caption{Three fork-faces}
         \label{subfig:web-with-three-forks}
     \end{subfigure}     
\hfill
     \begin{subfigure}{0.4\textwidth}
         \centering
         \includegraphics[scale=.25]{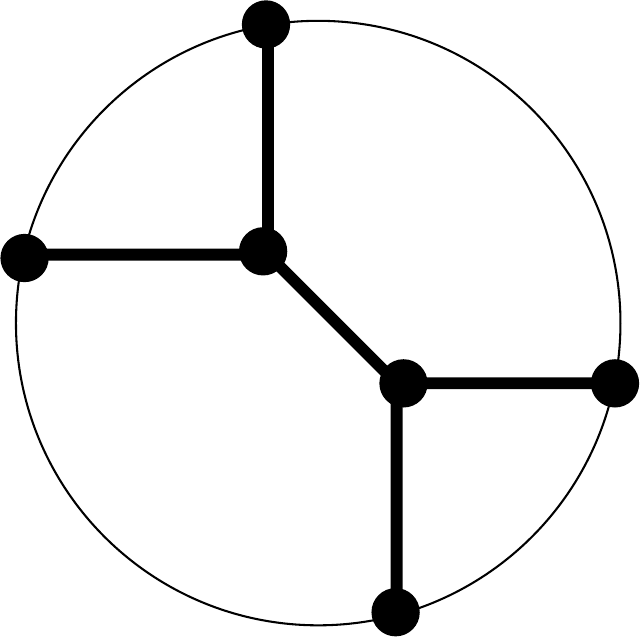}
         \caption{two fork- and two H-faces}
         \label{subfig:web-with-two-forks}
     \end{subfigure}    
    \hfill
     \begin{subfigure}{0.25\textwidth}
         \centering
         \includegraphics[scale=.25]{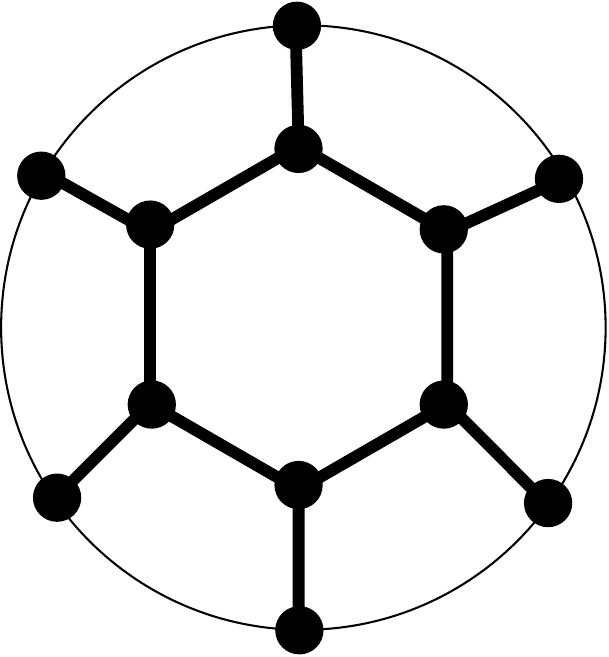}
         \caption{Six H-faces}
         \label{subfig:web-with-six-Hs}
     \end{subfigure}
         \caption{Non-elliptic local webs in the closed disk}
        \label{fig:small-webs-with-boundary-in-the-disk}
\end{figure}

		\subsection{Essential and rung-less local webs}
		\label{ssec:essential-rung-less-local-webs}

\begin{definition}
\label{def:essential-web-with-boundary}
	A local web $W$ in an ideal polygon $\poly_k$ ($k \geq 0$) is \textit{essential} if:
\begin{enumerate}
	\item  the local web $W$ is non-elliptic;
	\item  the web $W$ is \textit{taut}: for any compact arc $\alpha$ embedded in $\poly_k$ whose boundary $\partial \alpha$ lies in a component $E$ of the boundary $\partial \poly_k$ (and is disjoint from $W$), the number of intersection points $\iota(W, \overline{E})$ of $W$ with the segment $\overline{E} \subset E$ delimited by $\partial \alpha$ does not exceed the number of intersection points $\iota(W, \alpha)$ of $W$ with $\alpha$, that is $\iota(W, \overline{E}) \leq \iota(W, \alpha)$; see Figures \ref{fig:definition-of-essential} and \ref{fig:examples-and-non-exmaples-of-essential-webs}.  
\end{enumerate}
\end{definition}

\begin{figure}[t]
	\centering
	\includegraphics[scale=.32]{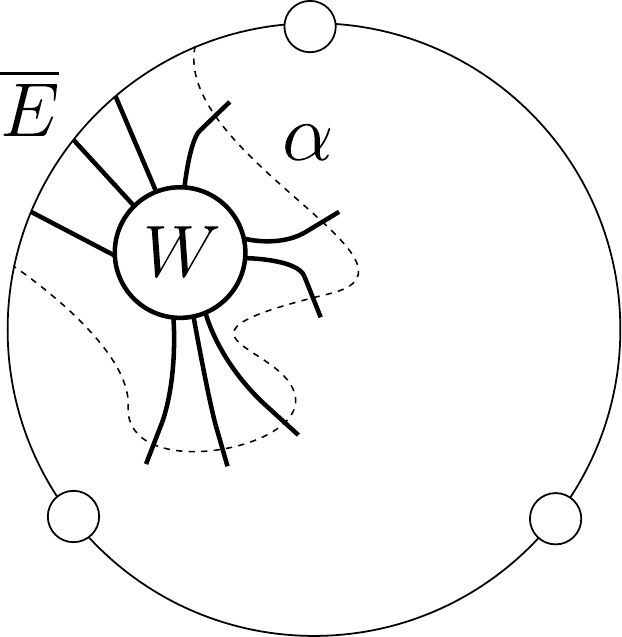}
	\caption{Tautness condition for an essential local web}
	\label{fig:definition-of-essential}
\end{figure}

\begin{figure}[t]
     \centering
     \begin{subfigure}{0.4\textwidth}
         \centering
         \includegraphics[scale=.18]{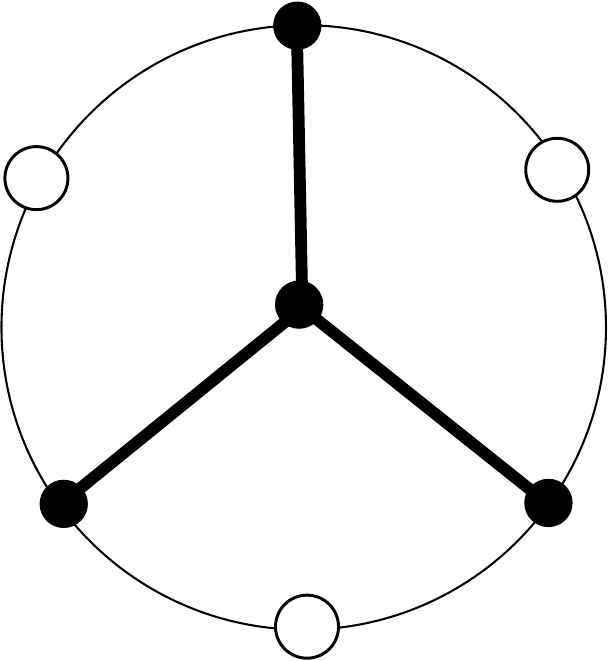}
         \caption{Essential}
         \label{subfig:essential-web-example}
     \end{subfigure}     
\hfill
     \begin{subfigure}{0.4\textwidth}
         \centering
         \includegraphics[scale=.18]{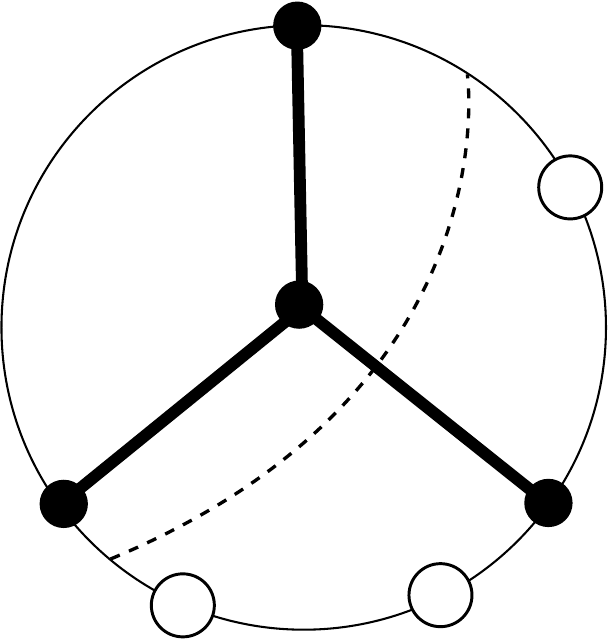}
         \caption{Non-elliptic, but not essential}
         \label{subfig:non-elliptic-not-essential-example}
     \end{subfigure}
     	\caption{More non-elliptic webs}
        \label{fig:examples-and-non-exmaples-of-essential-webs}
\end{figure}

\begin{figure}[t]
	\centering
	\includegraphics[scale=.26]{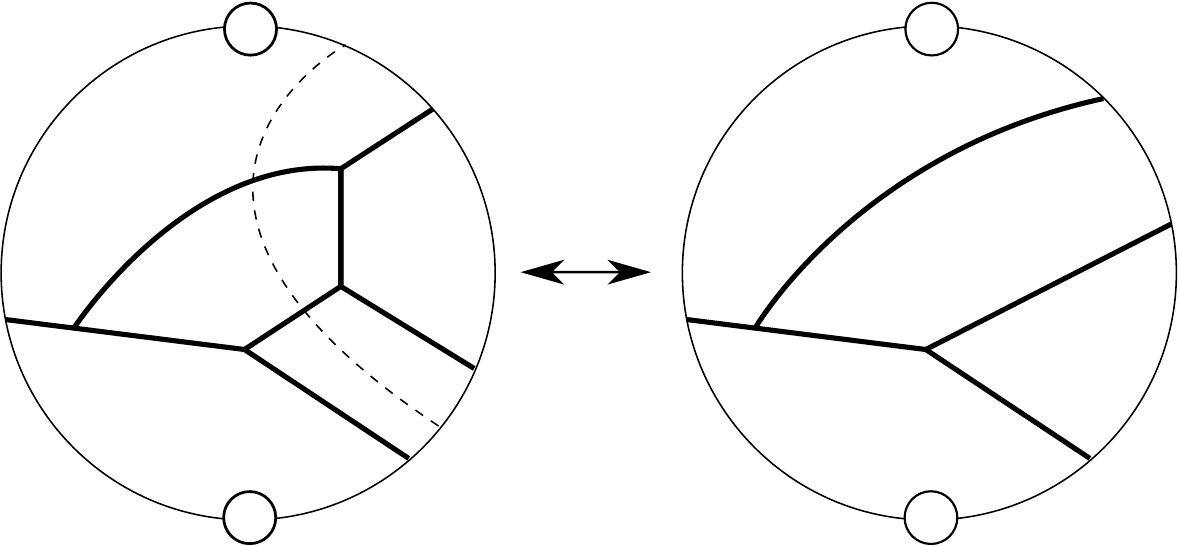}
	\caption{Adding or removing an H-face}
	\label{fig:add-remove-H}
\end{figure}

Note that essential local webs cannot have any cap- or fork-faces, but can have H-faces.  Later, we will need the operation of adding or removing an H-face, depicted in Figure \ref{fig:add-remove-H}.

\begin{definition}
\label{def:rung-less-webs}
	A local web $W$ in an ideal polygon $\poly_k$ ($k \geq 0$) is \textit{rung-less} if it does not have any H-faces; see Figure \ref{fig:Examples and non-examples of rung-less webs}.
\end{definition}

\begin{remark}   $ $
\label{rem:no-essential-webs-in-small-polygons}
\begin{enumerate}
	\item  A consequence of Proposition \ref{fact:strengthened-boundary-non-elliptic-web-regularity}, which we will not use, is that (non-empty) essential local webs in the closed disk $\poly_0$ or ideal monoangle $\poly_1$ do not exist.
	\item\label{subrem:kuperberg1}  Kuperberg \cite[\S 4, 6.1]{KuperbergCommMathPhys96} says ``(core of a) non-convex non-elliptic web in the $k$-clasped web space'' for our ``(rung-less) essential local web in the ideal $k$-polygon''.
	\end{enumerate}
\end{remark}

\begin{figure}[htb]
     \centering
     \begin{subfigure}{0.33\textwidth}
         \centering
         \includegraphics[scale=.26]{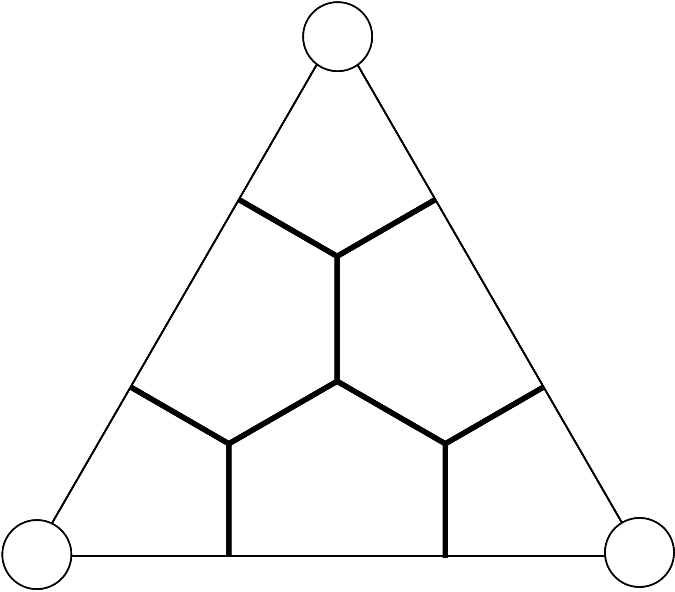}
         \caption{Essential and rung-less}
         \label{subfig:rung-less-essential-web}
     \end{subfigure}     
\hfill
     \begin{subfigure}{0.4\textwidth}
         \centering
         \includegraphics[scale=.26]{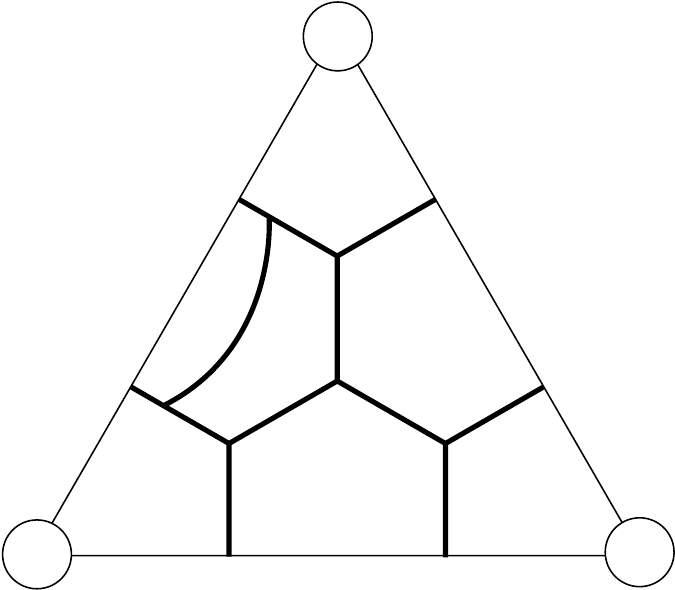}
         \caption{Essential, but not rung-less}
         \label{subfig:essential-web-not-rung-less}
     \end{subfigure}
     \caption{More essential webs}
        \label{fig:Examples and non-examples of rung-less webs}
\end{figure}

		\subsection{Ladder-webs in ideal biangles}
		\label{ssec:ladder-webs-in-ideal-biangles}

Another name for an ideal 2-polygon $\poly_2$ is an \textit{ideal biangle}, or just \textit{biangle}, denoted by $\biang$.  The boundary $\partial \biang$ consists of two ideal arcs $E^\prime$ and $E^{\prime\prime}$, called the \textit{boundary edges} of the biangle.  We want to characterize essential local webs $W$ in the biangle $\biang$; compare (1) in Remark \ref{rem:no-essential-webs-in-small-polygons}.  

\begin{definition}
\label{def:multi-curve}
	For any surface $\surfbord$, possibly with boundary, an \textit{immersed multi-curve}, or just \textit{multi-curve}, $\Gamma = \{ \gamma_i \}$ on $\surfbord$ is a finite collection of connected oriented curves (Definition \ref{def:curve}) $\gamma_i$ immersed in $\surfbord$, such that $\partial \gamma_i = \gamma_i \cap \partial \surfbord$.  Note that $\gamma_i$ and $\gamma_j$ might intersect in $\surfbord$ for any $i$ and $j$.  Note also that a component $\gamma_i$ may be either a loop or an arc.  
\end{definition} 

In the current section, components $\gamma_i$ of a multi-curve $\Gamma$ will always be embedded, but different components might intersect.  This will not be the case later on, in \S \ref{sec:proof-of-main-lemma}.  

A pair of arcs $\gamma_1$ and $\gamma_2$ each intersecting both boundary edges in $\biang$ are \textit{oppositely-oriented}  if $\gamma_1$ and $\gamma_2$ go into (resp. out of) and out of (resp. into) $E^\prime$, respectively.  Similarly, the arcs $\gamma_1$ and $\gamma_2$ are \textit{same-oriented} if $\gamma_1$ and $\gamma_2$ both go into (resp. out of) $E^\prime$, respectively.  

\begin{definition}
\label{def:ladder-web}
A \textit{symmetric strand-set pair} $S = (S^\prime, S^{\prime\prime})$ for the biangle $\biang$ is a pair of finite collections $S = (S^\prime, S^{\prime \prime}) =  ( \{s^\prime\}, \{s^{\prime\prime}\})$ of disjoint oriented \textit{strands} located on the boundary $\partial \biang = E^\prime \cup E^{\prime\prime}$, such that the strands $s^\prime$ (resp. $s^{\prime\prime}$) lie on the boundary edge $E^\prime$ (resp. $E^{\prime\prime}$), and such that the number of \textit{in-strands} (resp. \textit{out-strands}) on $E^\prime$ is equal to the number of out-strands (resp. in-strands) on $E^{\prime\prime}$; see the left-most picture in Figure \ref{fig:ladder-webs}.    
\end{definition}

Given a symmetric strand-set pair $S = (S^\prime, S^{\prime\prime})$, in the following definition we associate to $S$ a multi-curve in the biangle $\biang$, denoted $\left< W(S) \right>$.  

\begin{definition}
\label{def:local-picture-extra-def}
The \textit{local picture $\left< W(S) \right>$ associated to a symmetric strand-set pair $S= (S^\prime, S^{\prime\prime})$} is the multi-curve  in the biangle $\biang$ obtained by connecting the strands on $E^\prime$ to the strands on $E^{\prime\prime}$ with arcs, in an order preserving and  minimally intersecting way, loosely speaking, as illustrated in the middle picture in Figure \ref{fig:ladder-webs}.  Here, order preserving means in a way such that no same-oriented arcs intersect.  
\end{definition}

Observe, in the local picture $\left< W(S) \right>$, that $\gamma_1$ and $\gamma_2$ intersect if and only if (1) they are oppositely-oriented, and (2) they intersect exactly once.  We denote by $\mathscr{P}(S) \subset \biang$ the set of intersection points $p$ of pairs of oppositely-oriented arcs in the local picture  $\left< W(S) \right>$.  

Finally, we say how to associate a local web $W(S)$ in $\biang$ to a symmetric pair $S=(S^\prime, S^{\prime\prime})$.  

\begin{definition}
\label{def:actual-def-of-ladder-web}
The \textit{ladder-web} $W(S)$ in the biangle $\biang$ obtained from a symmetric strand-set pair $S= (S^\prime, S^{\prime\prime})$ is the unique (up to ambient isotopy of $\biang$) local web obtained by resolving each intersection point $p \in \mathscr{P}(S)$ into two vertices connected by a horizontal edge, relative to the biangle, called a \textit{rung}; see Figures  \ref{fig:ladder-webs} and \ref{fig:birds-eye-view}.  
\end{definition}

The following statement is implicit in \cite[Lemma 6.7]{KuperbergCommMathPhys96} and also appears in \cite[\S 8]{MR4359515frohmansikora}.  

\begin{figure}[htb]
	\centering
	\includegraphics[scale=1.11]{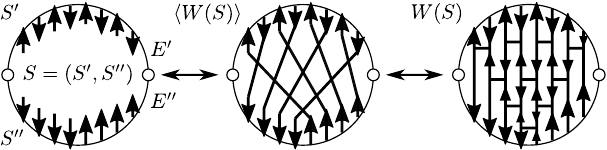}
	\caption{Construction of a ladder-web}
	\label{fig:ladder-webs}
\end{figure}

\begin{figure}[htb]
	\centering
	\includegraphics[scale=1.09]{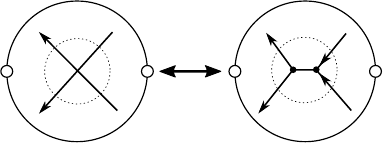}
	\caption{Replacing a local crossing with an H (also called a rung)}
	\label{fig:birds-eye-view}
\end{figure}

\begin{proposition}
\label{prop:ladder-webs}
	The ladder-web $W(S)$ is essential.  Conversely, given an essential local web $W$ in the biangle $\biang$, there exists a unique symmetric strand-set pair $S= (S^\prime, S^{\prime\prime})$ such that $W = W(S)$.  Thus, $W$ is a ladder-web.  
\end{proposition}

\begin{figure}[b]
	\centering
	\includegraphics[scale=.57]{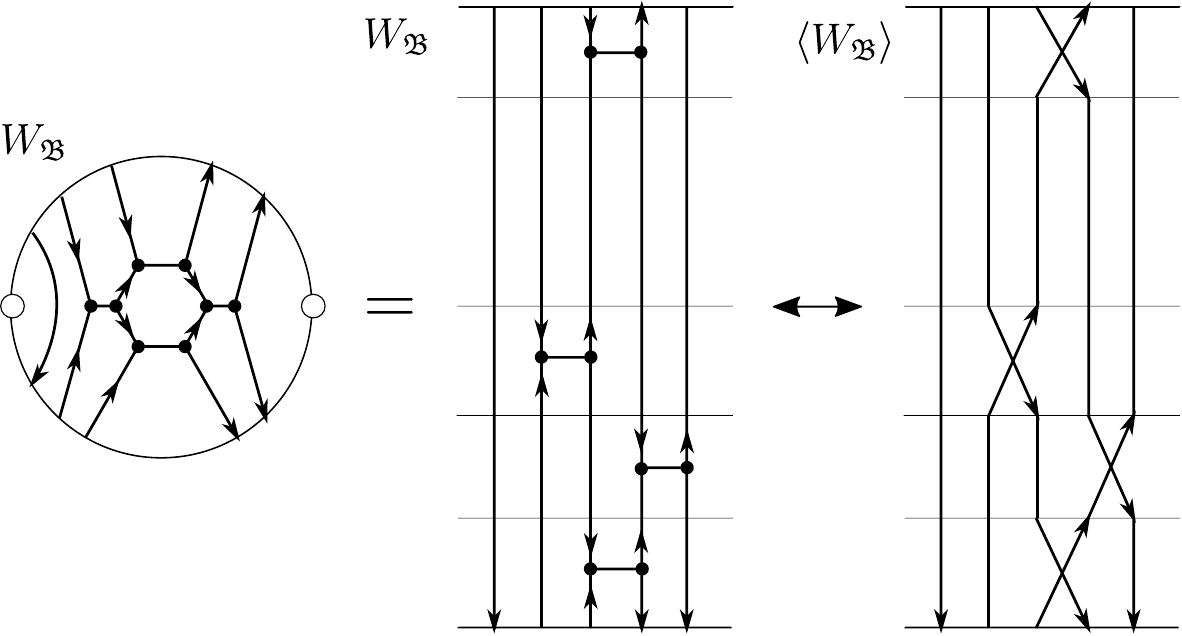}
	\caption{Essential local web $W_\biang$ in the biangle, and its corresponding local picture $\left< W_\biang \right>$}
	\label{fig:decomposing-essential-web-bigon}
\end{figure}

\begin{proof}	
	For the first statement, the non-ellipticity of $W(S)$ follows because two oppositely-oriented curves in the local picture $\left< W(S) \right>$ do not cross more than once (if there were a square-face, a pair of curves would cross twice), and the tautness of $W(S)$ is immediate.  
	
Conversely, let $W$ be an essential local web in $\biang$.  The collection of ends of $W$ located on the boundary edges $E^\prime \cup E^{\prime\prime}$ determines a strand-set pair $S= (S^\prime, S^{\prime\prime})$.  We show that $S$ is symmetric and $W = W(S)$.  In particular, $S$ is uniquely determined.  
	
	If $W$ has a tri-valent vertex, let $\overline{W}$ denote the induced local web in the closed disk $\poly_0$ underlying $\biang$, obtained by filling in the two punctures of $\biang$.  Applying Proposition \ref{fact:strengthened-boundary-non-elliptic-web-regularity} to $\overline{W}$ guarantees that $\overline{W}$ (possibly minus some arc components) has at least three fork- and/or H-faces.  At most two of these faces can straddle the two punctures of $\biang$, so we gather $W$ has one fork- or H-face $D^\mathrm{ext}$ lying on $E^\prime$ or $E^{\prime\prime}$.  Since $W$ is taut, $D^\mathrm{ext}$ is an H-face.  
	
	We can then remove this H-face from $\biang$ (recall Figure \ref{fig:add-remove-H}), obtaining a local web $W_1$ that is essential and has strictly fewer tri-valent vertices than $W$.  Repeating this process, we obtain a sequence of essential local webs $W=W_0, W_1, \dots, W_n$ such that $W_n$ has no tri-valent vertices and is obtained from $W$ by removing finitely many H-faces.  By non-ellipticity, $W_n$ consists of a collection of arcs $\gamma_i^{(n)}$ (as opposed to loops), and since $W_n$ is taut, each arc $\gamma_i^{(n)}$ connects to both boundary edges $E^\prime$ and $E^{\prime\prime}$ of the biangle~$\biang$.  

	Replacing the removed H-faces with local crossings (Figure \ref{fig:birds-eye-view}), we obtain a multi-curve $\Gamma$ in $\biang$ consisting of arcs $\gamma_i^{(0)}$, each intersecting both edges $E^\prime$ and $E^{\prime\prime}$, such that only oppositely-oriented arcs $\gamma_i^{(0)}$ intersect; see Figure \ref{fig:decomposing-essential-web-bigon}.  In particular, the pair $(S^\prime, S^{\prime\prime})$ is symmetric.  
	
	We claim $\Gamma$ is the local picture $\left< W(S) \right>$.  Since only oppositely-oriented arcs intersect, $\Gamma$ is order preserving (Definition \ref{def:local-picture-extra-def}).  It remains to show $\Gamma$ is minimally intersecting, namely that no arcs intersect more than once.  Suppose they did.  Then, because only oppositely-oriented arcs intersect, there would be an embedded bigon $B$ in the complement $\Gamma^c \subset \biang$; see the right side of Figure \ref{fig:interpreting-non-ellipticity}. Such an embedded bigon $B$ corresponds in the local web $W$ to a square-face, violating the non-ellipticity of $W$.  We gather $\Gamma = \left< W(S) \right>$, as claimed.  
	
	By definition of the multi-curve $\Gamma$ and the local web $W(S)$, it follows  that $W = W(S)$.  
\end{proof}

\begin{figure}[t]
     \centering
         \includegraphics[scale=1.17]{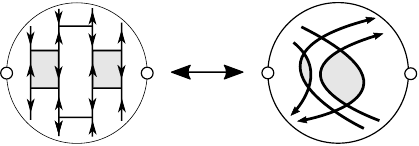}
         \caption{Prohibited ladder-webs and local pictures }
        \label{fig:interpreting-non-ellipticity}
         \label{subfig:internal-squares}
\end{figure}

For technical reasons, in \S \ref{sec:proof-of-main-lemma} we will need the following concept.  

\begin{definition}
\label{def:biangle-local-picture}
	The \textit{local picture} $\left< W_\biang \right>$ associated to an essential local web $W_\biang$ in the biangle $\biang$ is the local picture $\left< W(S) \right>$ (Definition \ref{def:local-picture-extra-def}) corresponding to the unique symmetric strand-set pair $S= (S^\prime, S^{\prime\prime})$ such that $W_\biang = W(S)$; see Figure \ref{fig:decomposing-essential-web-bigon}.  
\end{definition}

		\subsection{Honeycomb-webs in ideal triangles}
		\label{ssec:honeycomb-webs-in-ideal-triangles}

Another name for an ideal 3-polygon $\poly_3$ is an \textit{ideal triangle} $\triang$.  We want to characterize rung-less essential local webs $W$ in triangles  $\triang$.  

\begin{definition}
\label{def:honeycomb-web}
	For a positive integer $n > 0$, the \textit{$n$-out-honeycomb-web} $H^\mathrm{out}_n$ (resp. \textit{$n$-in-honeycomb-web} $H_n^\mathrm{in}$)  in the triangle $\triang$ is the local web $H_n$ dual to the \textit{$n$-triangulation} of $\triang$, where the orientation of $H_n$ is such that all the arrows go out of (resp. into) the triangle $\triang$.  
\end{definition}

For example, in Figure \ref{fig:honeycomb-web} we show the 5-out-honeycomb-web $H_5^\mathrm{out}$.  

\begin{figure}[b]
	\centering
	\includegraphics[scale=.87]{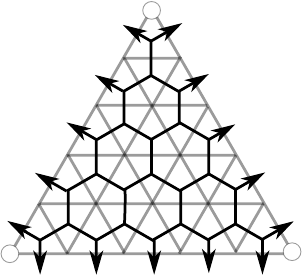}
	\caption{Honeycomb-web}
	\label{fig:honeycomb-web}
\end{figure}

The following statement is implicit in \cite[Lemma 6.8]{KuperbergCommMathPhys96} and also appears in \cite[\S 9]{MR4359515frohmansikora}.  

\begin{proposition}
\label{prop:honeycomb-webs}
	A honeycomb-web $H_n$ in the triangle $\triang$ is rung-less and essential.  Conversely, given a connected rung-less essential local web $W$ in $\triang$ having at least one tri-valent vertex, there exists a unique honeycomb-web $H_n = H_n^\mathrm{out}$ or $=H_n^\mathrm{in}$ such that $W = H_n$.  Consequently, a (possibly disconnected) rung-less essential local web $W$ in $\triang$ consists of a unique (possibly empty) honeycomb $H_n$ together with a collection of disjoint oriented arcs located on the corners of $\triang$; see the left hand side of Figure {\upshape\ref{fig:disconnected-rung-less-essential-web}}.  
\end{proposition}

\begin{figure}[htb]
	\centering
	\includegraphics[width=.60\textwidth]{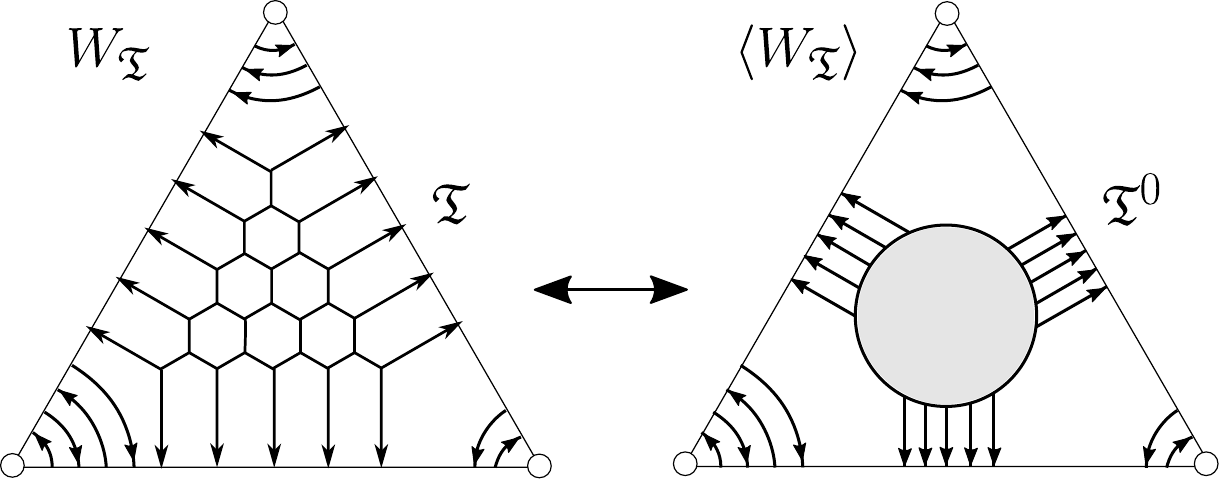}
	\caption{Rung-less essential local web $W_\triang$ in the triangle, and its corresponding local picture $\left< W_\triang \right>$ in the holed triangle}
	\label{fig:disconnected-rung-less-essential-web}
\end{figure}

\begin{proof}
	The first statement is immediate.  
	
	\textit{Step 1.}  Let $W$ be as in the second statement. Just like the proof of Proposition \ref{prop:ladder-webs}, applying Proposition \ref{fact:strengthened-boundary-non-elliptic-web-regularity} to the induced web $\overline{W}$ in the underlying closed disk $\poly_0$ guarantees that $\overline{W}$ has at least three fork- and/or H-faces, at most three of which can straddle the three punctures of $\triang$.  Since $W$ is taut and rung-less, $W$ has no fork- or H-faces.  Thus,  $\overline{W}$ has exactly three fork- and/or H-faces, each of which straddles a puncture.  Since these three faces are the only ones with a positive contribution in the formula of Proposition \ref{fact:GB}, they must be fork-faces.  Moreover, since the total contribution of these three fork-faces is $2 \pi$, every other face has exactly zero contribution.  We gather that each interior face of $W$ is a hexagon-face and each external face of $W$ is a half-hexagon-face.  

\textit{Step 2.}  To prove that $W$ is a honeycomb-web $H_n$, we argue by induction on $n$, showing that the triangle $\triang$ can be tiled by $W$ face-by-face, starting from a corner of $\triang$.  

\textit{(2.a)}  Assume inductively that some number of half-hexagon-faces have been laid down as part of the bottom layer of faces sitting on the bottom edge $E$, illustrated in Figure \ref{fig:honeycomb-proof-2}.  

\begin{figure}[b]
	\centering
	\includegraphics[scale=.65]{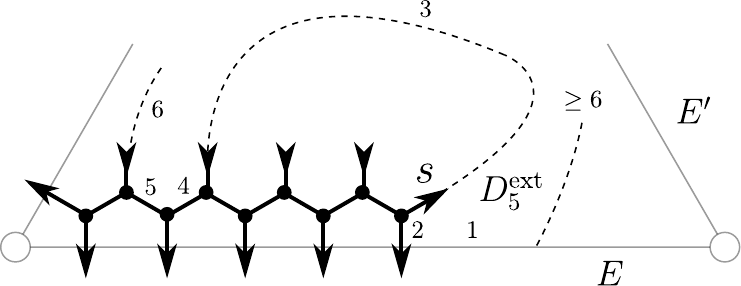}
	\caption{Laying down a honeycomb: 1 of 2}
	\label{fig:honeycomb-proof-2}
\end{figure}

The strand labeled $s$ either:  (1)  ends on the right edge $E^\prime$ of the triangle $\triang$, thereby creating a fork straddling the right-most puncture and completing the bottom layer of faces;  (2)  ends at a vertex disjoint from the vertices previously laid, hence the strand $s$ is part of the boundary of the next half-hexagon-face;  (3)  ends at one of the vertices previously laid.  

If (1), we continue to the next step of the induction, which deals with laying down the middle layers.  If (2), we repeat the current step.  Lastly, we argue (3) cannot occur.  Indeed, suppose it did.  The strand $s$ is part of the boundary of the next half-hexagon-face $D_5^\mathrm{ext}$.  But, as can be seen from the figure, the external face $D_5^\mathrm{ext}$ has $\geq 6$ sides, which is a contradiction.  

\textit{(2.b)}  Assume inductively that the bottom layer and some number of middle layers have been laid down, and moreover that some number of faces have been laid down as part of the current layer, illustrated in Figure \ref{fig:honeycomb-proof-3}.  Consider the next face $D$ shown in the figure.  

\begin{figure}[htb]
	\centering
	\includegraphics[scale=.74]{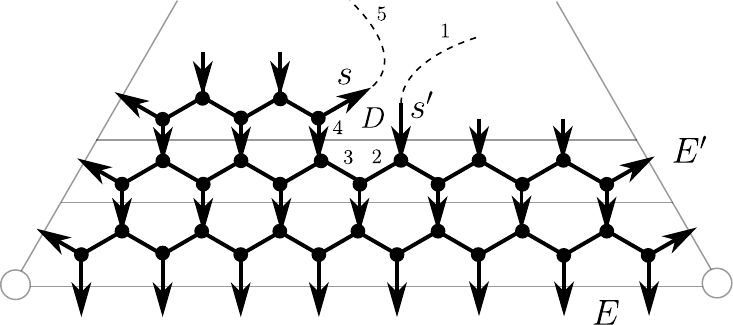}
	\caption{Laying down a honeycomb: 2 of 2}
	\label{fig:honeycomb-proof-3}
\end{figure}

The face $D$ is either external or internal.  If it is internal, then $D$ is a hexagon-face.  In this case, the strands $s$ and $s^\prime$ end at the fifth and sixth vertices of the hexagon-face, and we repeat the current step.  Otherwise, $D$ is external, so it is a half-hexagon-face, $D = D_5^\mathrm{ext}$.  However, we see from the figure that in this case $D_5^\mathrm{ext}$ has $\geq 6$ sides, which is a contradiction.  

To finish the induction, we repeat this step until the strand $s^\prime$ does not exist, in which case the strand $s$ is part of a non-external side of a half-hexagon-face lying on the boundary edge $E^\prime$.  

\textit{Step 3.}  The last statement of the proposition follows since each honeycomb-web $H_n$ attaches to all three boundary edges of the triangle $\triang$.  
\end{proof}

Later, in order to assign coordinates to webs, we will need to consider rung-less essential local webs $W_\triang$ in a triangle $\triang$ up to a certain equivalence relation.  Say that a \textit{local parallel-move} applied to $W_\triang$ is a move swapping two arcs on the same corner of $\triang$; see Figure \ref{fig:local-parallel-move-no-coordinates}.  

\begin{definition}
\label{def:corner-ambiguity}
	Let $\webbasis{\triang}$ denote the collection of rung-less essential local webs in the triangle $\triang$.  We say that two local webs $W_\triang $ and $W^\prime_\triang $ in $\webbasis{\triang}$ are \textit{equivalent up to corner-ambiguity} if they are related by local parallel-moves.  The corner-ambiguity equivalence class of a local web $W_\triang \in \webbasis{\triang}$ is denoted by $[ W_\triang ]$, and the set of corner-ambiguity classes is denoted $[ \webbasis{\triang} ]$.  
\end{definition}

For technical reasons, in \S \ref{sec:proof-of-main-lemma} we will need the following concept.  

\begin{definition}
\label{def:triangle-local-picture}
	Given a triangle $\triang$, a \textit{holed triangle} $\triang^0$ is the triangle minus an open disk $\triang^0 = \triang - \mathrm{Int}(\poly_0)$; see the right hand side of Figure \ref{fig:disconnected-rung-less-essential-web} above.    Let $W_\triang$ be a rung-less essential local web in  $\triang$, which by Proposition \ref{prop:honeycomb-webs} consists of a honeycomb-web $H_n$ together with a collection of disjoint oriented corner arcs $\{ \gamma_i \}$.  The \textit{local picture} $\left< W_\triang \right>$ associated to  $W_\triang$ is the multi-curve (Definition \ref{def:multi-curve}) in the holed triangle $\triang^0$ consisting of the corner arcs $\gamma_i$ together with $3n$ oriented arcs $\{ \gamma^\prime_j \}$ disjoint from each other and from the $\gamma_i$, and going either all out of or all into the boundary $\partial \poly_0$ of the removed disk, such that for each boundary edge $E$ of the triangle $\triang$ there are $n$ arcs $\gamma^\prime_j$ ending on $E$; see again Figure \ref{fig:disconnected-rung-less-essential-web}.  
\end{definition}

\begin{figure}[htb]
	\centering
	\includegraphics[scale=.58]{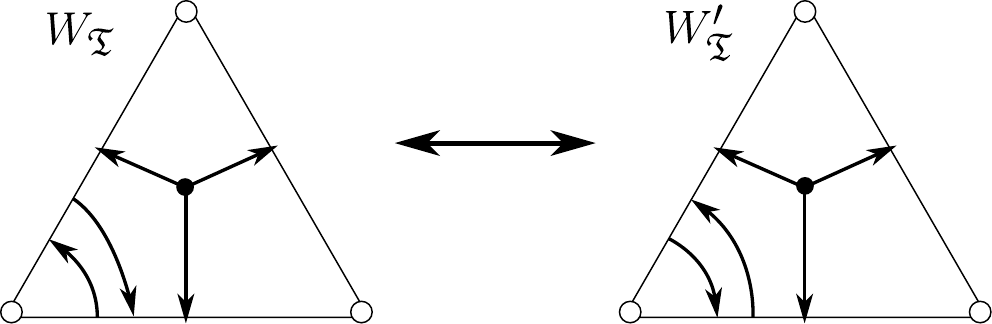}
	\caption{Local parallel-move}
	\label{fig:local-parallel-move-no-coordinates}
\end{figure}

		\section{Good position of a global web}
		\label{sec:good-position-of-a-global-web}

Using the technical results about local webs from \S \ref{sec:local-webs}, we continue studying global webs $W$ on the surface $\surf$.  We assume $\surf$ is equipped with an ideal triangulation $\idealtriang$; see \S \ref{ssec:topological-setting}.

		\subsection{Generic isotopies}
		\label{ssec:ideal-triangulations}

\begin{definition}
\label{def:generic}
	A web $W$ on $\surf$ is \textit{generic} with respect to the ideal triangulation $\idealtriang$ if none of its vertices intersect the edges $E$ of $\idealtriang$, and if in addition $W$ intersects $\idealtriang$ transversally. 
	
Two generic webs $W$ and $W^\prime$ are \textit{generically isotopic} if they are isotopic through generic webs; see Definition \ref{def:parallel-equivalent-webs}.  
\end{definition}

	Whenever there is an ideal triangulation $\idealtriang$ present, we always assume that ``web'' means ``generic web''.  However, we distinguish between isotopies and generic isotopies.

		\subsection{Minimal position}
		\label{ssec:minimal-position}

Recall the notion of two parallel-equivalent webs; see Definition \ref{def:parallel-equivalent-webs}.  

\begin{definition}
\label{def:minimal position}
	Given a web $W$ on the surface $\surf$ and given an edge $E$ of the ideal triangulation $\idealtriang$, the \textit{local geometric intersection number of the web $W$ with the edge $E$} is
\begin{equation*}
	I(W, E) = \min_{W^\prime}(\iota(W^\prime, E))
	\quad
	\in  \Zpos
	\quad\quad
	\left(
		W^\prime \text{ is parallel-equivalent to } W
	\right),
\end{equation*}
where $\iota(W^\prime, E)$ is the number of intersection points of $W^\prime$ with $E$.  
	
The web $W$ is in \textit{minimal position with respect to the ideal triangulation $\idealtriang$} if
\begin{equation*}
	\iota(W, E) = I(W, E)
	\quad  \in  \Zpos
	\quad\quad
	\left(
	\text{for all edges } E \text{ of } \idealtriang
	\right).
\end{equation*}
(If this is the case, $W$ minimizes the intersection number $\iota(W, \idealtriang)$ with the ideal triangulation~$\idealtriang$.)
\end{definition}

Let $W^\prime$ be a web, let $\triang$ be a triangle in the ideal triangulation $\idealtriang$, and let $W^\prime_\triang = W^\prime \cap \triang$ be the \textit{restriction} of $W^\prime$ to $\triang$.  Suppose that the local web $W^\prime_\triang$ is not taut; see Definition \ref{def:essential-web-with-boundary}.  Then there is an edge $E$ of $\idealtriang$ and a compact arc $\alpha$ ending on $E$ such that $\iota(W^\prime, E) > \iota(W^\prime, \alpha)$; see Figure \ref{fig:0-vertex-move-and-1-vertex-move}.  We can then isotope the part of $W^\prime$ that is inside the \textit{bigon} $B$, which is bounded by $\alpha$ and the segment $\overline{E}$ of $E$ delimited by $\partial \alpha$, into the adjacent triangle, resulting in a new web $W$.  This is called a \textit{tightening-move}.  Similarly, if the restriction $W^\prime_\triang$ has an H-face, then we may apply an \textit{H-move} to push the H into the adjacent triangle; see again Figure \ref{fig:0-vertex-move-and-1-vertex-move}.  

Note that tightening- and H-moves can be achieved with an isotopy of the web, but not a generic isotopy.  Also, by definition, in order to apply an H-move, we assume that the shaded region shown at the bottom of Figure \ref{fig:0-vertex-move-and-1-vertex-move} is \textit{empty}, namely it does not intersect the web.  

\begin{figure}[b]
 	\centering
	\includegraphics[width=.65\textwidth]{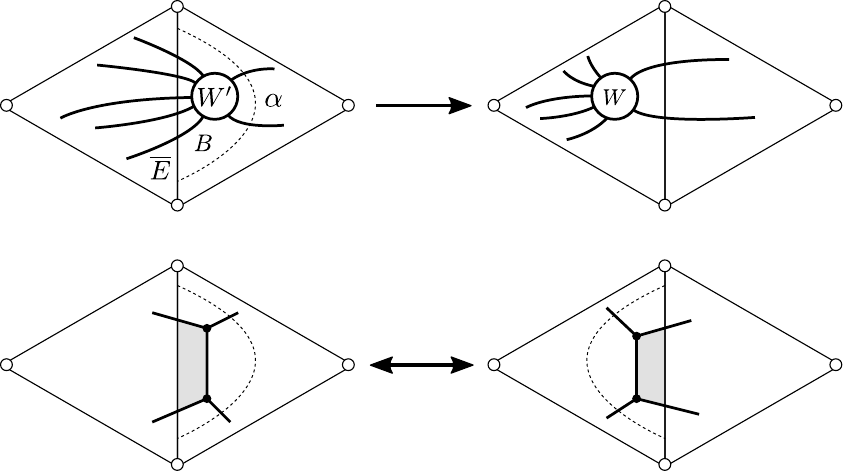}
	\caption{Tightening- and H-moves}
	\label{fig:0-vertex-move-and-1-vertex-move}
 \end{figure}
 
 We borrow the following result from \cite[\S 6]{MR4359515frohmansikora} and give essentially the same proof.

\begin{proposition}
\label{prop:minimal-position}
	If $W^\prime$ is a non-elliptic web on the surface $\surf$, then (by applying tightening-moves) there exists a non-elliptic web $W$ that is isotopic (in particular, parallel-equivalent) to $W^\prime$ and that is in minimal position with respect to the ideal triangulation $\idealtriang$; see Definition~{\upshape\ref{def:essential-global-webs}}.  

	Moreover, given any two parallel-equivalent non-elliptic webs $W$ and $W^\prime$ in minimal position, then $W$ can be taken to $W^\prime$ by a sequence of H-moves, global parallel-moves, and generic isotopies; see Definition {\upshape\ref{def:parallel-equivalent-webs}}.  
\end{proposition}

\begin{proof}
	We give an algorithm putting the web $W^\prime$ into minimal position $W$.  If a tightening-move can be applied, do so.  Else, stop.  Since tightening-moves strictly decrease the quantity
\begin{equation*}
	\sum_{E \text{ edge of } \idealtriang} \iota(W^\prime, E)
	\quad  \in  \Zpos,
\end{equation*} 
the algorithm stops.  We claim that the resulting non-elliptic web $W$ is in minimal position.  

Let $W$ be this resulting web.  Let $E$ be an edge of $\idealtriang$.  By definition of the local geometric intersection number $I(W, E)$ there exists a non-elliptic web, which by abuse of notation we also call $W^\prime$, parallel-equivalent to $W$ such that $\iota(W^\prime, E) = I(W, E)$.  

By applying global parallel-moves to $W^\prime$, we may assume that $W$ and $W^\prime$ are isotopic, by an ambient isotopy $\varphi_t$ of the surface $\surf$ such that $\varphi_0$ is the identity and $\varphi_1(W) = W^\prime$.  We may also assume that the isotopy is fixed near the punctures and satisfies the property that $E$ and $\varphi_1^{-1}(E)$ intersect finitely many times.  A classical theorem in topology (see, for instance, \cite{EpsteinActa66}) guarantees the existence of an embedded bigon $B$ bounded by a segment $\overline{E}$ of $E$ and a segment $\alpha$ of $\varphi_1^{-1}(E)$; see Figure \ref{fig:epstein-bigon-theorem}.  Note that $B \cap W$ may be non-empty.  

\begin{figure}[t]
	\centering
	\includegraphics[width=.81\textwidth]{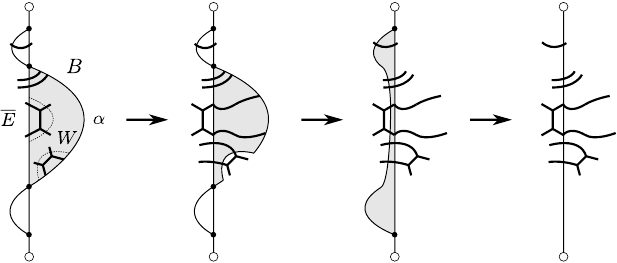}
	\caption{Relating $W$ and $W^\prime$ by H-moves and generic isotopies}
	\label{fig:epstein-bigon-theorem}
\end{figure}

By removing the two intersection points $\overline{E} \cap \alpha$ from the bigon $B$, we obtain a biangle $\biang$.  The minimality properties of $W$ and $W^\prime$ imply that the local web restriction $W_\biang$ is taut (this requires a small argument if the bigon $B$ cuts through many triangles $\triang$ of $\idealtriang$).  It is also non-elliptic since, more or less by hypothesis, $W$ is non-elliptic.  So $W_\biang$ is essential; Definition \ref{def:essential-web-with-boundary}.   By Proposition \ref{prop:ladder-webs}, $W_\biang$ is a ladder-web.  

Thus, by performing a finite number of H-moves to (retroactively) adjust $W$, $W^\prime$, and $\varphi_t$ we may assume that $W_\biang$ consists of a finite number of arcs stretching from $\overline{E}$ to $\alpha$; see Figure \ref{fig:epstein-bigon-theorem}.  With further adjustments by generic isotopies, the bigon $B$ can be removed completely.  Note that the number of intersection points of $W$ and $W^\prime$, respectively, with the ideal triangulation $\idealtriang$ is preserved throughout this adjustment process.  

By repeating the above step finitely many times in order to remove all of the bigons $B$, we may arrange that the symmetry $\varphi_1$ taking $W$ to $W^\prime$ restricts to the identity mapping on the edge $E$ (in fact, on a neighborhood of $E$).  Hence,  $\iota(W, E) = \iota(W^\prime, E) = I(W, E)$.  Since the edge $E$ was arbitrary, we are done.  

The second statement of the proposition is achieved by applying the above argument to each edge $E_i$ of $\idealtriang$, one at a time.  The key point is that if the symmetry $\varphi_1$ fixes pointwise the edges $E_1, E_2, \dots, E_{k-1}$, then a bigon $B$ formed between $E_k$ and $\varphi^{-1}_1(E_k)$ does not intersect $E_1 \cup E_2 \cup \cdots \cup E_{k-1}$.  We gather that we may assume the symmetry $\varphi_1$ sending $W$ to $W^\prime$ restricts to the identity mapping on a neighborhood of $\idealtriang$, and thus also maps each triangle $\triang$ of $\idealtriang$ to itself.  To finish, $W$ can be brought to $W^\prime$ through a generic isotopy fixing pointwise the ideal triangulation $\idealtriang$ (by Smale's theorem, for instance).  

As a remark, note that throughout this proof we never had to consider the arbitrary behavior of the ambient isotopy $\varphi_t$ between $0 < t < 1$.  
\end{proof}

		\subsection{Split ideal triangulations}
		\label{ssec:split-ideal-triangulations}

A \textit{split ideal triangulation} $\splitidealtriang$ with respect to the ideal triangulation $\idealtriang$ is a collection of bi-infinite arcs obtained by doubling every edge $E$ of $\idealtriang$.  In other words, we fatten each edge $E$ into a biangle $\biang$; see Figure \ref{fig:split-ideal-triangulation}.

\begin{figure}[t]
	\centering
	\includegraphics[scale=.55]{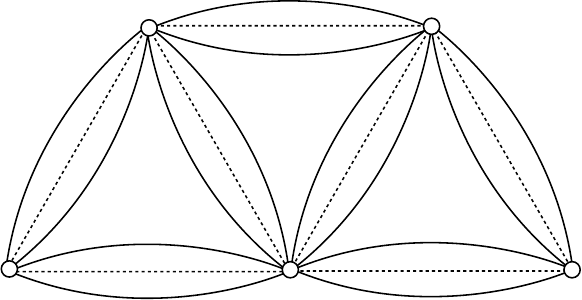}
	\caption{Split ideal triangulation}
	\label{fig:split-ideal-triangulation}
\end{figure}

The notions of generic web and generic isotopy for webs with respect to the split ideal triangulation $\splitidealtriang$ are the same as those for webs with respect to the ideal triangulation $\idealtriang$.  We always assume that webs are generic with respect to $\splitidealtriang$.  

To avoid cumbersome notation, we identify the triangles $\triang$ of the ideal triangulation $\idealtriang$ to the triangles $\triang$ of the split ideal triangulation $\splitidealtriang$.  

\begin{remark}\label{rem:bonahonwong}
	For a related usage of split ideal triangulations, in the $\mathrm{SL}_2$-case, see \cite{BonahonGT11}.
\end{remark}

		\subsection{Good position}
		\label{ssec:good-position}

\begin{definition}
\label{def:good-position}
	For a fixed split ideal triangulation $\splitidealtriang$, a web $W$ on $\surf$ is in \textit{good position} with respect to $\splitidealtriang$ if the restriction $W_\biang = W \cap \biang$ (resp. $W_\triang = W \cap \triang$) of $W$ to each biangle $\biang$ (resp. triangle $\triang$) of $\splitidealtriang$ is an essential (resp. rung-less essential) local web; see Figure \ref{fig:good-position-example-zero}.  
\end{definition}

\begin{figure}[b]
	\centering
	\includegraphics[scale=.71]{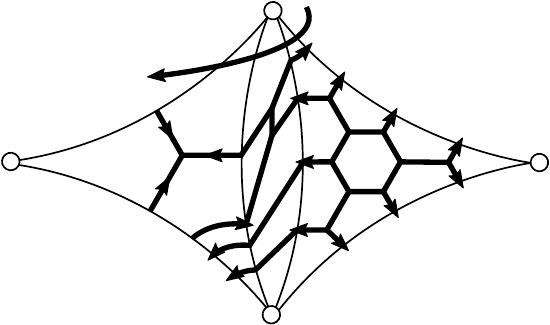}
	\caption{(Part of) a web in good position}
	\label{fig:good-position-example-zero}
\end{figure}

Note that for a web $W$ in good position, each restriction $W_\biang$ to a biangle $\biang$ of $\splitidealtriang$ is a ladder-web; see Definition \ref{def:actual-def-of-ladder-web}, Proposition \ref{prop:ladder-webs}, and Figures \ref{fig:ladder-webs} and \ref{fig:decomposing-essential-web-bigon}.  Also, each restriction $W_\triang$ to a triangle $\triang$ of $\splitidealtriang$ is a (possibly empty) honeycomb-web $H_n$ together with a collection of disjoint oriented corner arcs; see Definition \ref{def:honeycomb-web}, Proposition \ref{prop:honeycomb-webs}, and  Figures \ref{fig:honeycomb-web} and \ref{fig:disconnected-rung-less-essential-web}.  

If $W$ is a web in good position, then a \textit{modified H-move} carries an H-face in a biangle $\biang$ to an H-face in an adjacent biangle $\biang^\prime$, thereby replacing $W$ with a new web $W^\prime$; see Figure \ref{fig:modified-H-move-no-coordinates}.  If, in addition, $W$ is non-elliptic, then $W^\prime$ is also in good position.  The non-elliptic condition for $W$ is required to ensure that the new local web restriction $W^\prime_{\biang^\prime}$ is non-elliptic.  

\begin{figure}[t]
	\centering
	\includegraphics[scale=.59]{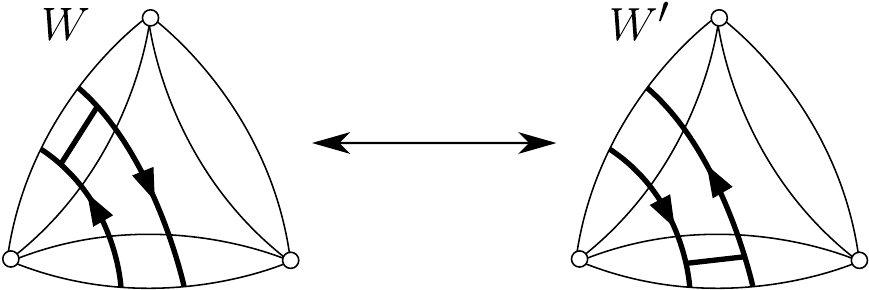}
	\caption{Modified H-move}
	\label{fig:modified-H-move-no-coordinates}
\end{figure}

\begin{remark}
	Of importance will be that the effect in the intermediate triangle $\triang$ of a modified H-move is to swap two parallel oppositely-oriented corner arcs; see again Figure \ref{fig:modified-H-move-no-coordinates}.  
\end{remark}

Once more, the following result is implicit in \cite[Lemma 6.5 and the proof of Theorem 6.2, pp. 139-140]{KuperbergCommMathPhys96} (in the setting of an ideal $k$-polygon $\poly_k$) and also appears in \cite[\S 10]{MR4359515frohmansikora}.  

\begin{proposition}
\label{prop:good-position}
	If $W^\prime$ is a non-elliptic web on the surface $\surf$, then there exists a non-elliptic web $W$ that is isotopic (in particular, parallel-equivalent) to $W^\prime$ and that is in good position with respect to the split ideal triangulation $\splitidealtriang$.  
	
	Moreover, given any two parallel-equivalent non-elliptic webs $W$ and $W^\prime$ in good position, then $W$ can be taken to $W^\prime$ by a sequence of modified H-moves, global parallel-moves, and generic isotopies. 
\end{proposition}

\begin{proof}

We will keep track of isotopies by moving the split triangulation $\splitidealtriang$ instead of webs.  

By Proposition \ref{prop:minimal-position}, we can replace $W^\prime$ with a non-elliptic web $W$ that is isotopic to $W^\prime$ and that is in minimal position with respect to the ideal triangulation $\idealtriang$.  We proceed to construct the split ideal triangulation $\splitidealtriang$.

Let us begin by splitting each edge $E$ of $\idealtriang$ into two edges $E^\prime$ and $E^{\prime\prime}$ that are very close to $E$.  These split edges form a preliminary split ideal triangulation $\splitidealtriang$, whose triangles (resp. biangles) are denoted by $\widehat{\triang}$ (resp. $\biang_E$); see the left hand side of Figure \ref{fig:billowing-biangle}.  

\begin{figure}[b]
	\centering
	\includegraphics[width=.84\textwidth]{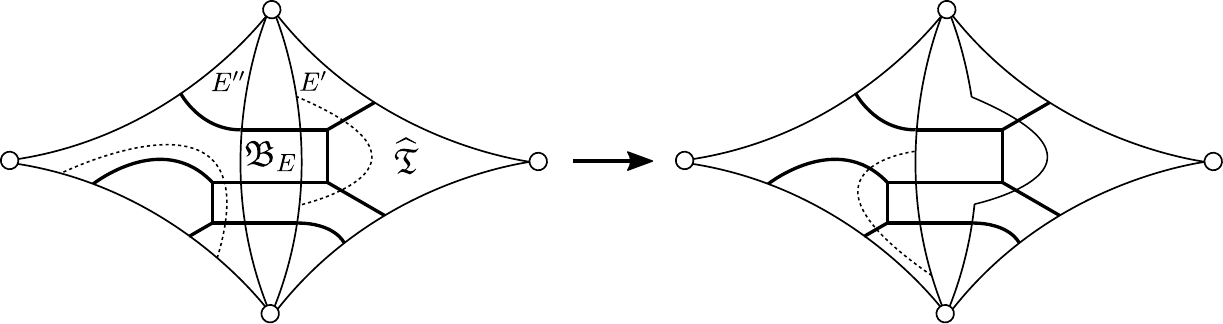}
	\caption{Enlarging a biangle}
	\label{fig:billowing-biangle}
\end{figure}

By definition of minimal position, the restriction $W_\triang$ of $W$ to a triangle $\triang$ of the ideal triangulation $\idealtriang$ is taut.  Since, in addition, $W$ is non-elliptic, we have that $W_\triang$ is essential.  If the preliminary split ideal triangulation $\splitidealtriang$ is sufficiently close to $\idealtriang$, then the restriction $W_{\widehat{\triang}}$ of $W$ to the triangle $\widehat{\triang} \subset \triang$ associated to $\triang$ is also an essential local web.  If all of the local webs $W_{\widehat{\triang}}$ are rung-less, then $W$ is in good position with respect to $\splitidealtriang$.  

Otherwise, assume $W_{\widehat{\triang}}$ has an H-face on an edge of $\splitidealtriang$, say the edge $E^\prime$.  Then by isotopy we can enlarge the biangle $\biang_E$ until it just envelops this H-face.  In other words, we can isotope the edge $E^\prime$ so that it cuts out this H-face from the triangle $\widehat{\triang}$; see Figure \ref{fig:billowing-biangle}.  The result of this step is a new split ideal triangulation $\splitidealtriang$, retaining the property that the local web restrictions $W_{\widehat{\triang}}$ are essential.  Repeating this process until all of the local webs $W_{\widehat{\triang}}$ are rung-less, we obtain the desired split ideal triangulation $\splitidealtriang$.  Notice it might be the case that there is more than one biangle into which an H-face can be moved; see again Figure \ref{fig:billowing-biangle}.  

For the second statement of the proposition, note that if a non-elliptic web $W$ is in good position with respect to $\splitidealtriang$, then $W$ is minimal with respect to the ideal triangulation $\idealtriang$ (which, for the sake of argument, we can take to be contained in $\splitidealtriang$, that is $\idealtriang \subset \splitidealtriang$).  (Indeed, this follows by the proof of the first part of Proposition \ref{prop:minimal-position},  and uses the fact that adding a ladder web $W_\biang$ to a rung-less essential web $W_\triang$ preserves the tautness property.)  Similarly, $W^\prime$ is in minimal position.  Thus, applying the second part of Proposition \ref{prop:minimal-position}, we gather that $W$ can be taken to $W^\prime$ by a finite sequence of H-moves, global parallel-moves, and generic isotopies.  The result follows by the definition of good position and modified H-moves.  
\end{proof}

		\section{Global coordinates for non-elliptic webs}
		\label{sec:global-coordinates-for-non-elliptic-webs}

Recall that $[\webbasis{\surf}]$ denotes the collection of parallel-equivalence classes of non-elliptic webs on the surface $\surf$; see just below Definition \ref{def:essential-global-webs}.  Our goal in this section is to define a function $\themap{\idealtriang}^\mathrm{FG}: [\webbasis{\surf}] \to \Zpos^N$ depending on the ideal triangulation $\idealtriang$, where $N = -8\chi(\surf) > 0$ is a positive integer depending only on the topology of $\surf$.  In \S \ref{sec:knutson-tao-positive-integer-cone}-\ref{sec:proof-of-main-lemma}, we characterize the image of $\themap{\idealtriang}^\mathrm{FG}$ and prove that it is injective.  We think of $\themap{\idealtriang}^\mathrm{FG}$ as putting global coordinates on $[\webbasis{\surf}]$.

		\subsection{Dotted ideal triangulations}
		\label{ssec:dotted-ideal-triangulations}

Consider a surface $\surfbord = \surf$ or $= \triang$ equipped with an ideal triangulation $\idealtriang$, where, in this sub-section, $\idealtriang = \triang$ when $\surfbord = \triang$.  The associated \textit{dotted ideal triangulation} is the pair consisting of $\idealtriang$ together with $N^\prime=N$ or $=7$ distinct \textit{dots} attached to the 1- and 2-cells of $\idealtriang$, where there are two \textit{edge-dots} attached to each 1-cell and there is one \textit{triangle-dot} attached to each 2-cell; see Figure \ref{fig:example-ideal-triangulations-dotted}.  Given a triangle $\triang$ of $\idealtriang$ and an edge $E$ of $\triang$, it makes sense to talk about the \textit{left-edge-dot} and \textit{right-edge-dot} as viewed from $\triang$; see Figure \ref{fig:dotted-triangle}.  Choosing an ordering for the $N^\prime$ dots lying on the dotted ideal triangulation $\idealtriang$ defines a one-to-one correspondence between functions $\{ \text{dots} \} \to \Z$ and elements of $\Z^{N^\prime}$.  We always assume that such an ordering has been chosen.

\begin{figure}[htb]
     \centering
     \begin{subfigure}{0.35\textwidth}
         \centering
         \includegraphics[width=\textwidth]{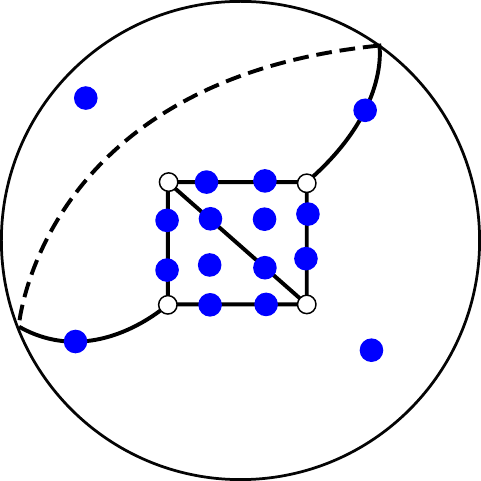}
         \caption{Four times punctured sphere}
         \label{fig:four-times-punctured-sphere-triang-dotted}
     \end{subfigure}     
\hfill
     \begin{subfigure}{0.2\textwidth}
         \centering
         \includegraphics[width=\textwidth]{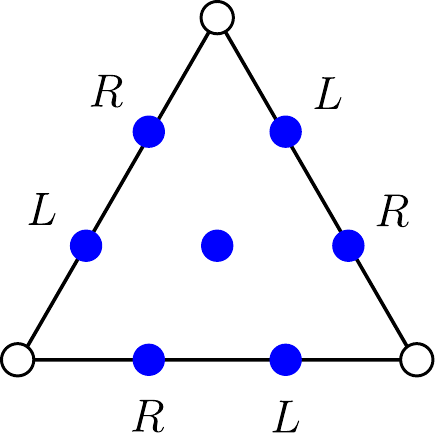}
         \caption{Ideal triangle}
         \label{fig:dotted-triangle}
     \end{subfigure}
        \caption{Dotted ideal triangulations}
        \label{fig:example-ideal-triangulations-dotted}
\end{figure}

		\subsection{Local coordinate functions }
		\label{ssec:local-web-adapted-functions}

Consider a dotted ideal triangle $\triang$; see Figure \ref{fig:dotted-triangle}.  Recall (Definition \ref{def:corner-ambiguity}) that $\webbasis{\triang}$ denotes the collection of rung-less essential local webs $W_\triang$ in $\triang$, and that $[\webbasis{\triang}]$ denotes the set of corner-ambiguity classes $[ W_\triang ]$ of local webs $W_\triang$ in $\webbasis{\triang}$.  

\begin{definition}
\label{def:local-coordinate-function}
	An \textit{integer local coordinate function}, or just \textit{local coordinate function},
\begin{equation*}
	\themap{\triang} :  \webbasis{\triang}  \longrightarrow \Z^7
\end{equation*}
is a function assigning to each local web $W_\triang$ in $\webbasis{\triang}$ one  integer coordinate per dot lying on the dotted triangle $\triang$, satisfying the following properties:
\begin{enumerate}
	\item  if a local web $W_\triang$ in $\webbasis{\triang}$ can be written $W_\triang = W^\prime_\triang \sqcup W^{\prime\prime}_\triang$ as the disjoint union of two local webs, each  in $\webbasis{\triang}$, then
	\begin{equation*}
		\themap{\triang}(W_\triang) 
		= \themap{\triang}(W^\prime_\triang) + \themap{\triang}(W^{\prime\prime}_\triang)
		\quad  \in \Z^7; 
	\end{equation*}
	\item  for an edge $E$ of $\triang$, the ordered pair of coordinates $(a^L_E, a^R_E)$ of the function $\themap{\triang}$ assigned to the left- and right-edge-dots lying on $E$, respectively, depends only on the pair $(n^\text{in}_E, n^\text{out}_E)$ of numbers of in- and out-strands of the local web $ W_\triang$ on the edge $E$; moreover, different pairs $(n_E^\mathrm{in}, n_E^\mathrm{out})$ yield different pairs of coordinates $(a^L_E, a^R_E)$;
	\item  there are two symmetries; the first is that $\themap{\triang}$ respects the rotational symmetry of the triangle (see Remark \ref{rem:frompicturestocoordinates} below for a more precise statement), and the second is that if the numbers $n^\text{in}_E$ and $n^\text{out}_E$ of in- and out-strands on an edge $E$ are exchanged, then the coordinates $a^L_E$ and $a^R_E$ are exchanged as well;
	\item  observe, by property (1), the function $\themap{\triang}(W_\triang) = \themap{\triang}(W^\prime_\triang)$ agrees on local webs $W_\triang$ and $W^\prime_\triang$ in $\webbasis{\triang}$ representing the same corner-ambiguity class $[W_\triang] = [W^\prime_\triang]$ in $[ \webbasis{\triang} ]$ (because $W_\triang$ and $W^\prime_\triang$ differ only by permutations of oriented corner arcs), thus inducing
	\begin{equation*}
		\themap{\triang} :  [\webbasis{\triang}]  \longrightarrow \Z^7,
	\end{equation*}	
also called $\themap{\triang}$; we require that this induced function $\themap{\triang}$ is an injection.
\end{enumerate}
The coordinates assigned by $\themap{\triang}$ to edge-dots (resp. triangle-dots) are called \textit{edge-coordinates} (resp. \textit{triangle-coordinates}).  
\end{definition}

We illustrate properties (1), (2), (3) in Figures \ref{fig:additive-property-of-coordinates} and \ref{fig:dotted-triangle-symmetry}.

\begin{figure}[b]
	\centering
	\includegraphics[scale=.52]{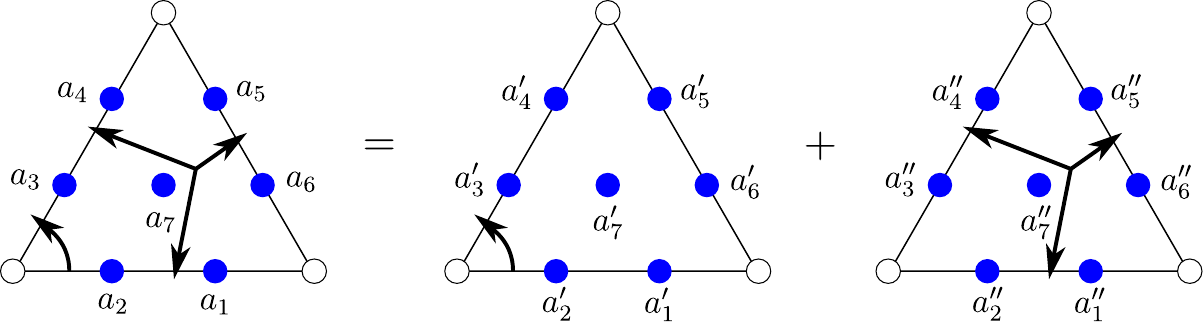}
	\caption{Property (1): $a_i = a_i^\prime + a_i^{\prime\prime}$}
	\label{fig:additive-property-of-coordinates}
\end{figure}

\begin{figure}[htb]
	\centering
	\includegraphics[width=\textwidth]{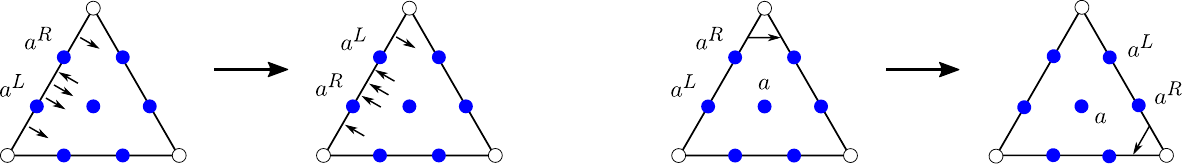}
	\caption{Properties (2) and (3)}
	\label{fig:dotted-triangle-symmetry}
\end{figure}

\begin{remark}[from pictures to coordinates]\label{rem:frompicturestocoordinates}
Let us be more precise about what we mean by the first symmetry of property (3), which will also allow us the opportunity to give a clearer explanation of the meaning of pictures such as those shown in the figures below.  We will use the picture displayed on the left hand side of Figure \ref{fig:additive-property-of-coordinates} as a reference.  When we draw such a picture, we have implicitly selected a preferred vertex of the triangle $\triang$, say the vertex appearing at the top of the picture; write $\triang_0$ to indicate this extra data.  A tuple $(a_1, a_2, \dots, a_7)\in\Z^7$  defines a function $\{\text{dots of }\triang_0\}\to\Z$ by sending the $i$-th dot to $a_i$, as indicated in the picture.  If this tuple is associated to a local web $W_{\triang_0}$, then we say this tuple is the value $\themap{\triang_0}(W_{\triang_0})$ of the local coordinate function evaluated on the web $W_{\triang_0}$.  
The rotational symmetry of property (3) says that if $W^\prime_{\triang_0}$ is the different local web obtained by rotating $W_{\triang_0}$ by $2\pi/3$ radians clockwise, with coordinates $\themap{\triang_0}(W^\prime_{\triang_0})=(a^\prime_1, a^\prime_2, \dots, a^\prime_7)$, then $a^\prime_1=a_5$, $a^\prime_2=a_6$, $a^\prime_3=a_1$, $a^\prime_4=a_2$, $a^\prime_5=a_3$, $a^\prime_6=a_4$, and $a^\prime_7=a_7$.  Lastly, we define $\themap{\triang}(W_\triang)=\themap{\triang_0}(W_{\triang_0})$, and the rotational symmetry implies that this is independent of the choice of preferred vertex.  
\end{remark}


		\subsection{Local coordinates from Fock-Goncharov theory}
		\label{ssec:local-coordinates-from-Fock-Goncharov-theory}

We define an explicit \textit{Fock-Goncharov local coordinate function} $\themap{\triang}^\mathrm{FG}: \webbasis{\triang} \to \Zpos^7$ valued in non-negative integers.  

By property (1) in Definition \ref{def:local-coordinate-function}, it suffices to define $\themap{\triang}^\mathrm{FG}$ on connected local webs in $\webbasis{\triang}$.  By Proposition \ref{prop:honeycomb-webs}, these come in one of exactly eight types $H_n^\text{out}$, $H_n^\text{in}$, $R_1$, $L_1$, $R_2$, $L_2$, $R_3$, $L_3$ illustrated in Figure \ref{fig:triangle-hilbert-basis}.  In the figure, note that in the two top left triangles we have, for visibility, drawn the local pictures $\left< H_n^\text{out} \right>$ and $\left< H_n^\text{in} \right>$ as a short-hand for the actual $n$-out-honeycomb-web $H_n^\text{out}$ and $n$-in-honeycomb-web $H_n^\text{in}$, respectively; see Definition \ref{def:triangle-local-picture}.  It is immediate that  $\themap{\triang}^\mathrm{FG}$ satisfies property (3) and the first part of (2).  The second part of (2) follows by the invertibility of the matrix $\left( \begin{smallmatrix} 2 & 1 \\ 1 & 2 \end{smallmatrix} \right)$.  We will check property (4) in \S \ref{sec:knutson-tao-positive-integer-cone}.  

\begin{remark}\label{rem:xiemotivation}		$ $
\begin{enumerate}
	\item\label{subrem:xie1}  Xie \cite{XieArxiv13} writes down the same local coordinates (up to a multiplicative factor of 3) for $R_1$, $L_1$, $R_2$, $L_2$, $R_3$, $L_3$ as well as the 1-honeycomb-webs $H_1^\text{out}$ and $H_1^\text{in}$.  
	\item\label{subrem:experiment}  The definition of these local coordinates can be checked experimentally by studying the highest terms of the Fock-Goncharov $\mathrm{SL}_3$-trace polynomials; see the introduction as well as \cite[Proposition 5.80]{KimArxiv20} (and \cite[Proposition 3.15]{kim2021mutation}).  Moreover, it appears that these coordinates fit into a broader geometric context \hbox{\cite[Theorem 8.22(2)]{SunGeomFunctAnal20}}.  
	\item  The coordinates in the $\mathrm{SL}_2$-setting are geometric intersection numbers; see the introduction.  In contrast, the $\mathrm{SL}_3$-coordinates depend on the choice of orientation of~$\surf$.  
\end{enumerate} 
\end{remark}

\begin{figure}[htb]
	\centering
	\includegraphics[width=.86\textwidth]{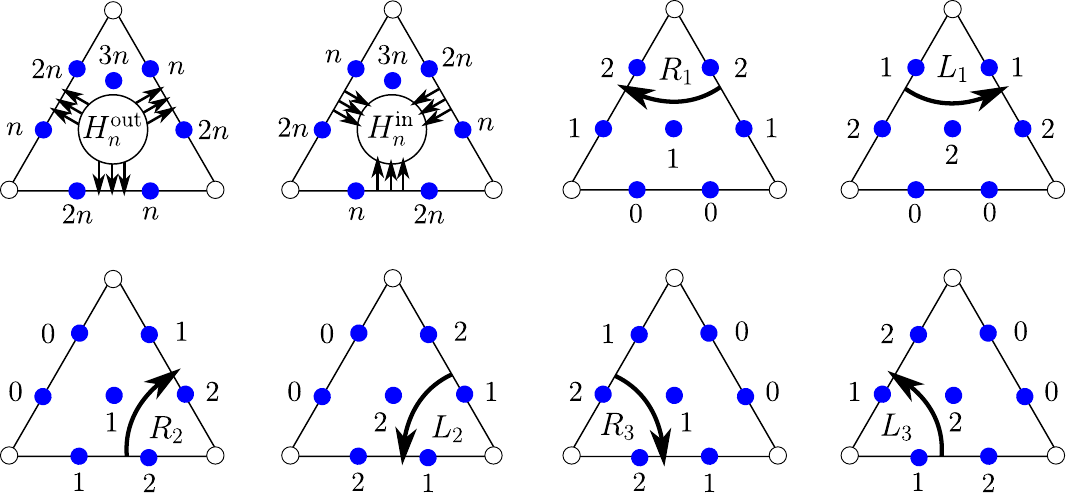}
	\caption{Fock-Goncharov local coordinate function $\themap{\triang}^\mathrm{FG}$}
	\label{fig:triangle-hilbert-basis}
\end{figure}

		\subsection{Global coordinates from local coordinate functions}
		\label{ssec:global-coordinates-from-local-coordinate-functions}

Assume that, for an abstract dotted triangle $\triang$, we have chosen an arbitrary local coordinate function $\themap{\triang}: \webbasis{\triang} \to \Z^7$.  We show that this induces a \textit{global coordinate function} $\themap{\idealtriang} : [\webbasis{\surf}] \to \Z^N$ that is well-adapted to the choice of $\themap{\triang}$.  The argument uses only properties (1), the first part of (2), and (3) of~$\themap{\triang}$.  

As a guiding example of the construction to come, reference Figure \ref{fig:coordinates-example}, which uses the Fock-Goncharov local coordinate function $\themap{\triang}^\mathrm{FG}$.  This is an example on the once punctured torus $\surf$.  Note that the web $W$ in the example has one hexagon-face.  All of the other components of $W^c$ are not contractible.  So $W$ is  non-elliptic.  

\textit{Step 1.}  Consider the split ideal triangulation $\splitidealtriang$ (\S \ref{ssec:split-ideal-triangulations}).  We put dots on each triangle $\triang$ of $\splitidealtriang$.  The chosen local coordinate function $\themap{\triang}$ can be associated to each of these dotted triangles $\triang$; see the left hand side of Figure \ref{fig:split-ideal-triangulation2}.  

\begin{figure}[t]
	\centering
	\includegraphics[width=.7\textwidth]{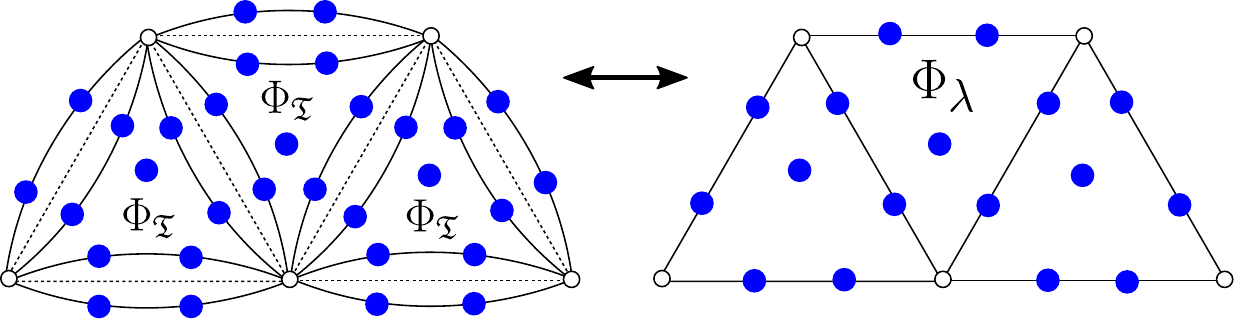}
	\caption{Local coordinates $\themap{\triang}$ attached to the triangles $\triang$ of $\splitidealtriang$ (left), and the corresponding global coordinates $\themap{\idealtriang}$ attached to $\idealtriang$ (right)}
	\label{fig:split-ideal-triangulation2}
\end{figure}

\textit{Step 2.}  Fix a non-elliptic web $W$ on $\surf$ that is in good position (Definition \ref{def:good-position}) with respect to the split ideal triangulation $\splitidealtriang$.  We assign to $W$ one  integer coordinate per dot lying on the dotted ideal triangulation $\idealtriang$, namely an element $\themap{\idealtriang}(W)$ in $\Z^N$.  

By good position, the local web restriction $W_\triang = W \cap \triang$ is in $\webbasis{\triang}$ for each triangle $\triang$ of $\splitidealtriang$.  So, we may evaluate the local coordinate function $\Phi_\triang$ on $W_\triang$, obtaining coordinates for each of the seven dots lying on the dotted triangle $\triang$ of $\splitidealtriang$.  For instance, in this way we assign coordinates to all of the dots shown on the left hand side of Figure \ref{fig:split-ideal-triangulation2} above.  We claim that these coordinates glue together along each biangle $\biang$ of $\splitidealtriang$ in such a way that we obtain one coordinate per dot lying on the dotted ideal triangulation $\idealtriang$; see Figure \ref{fig:split-ideal-triangulation2}.  

Indeed, suppose $\biang$ is a biangle between two triangles $\triang^\prime$ and $\triang^{\prime\prime}$ of $\splitidealtriang$.  Let $E^\prime$ and $E^{\prime\prime}$ be the corresponding boundary edges of $\biang$, and let $a^{ L}_{E^\prime}$ and $a^{ R}_{E^\prime}$ (resp. $a^{ L}_{E^{\prime\prime}}$ and $a^{ R}_{E^{\prime\prime}}$) be the coordinates assigned by $\Phi_{\triang^\prime}$ (resp. $\Phi_{\triang^{\prime\prime}}$) to the left- and right-edge-dots, respectively, lying on $E^\prime$ (resp. $E^{\prime\prime}$) as viewed from $\triang^\prime$ (resp. $\triang^{\prime\prime}$).  Also, denote by $n^\text{in}_{E^\prime}$ and $n^\text{out}_{E^\prime}$ (resp. $n^\text{in}_{E^{\prime\prime}}$ and $n^\text{out}_{E^{\prime\prime}}$) the numbers of in- and out-strands of the  local web restriction $W_{\triang^\prime}$ (resp. $W_{\triang^{\prime\prime}}$) lying on the edge $E^\prime$ (resp. $E^{\prime\prime}$); see Figure \ref{fig:gluing-together-coordinates}.  

\begin{figure}[htb]
	\centering
	\includegraphics[scale=.75]{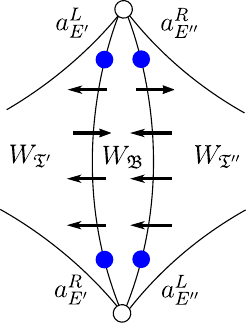}
	\caption{Local coordinates attached to a biangle: 	
	$a^L_{E^\prime} = a^R_{E^{\prime\prime}}$ and
	$a^R_{E^\prime} = a^L_{E^{\prime\prime}}$
	}
	\label{fig:gluing-together-coordinates}
\end{figure}

Since, by good position, the restriction $W_\biang = W \cap \biang$ is a ladder-web, we have $n^\text{out}_{E^\prime} = n^\text{in}_{E^{\prime\prime}}$ and $ n^\text{in}_{E^{\prime}} = n^\text{out}_{E^{\prime\prime}}$.  It follows immediately from properties (3) and the first part of (2) that the coordinates across from each other agree $a^L_{E^\prime} = a^R_{E^{\prime\prime}}$ and $a^R_{E^\prime} = a^L_{E^{\prime\prime}}$.  So, we may glue together the two pairs of coordinates into two coordinates lying on the edge $E$ of $\idealtriang$, as desired.  

\textit{Step 3.}  For a general non-elliptic web $W^\prime$ on $\surf$, by the first part of Proposition \ref{prop:good-position} there exists a non-elliptic web $W$ that is parallel-equivalent to $W^\prime$ and that is in good position with respect to the split ideal triangulation $\splitidealtriang$.  Define $\themap{\idealtriang}(W^\prime) = \themap{\idealtriang}(W)$ in $\Z^N$.  

To show $\themap{\idealtriang}(W^\prime)$ is well-defined, suppose $W_2$ is another web as $W$.  By the second part of Proposition \ref{prop:good-position}, the non-elliptic webs $W$ and $W_2$ are related by a sequence of modified H-moves and global parallel-moves.  The effect of either of these moves on a web in good position is to swap, possibly many, parallel oppositely-oriented corner arcs in the triangles $\triang$ of $\splitidealtriang$; recall Figures \ref{fig:modified-H-move-no-coordinates} and \ref{fig:parallel-move} above, respectively.  By property (1) of $\themap{\triang}$, we have $\themap{\idealtriang}(W) = \themap{\idealtriang}(W_2)$.  
	
From this point on, our approach diverges from that in \cite{MR4359515frohmansikora}.  In particular, our coordinates are different from theirs.

\begin{definition}
\label{def:fock-goncharov-global-coordinate-function}
	The \textit{Fock-Goncharov global coordinate function}
\begin{equation*}
	\themap{\idealtriang}^\FG : [\webbasis{\surf}] \longrightarrow \Zpos^N
\end{equation*}
	is the well-defined global coordinate function on $[\webbasis{\surf}]$, valued in non-negative integers, induced by the Fock-Goncharov local coordinate function $\themap{\triang}^\FG$.  In  \S \ref{sec:main-theorem}-\ref{sec:proof-of-main-lemma} we prove:
\end{definition}

\begin{proposition}
\label{prop:injectivity}
		The Fock-Goncharov global coordinate function $\themap{\idealtriang}^\FG$ is an injection of sets.  
\end{proposition} 

\begin{remark}
	Proposition \ref{prop:injectivity} is  valid for any global coordinate function $\themap{\idealtriang} : [\webbasis{\surf}] \to \Z^N$ induced by a local coordinate function $\themap{\triang} : [\webbasis{\triang}] \to \Z^7$.  The proof is the same as the one we will give for $\themap{\idealtriang}^\FG$, and uses properties (4) and the second part of (2) in Definition \ref{def:local-coordinate-function}.  
\end{remark}

\begin{figure}[b]
	\centering
	\includegraphics[width=.93\textwidth]{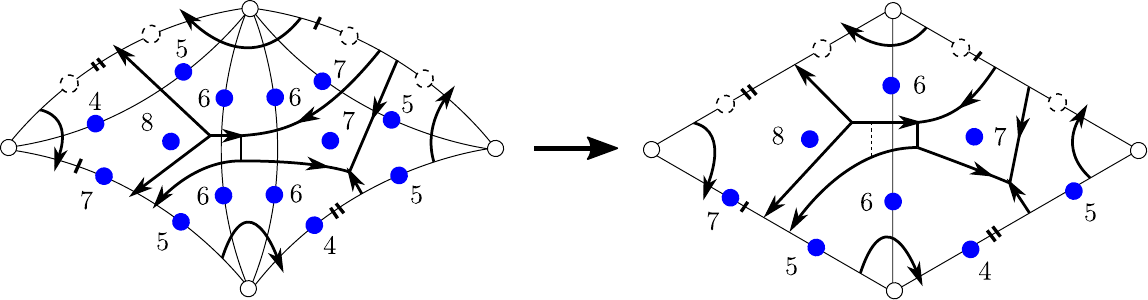}
	\caption{Tropical Fock-Goncharov $\mathcal{A}$-coordinates for a non-elliptic web}
	\label{fig:coordinates-example}
\end{figure}

\begin{remark}[relation to Fock-Goncharov theory; see \cite{DouglasArxiv20b} for a more detailed discussion]
\label{rem:relation-to-fock-goncharov theory}
To a surface-with-boundary $\surfbord$ (see \S \ref{sec:webs-on-surfaces-with-boundary} below), Fock-Goncharov/Goncharov-Shen \cite{FockIHES06, GoncharovInvent15} associated two dual moduli spaces $\mathcal{A}_{\mathrm{PGL}_3,\surfbord}$ and $\mathscr{R}_{\mathrm{SL}_n,\surfbord}$, both of which are certain generalizations of the character variety.  They are dual in the sense of Fock-Goncharov-Shen duality, which in particular says that the positive tropical integer points $\mathcal{A}_{\mathrm{PGL}_3,\surfbord}^+(\Z^t)$ of the $\mathcal{A}$-moduli space index a natural linear basis for the ring $\mathscr{O}(\mathscr{R}_{\mathrm{SL}_n,\surfbord})$ of regular functions on the generalized character variety.  Here, the positivity is taken with respect to the tropicalized Goncharov-Shen potential. 

An ideal triangulation $\idealtriang$ determines a coordinate chart $\mathcal{A}_{\mathrm{PGL}_3,\surfbord}^+(\Z^t)_\idealtriang$ of $\mathcal{A}_{\mathrm{PGL}_3,\surfbord}^+(\Z^t)$.  More concretely, in coordinates the positivity condition with respect to the tropicalized Goncharov-Shen potential translates to the Knutson-Tao rhombus inequalities (see \S \ref{sec:knutson-tao-positive-integer-cone} below), and in this way the coordinate chart $\mathcal{A}_{\mathrm{PGL}_3,\surfbord}^+(\Z^t)_\idealtriang\cong\KTcone{\idealtriang}$ becomes identified with the Knutson-Tao cone $\KTcone{\idealtriang}$.  If $\idealtriang^\prime$ is another ideal triangulation, the coordinate change map $\mathcal{A}_{\mathrm{PGL}_3,\surfbord}^+(\Z^t)_\idealtriang \to \mathcal{A}_{\mathrm{PGL}_3,\surfbord}^+(\Z^t)_{\idealtriang^\prime}$ takes the form of a tropicalized \hbox{$\mathcal{A}$-coordinate} cluster transformation.  For these reasons, Theorems \ref{thm:first-theorem-intro} and \ref{thm:second-theorem-intro} can be interpreted as saying that the web coordinates constructed above provide a natural identification between the set $[\webbasis{\surfbord}]$ of parallel-equivalence classes of rung-less essential webs (see \S \ref{sec:webs-on-surfaces-with-boundary} below) and the positive tropical integer points $\mathcal{A}_{\mathrm{PGL}_3,\surfbord}^+(\Z^t)$.  

As another concrete manifestation of Fock-Goncharov duality, when the trace function $\mathrm{Tr}_W$ on the $\mathrm{SL}_3$-character variety associated to a basis web $W$ is written as a polynomial in the Fock-Goncharov $\mathcal{X}$-coordinates, then this polynomial has a unique highest term, whose exponents are precisely the tropical $\mathcal{A}$-coordinates assigned to the web $W$; see \S \ref{ssec:application-geometry-and-topology-of-sl3c-character-varieties} below for more precise statements.  

For a discussion of previous works motivating our construction, see the introduction as well as Remarks \ref{rem:no-essential-webs-in-small-polygons}(\ref{subrem:kuperberg1}), \ref{rem:bonahonwong}, \ref{rem:xiemotivation}(\ref{subrem:xie1}, \ref{subrem:experiment}), \ref{rem:modulo-3-congruence-conditions}(\ref{subrem:GS1}), \ref{rem:kim}, \ref{rem:lastrem}(\ref{subrem:kup2}).  
\end{remark}

		\section{Knutson-Tao cone}
		\label{sec:knutson-tao-positive-integer-cone}

For $N = - 8 \chi(\surf) > 0$, we construct a subset $\KTcone{\idealtriang} \subset \Zpos^N$ that we will show, in \S \ref{sec:main-theorem}-\ref{sec:proof-of-main-lemma}, is the image $\KTcone{\idealtriang}=\themap{\idealtriang}^\mathrm{FG}([\webbasis{\surf}])$ of the mapping $\themap{\idealtriang}^\mathrm{FG} : [\webbasis{\surf}] \to \Zpos^N$ constructed in \S \ref{sec:global-coordinates-for-non-elliptic-webs}.  The subset $\KTcone{\idealtriang}$ is called the Knutson-Tao cone associated to the ideal triangulation $\idealtriang$, and is defined by finitely many Knutson-Tao rhombus inequalities and modulo 3 congruence conditions.

		\subsection{Integer cones}
		\label{ssec:integer-cones}

\begin{definition}
\label{def:cone-definitions}
	An \textit{integer cone}, or just \textit{cone}, $\Cone$ is a sub-monoid of $\Z^n$ for some positive integer $n$.  In other words, $\Cone \subset \Z^n$ is a subset that contains 0 and is closed under addition. 
	
A \textit{partition} of $\Cone$ is a decomposition $\Cone = \Cone_1 \sqcup \Cone_2 \sqcup \dots \sqcup \Cone_k$ as a disjoint union of subsets.  
	
A \textit{positive integer cone}, or just \textit{positive cone}, $\Cone^+$ is a cone that is contained in $\Zpos^n$.   
\end{definition}

We define notions of independence for cones.  
	
\begin{definition}
\label{def:cone-definitions-independence}
	Let $\Cone \subset \Z^n \subset \Q^n$ be a cone, and let $\Omega \subset \Q$ be a subset such that $0 \in \Omega$.  Let $c_1, c_2, \dots, c_k$ be a collection of cone points in $\Cone$.  We say that the cone points $\{ c_i \}$
\begin{enumerate}
	\item   \emph{span} the cone $\Cone$ if every cone point $c \in \Cone$ can be written as a $\Zpos$-linear combination of the cone points $\{ c_i \}$;
	\item  are \emph{weakly independent} over $\Omega$ if 
\begin{equation*}
	\omega_1 c_1 + \cdots + \omega_k c_k = 0 \in \Q^n
	\quad \Longrightarrow \quad
	\omega_1 = \cdots = \omega_k = 0
	\quad\quad
	\left( \omega_1, \dots, \omega_k \in \Omega \right);
\end{equation*} 
	\item  form a \emph{weak basis} of  $\Cone$ if they span $\Cone$ and are weakly independent over $\Omega=\Zpos \subset \Q$;
	\item  are \emph{strongly independent} over $\Omega$ if 
\begin{equation*}
	\omega_1 c_1 + \cdots + \omega_k c_k 
	= \omega^\prime_1 c_1 + \cdots + \omega^\prime_k c_k
	\in \Q^n
	\quad \Longrightarrow \quad
	\omega_1 = \omega^\prime_1, \dots, \omega_k = \omega^\prime_k
	\quad\quad
	\left( \omega_i, \omega^\prime_j \in \Omega \right). 
\end{equation*} 
\end{enumerate}
\end{definition}

Note:
\begin{itemize}
	\item  strongly independent over $\Omega$ $\Longrightarrow$ weakly independent over $\Omega$;
	\item  strongly independent over $\Zpos$
	$\Longleftrightarrow$ weakly independent over $\Z$
	$\Longleftrightarrow$ linearly independent over $\Q$ (the usual definition from Linear Algebra).
\end{itemize}

The following technical fact is immediate from the definitions.  

\begin{lemma}
\label{lem:KTcone-technical-lemma}
	Let $\Cone, \Cone^\prime \subset \Z^n$ be two cones.  Consider a $\Zpos$-linear bijection $\psi \colon \Cone^\prime \to \Cone$ that extends to a $\Q$-linear isomorphism $\widetilde{\psi} : \Q^n \to \Q^n$.  Let $\{ c_i \}$ be cone points of $\Cone$ and let $\{ c^\prime_i\}$ be cone points of $\Cone^\prime$, such that $\psi(c_i^\prime) = c_i$.  Then,
\begin{enumerate}
	\item  if the cone points $\{ c^\prime_i \}$ span $\Cone^\prime$, then the cone points $\{ c_i \}$ span $\Cone$;
	\item  if the $\{ c^\prime_i \}$ are weakly independent over $\Zpos$, then so are the $\{ c_i \}$;
	\item  therefore, if the $\{ c^\prime_i \}$ form a weak basis of $\Cone^\prime$, then the $\{ c_i \}$ form a weak basis of $\Cone$;
	\item  if the $\{ c^\prime_i \}$ are strongly independent over $\Zpos$, then so are the $\{ c_i \}$;
		\item  the function $\psi$ sends partitions of $\Cone^\prime$ to partitions of $\Cone$.  \qed
\end{enumerate}
\end{lemma}

		\subsection{Local Knutson-Tao cone}
		\label{ssec:local-knutson-tao-cone}

Let $\triang$ be a dotted ideal triangle (\S \ref{ssec:dotted-ideal-triangulations}); recall Figure \ref{fig:dotted-triangle} above.  In this section, we are going to order the dots on $\triang$ so that if the dots are labeled as in the left hand side of Figure \ref{fig:triangle-7-tropical-coordinates}, then a point $c \in \Z^7$ will be written
\begin{equation*}
\label{eq:integer-point}
\tag{$\ast$}
	c = \left(a_{11}, a_{12}, a_{21}, a_{22}, a_{31}, a_{32}, a\right)
	\quad  \in \Z^7.
\end{equation*}  

\begin{figure}[t]
	\includegraphics[width=.88\textwidth]{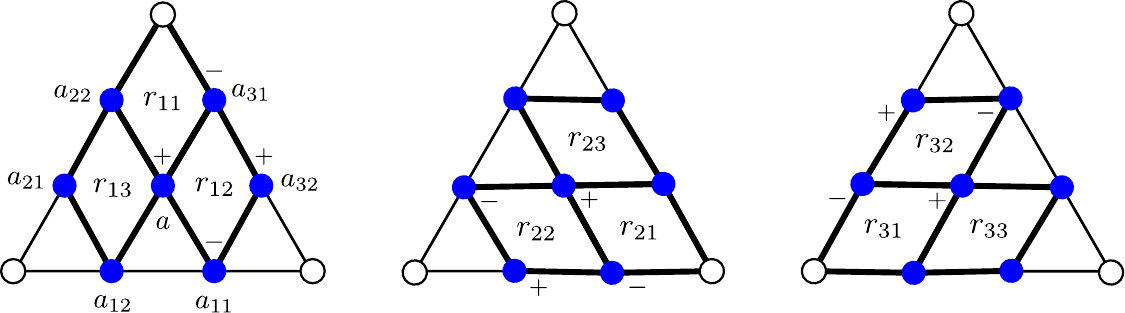}
	\centering
	\caption{Rhombus numbers}
	\label{fig:triangle-7-tropical-coordinates}
\end{figure}

Let $\Z/3 \subset \Q$ denote the set of integer thirds within the rational numbers, namely $\Z/3$ is the image of the map $\Z \to \Q$ sending $n \mapsto n/3$.  Note that $\Z \subset \Z/3$.  

To each point $c \in  \Z^7$, as in Equation \eqref{eq:integer-point}, associate a 9-tuple of \textit{rhombus numbers}
\begin{equation*}
	r(c) = \left( r_{11}, r_{12}, r_{13}, r_{21}, r_{22}, r_{23}, r_{31}, r_{32}, r_{33}\right)
	\quad	\in \left(\Z/3\right)^9
\end{equation*}
by the linear equations (see Figure \ref{fig:triangle-7-tropical-coordinates} above)
\begin{equation*}
\begin{split}
	r_{12} &= \left(a + a_{32} - a_{11} - a_{31}\right)/3,	
	\quad\quad			r_{11} = \left(a_{22} + a_{31} - a - 0\right)/3, 
	\\		r_{13} &= \left(a_{21} + a - a_{12} - a_{22}\right)/3;
	\\		r_{22} &= \left(a + a_{12} - a_{21} - a_{11}\right)/3,	
	\quad\quad	r_{21} = \left(a_{32} + a_{11} - a-0\right)/3,	
	\\		r_{23} &= \left(a_{31} + a - a_{22} -a_{32}\right)/3;	
	\\		r_{32} &= \left(a+a_{22} -a_{31} -a_{21}\right)/3,
	\quad\quad	r_{31} = \left(a_{12} + a_{21} - a -0\right)/3,		
	\\		r_{33} &= \left(a_{11}+a-a_{32} -a_{12}\right)/3.
\end{split}
\end{equation*}

\begin{definition}
\label{def:knutson-tao-cone-triangle}
	The \textit{local Knutson-Tao positive cone}, or just \textit{local Knutson-Tao  cone} or \textit{local cone}, $\KTcone{\triang}$ associated to the dotted ideal triangle $\triang$ is defined by
\begin{equation*}
	\KTcone{\triang} = 
	\left\{
		c \in \Z^7 ; \quad
	r(c) = \left(r_{11}, r_{12}, r_{13}, r_{21}, r_{22}, r_{23}, r_{31}, r_{32}, r_{33}\right) \quad\in \Zpos^9 \subset \left(\Z/3\right)^9
	\right\}.
\end{equation*}
\end{definition}

By linearity, this indeed defines a cone contained in $\Z^7$.  We will prove below in this sub-section that $\KTcone{\triang} \subset \Zpos^7$ is, in fact, a positive cone.  

\begin{remark}  $  $
\label{rem:modulo-3-congruence-conditions} 
\begin{enumerate}
	\item\label{subrem:GS1}  The inequalities $3 r_{ij} \geq 0$ are known as the Knutson-Tao rhombus inequalities; see \cite[Appendix 2]{KnutsonJAmerMathsoc99} and \cite[\S 3.1]{GoncharovInvent15}.  Note that $3 r_{ij}$ is always in $\Z$ by definition.  We impose the additional modulo 3 congruence condition that the $r_{ij}$ are integers.  This is analogous to the parity condition imposed in \cite[\S 3.1]{FockArxiv97} in the case of $\mathrm{SL}_2$.  
	\item  By the proof of Proposition \ref{LEM:POSITIVITY-OF-CONE} below, we could just as well have taken rational coefficients $c \in \mathbb{Q}^7$ in Definition \ref{def:knutson-tao-cone-triangle} without changing the resulting cone $\KTcone{\triang} \subset \Zpos^7 \subset \mathbb{Q}^7$.  That is, any rational solution to $r(c) \in \Zpos^9$ is, in fact, a non-negative integer solution.  
\end{enumerate}
\end{remark}

To see that $\KTcone{\triang}$ is non-trivial, one checks that the image $\themap{\triang}^\mathrm{FG}(\webbasis{\triang})$ of the Fock-Goncharov local coordinate function $\themap{\triang}^\mathrm{FG} : \webbasis{\triang} \to \Zpos^7$ (\S \ref{ssec:local-coordinates-from-Fock-Goncharov-theory}) lies in the local cone $\themap{\triang}^\mathrm{FG}(\webbasis{\triang}) \subset \KTcone{\triang}$.  By property (1) in Definition \ref{def:local-coordinate-function}, it suffices to check this on the connected local webs in $\webbasis{\triang}$; recall Figure \ref{fig:triangle-hilbert-basis} above. Specifically, using the convention in Equation \eqref{eq:integer-point}, we have
\begin{equation*}
\label{eq:eight-KT-points}	
\begin{split}
	c(R_1) &= \themap{\triang}^\mathrm{FG}(R_1) = \left(0, 0, 1, 2, 2, 1, 1\right), 	
	\quad\quad\quad\quad\quad	c(L_1) = \themap{\triang}^\mathrm{FG}(L_1) = \left(0, 0, 2, 1, 1, 2, 2\right), 	
	\\	c(R_2) &= \themap{\triang}^\mathrm{FG}(R_2) = \left(2, 1, 0, 0, 1, 2, 1\right), 	
	\quad\quad\quad\quad\quad	c(L_2) =  \themap{\triang}^\mathrm{FG}(L_2) =\left(1, 2, 0, 0, 2, 1, 2\right),
	\\	c(R_3) &= \themap{\triang}^\mathrm{FG}(R_3) =\left(1, 2, 2, 1, 0, 0, 1\right), 
	\quad\quad\quad\quad\quad	c(L_3) = \themap{\triang}^\mathrm{FG}(L_3) =\left(2, 1, 1, 2, 0, 0, 2\right),	
	\\	c(H_n^{\mathrm{in}}) &= \themap{\triang}^\mathrm{FG}(H_n^{\mathrm{in}}) =\left(2n, n, 2n, n, 2n, n, 3n\right),
	\\	c(H_n^{\mathrm{out}}) &= \themap{\triang}^\mathrm{FG}(H_n^{\mathrm{out}}) =\left(n, 2n, n, 2n, n, 2n, 3n\right).
\end{split}
\end{equation*}
The associated 9-tuples of rhombus numbers are
\begin{equation*}
\label{eq:nine-diamond-numbers}
\begin{split}
	r(c(R_1)) &= \left(1, 0, 0, 0, 0, 0, 0, 0, 0\right), 	
	\quad\quad	r(c(L_1)) = \left(0, 1, 1, 0, 0, 0, 0, 0, 0\right), 	
	\\	r(c(R_2)) &= \left(0, 0, 0, 1, 0, 0, 0, 0, 0\right), 	
	\quad\quad	r(c(L_2)) = \left(0, 0, 0, 0, 1, 1, 0, 0, 0\right),	
	\\	r(c(R_3)) &= \left(0, 0, 0, 0, 0, 0, 1, 0, 0\right), 	
	\quad\quad	r(c(L_3)) = \left(0, 0, 0, 0, 0, 0, 0, 1, 1\right),	
	\\	r(c(H_n^{\mathrm{in}})) &= \left(0, 0, n, 0, 0, n, 0, 0, n\right),	
	\\  r(c(H_n^{\mathrm{out}})) &= \left(0, n, 0, 0, n, 0, 0, n, 0\right).
\end{split}
\end{equation*}	

By rank considerations, the eight cone points $c(R_1)$, $c(L_1)$, $c(R_2)$, $c(L_2)$, $c(R_3)$, $c(L_3)$, $c(H_n^{\mathrm{in}})$, $c(H_n^{\mathrm{out}})$  have a linear dependence relation over $\Z$.  For instance (see Figure \ref{fig:coordinate-dependence-relation}),
\begin{equation*}
	 c(H_n^\mathrm{out}) +  c(H_n^\mathrm{in})   = n \left( c(L_1) +  c(L_2) +  c(L_3) \right)
	 \quad \in \KTcone{\triang}.
\end{equation*}  

\begin{figure}[t]
	\centering
	\includegraphics[scale=.78]{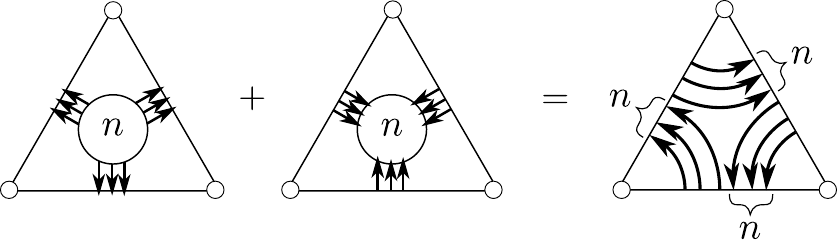}
	\caption{Linear dependence relation over $\Z$}
	\label{fig:coordinate-dependence-relation}
\end{figure}

Nevertheless, we can say the following:

\begin{proposition}
\label{LEM:POSITIVITY-OF-CONE}
	The collection of eight cone points 
\begin{equation*}
	c(R_1), c(L_1), c(R_2), c(L_2), c(R_3), c(L_3), c(H_1^{\mathrm{in}}), c(H_1^{\mathrm{out}}) \quad \in \themap{\triang}^\mathrm{FG}(\webbasis{\triang}) \subset \KTcone{\triang}
\end{equation*} 
forms a weak basis of the Knutson-Tao local cone $\KTcone{\triang}$.  
	
	Among these eight cone points, the seven points 
\begin{equation*}
	c(R_1), c(L_1), c(R_2), c(L_2), c(R_3), c(L_3), c(H_1^\mathrm{in}) 
\end{equation*}
are strongly independent over $\Zpos$, and the seven  points 
\begin{equation*}
	c(R_1), c(L_1), c(R_2), c(L_2), c(R_3), c(L_3), c(H_1^\mathrm{out})
\end{equation*} 
are strongly independent over $\Zpos$.  Moreover, each cone point $c$ in $\KTcone{\triang}$ can be uniquely expressed in exactly one of the following three forms:
\begin{align*}	
	c &= n_1 c(R_1) + n_2 c(L_1) + \cdots + n_6 c(L_3),		\\	
	c &= n_1 c(R_1) + n_2 c(L_1) + \cdots + n_6 c(L_3) + n c(H_1^\mathrm{in}),			\\	
	c &= n_1 c(R_1) + n_2 c(L_1) + \cdots + n_6 c(L_3) + n c(H_1^\mathrm{out}),
\quad\quad\left(n_i \in \Zpos, \quad n \in \Z_{>0}\right).	
\end{align*}
\end{proposition}

Because the spanning set $c(R_1), c(L_1), c(R_2), c(L_2), c(R_3), c(L_3), c(H_1^\mathrm{in}), c(H_1^\mathrm{out})$ consists of positive points, we immediately obtain:

\begin{corollary}
\label{cor:first-corollary}
	The local Knutson-Tao cone satisfies the property that $\KTcone{\triang} = \themap{\triang}^\mathrm{FG}(\webbasis{\triang}) \subset \Zpos^7$.  In particular, $\KTcone{\triang}$ is a positive cone.  \qed
\end{corollary}

\begin{corollary}
\label{cor:second-corollary}
	The Fock-Goncharov local coordinate function $\themap{\triang}^\mathrm{FG}: \webbasis{\triang} \to \KTcone{\triang}$ satisfies property (4) in Definition {\upshape\ref{def:local-coordinate-function}}, namely, the induced function $\themap{\triang}^\mathrm{FG}: [\webbasis{\triang}] \hookrightarrow \KTcone{\triang}$, defined on the collection of corner-ambiguity classes $[W_\triang]$ of local webs $W_\triang$ in $\webbasis{\triang}$, is an injection.  
\end{corollary}

\begin{proof}
	Assume $\themap{\triang}^\mathrm{FG}(W_\triang) = \themap{\triang}^\mathrm{FG}(W^\prime_\triang) \in \KTcone{\triang}$.  This cone point falls into one of the three families in Proposition \ref{LEM:POSITIVITY-OF-CONE}.  For the sake of argument, suppose 
\begin{equation*}
	\themap{\triang}^\mathrm{FG}(W_\triang) = \themap{\triang}^\mathrm{FG}(W^\prime_\triang) = n_1 c(R_1) + n_2 c(L_1) + \cdots + n_6 c(L_3) + n c(H_1^\mathrm{in})
	\quad\left(n_i \in \Zpos, \quad n \in \Z_{>0}\right).  
\end{equation*}
Note that $n c(H_1^\mathrm{in}) = c(H_n^\mathrm{in})$ in $\KTcone{\triang}$; see  Figure \ref{fig:triangle-hilbert-basis}.  By the uniqueness property in Proposition \ref{LEM:POSITIVITY-OF-CONE} together with property (1) in Definition \ref{def:local-coordinate-function}, we gather that $W_\triang$ and $W^\prime_\triang$ have $1 + \sum_{i=1}^6 n_i$ connected components, one of which is a $n$-in-honeycomb $H_n^\mathrm{in}$, and $n_1$ (resp. $n_2$, $n_3$, $n_4$, $n_5$, $n_6$) of which are corner arcs $R_1$ (resp. $L_1$, $R_2$, $L_2$, $R_3$, $L_3$).  The only ambiguity is how these corner arcs are permuted on their respective corners, that is  $[W_\triang] = [W^\prime_\triang]$ in $[\webbasis{\triang}]$.   
\end{proof}

\begin{proof}[Proof of Proposition \ref{LEM:POSITIVITY-OF-CONE}]
	Define two subsets $\overline{(\KTcone{\triang})^{\mathrm{in}}}$ and $(\KTcone{\triang})^{\mathrm{out}}$ of $\KTcone{\triang}$ by
\begin{align*}
\tag{$\#$}
\label{eq:def-of-in-out-partition}
	\overline{(\KTcone{\triang})^{\mathrm{in}}} &= \mathrm{Span}_{\Zpos}\left(c(R_1), c(L_1), c(R_2), c(L_2), c(R_3), c(L_3)\right) + \Zpos \cdot c(H_1^{\mathrm{in}}), \\  \tag{$\#\#$}  \label{eq:def-of-in-out-partition2}
	(\KTcone{\triang})^{\mathrm{out}} &= \mathrm{Span}_{\Zpos}\left(c(R_1), c(L_1), c(R_2), c(L_2), c(R_3), c(L_3)\right) + \Z_{> 0} \cdot c(H_1^{\mathrm{out}}).  
\end{align*}
(Here, $A+B = \{ a+b; \quad a \in A \text{ and } b \in B \}$.)  Put
\begin{align*}
	c_1 &= c(R_1), &
	 c_2 &= c(L_1), &
	 c_3 &= c(R_2), &
	 c_4 &= c(L_2), \\	
	 c_5 &= c(R_3), &
	 c_6 &= c(L_3), &
	 c_7 &= c(H_1^\mathrm{in}), &
	 c_8 &= c(H_1^\mathrm{out}).  
\end{align*}
By Lemma \ref{lem:KTcone-technical-lemma}, with $\Cone = \KTcone{\triang}$, in order to prove Proposition \ref{LEM:POSITIVITY-OF-CONE} it suffices to establish:  

\begin{claim}
	There exists
	\begin{enumerate}
		\item    a cone $\Cone^\prime \subset \Z^7$;
		\item  a collection of cone points $c'_1, \dots, c'_8$ in $\Cone^\prime$;
		\item   a partition $\Cone^\prime = \overline{(\Cone^\prime)^{>0}} \sqcup (\Cone^\prime)^{< 0}$;
		\item  a $\Zpos$-linear bijection $\psi \colon \Cone^\prime \to \KTcone{\triang}$; 
		\item  an extension $\widetilde{\psi}$ of $\psi$ to a $\Q$-linear isomorphism $\widetilde{\psi} \colon \Q^7 \to \Q^7$;
	\end{enumerate}
	such that
	\begin{enumerate}
		\item  we have $\psi(c'_i) = c_i$;
		\item  we have $\psi(\overline{(\Cone^\prime)^{>0}}) = \overline{(\KTcone{\triang})^\mathrm{in}}$ and $\psi((\Cone^\prime)^{<0}) = (\KTcone{\triang})^\mathrm{out}$;
		\item  the eight cone points $c'_1, \dots, c'_6, c'_7, c'_8$ form a weak basis of the cone $\Cone^\prime$;
		\item  the seven cone points $c'_1, \dots, c'_6, c'_7$ are strongly independent over $\Zpos$;
		\item  the seven cone points $c'_1, \dots, c'_6, c'_8$ are strongly independent over $\Zpos$.  
	\end{enumerate}
\end{claim}

We prove the claim.  Define $\Cone^\prime \subset \Zpos^6 \times \Z \subset \Z^7$ by
\begin{equation*}
\tag{$\ast\ast$}
\label{eq:second-cone-def}
	\Cone^\prime = \left\{	\left(r_{11}, r_{12}, r_{21}, r_{22}, r_{31}, r_{32}, x\right) \in \Zpos^6 \times \Z;
	\quad	-x \leq \mathrm{min}\left(r_{12}, r_{22}, r_{32}\right)	\right\}.
\end{equation*}
It follows from the definition that $\Cone^\prime$ is a cone.  Put
\begin{align*}	
	c'_1 &= \left(1, 0, 0, 0, 0, 0, 0\right), 	&
	c'_2 &= \left(0, 1, 0, 0, 0, 0, 0\right), 	\\	
	c'_3 &= \left(0, 0, 1, 0, 0, 0, 0\right), 	&
	c'_4 &= \left(0, 0, 0, 1, 0, 0, 0\right),	\\
	c'_5 &= \left(0, 0, 0, 0, 1, 0, 0\right), 	&
	c'_6 &= \left(0, 0, 0, 0, 0, 1, 0\right),	\\	
	\label{eq:transformed-cone-points}
	c'_7 &= \left(0, 0, 0, 0, 0, 0, 1\right),	&
	c'_8 &= \left(0,1,0,1,0,1,-1\right).
\end{align*}
One checks that $c'_1, \dots, c'_8$ are in $\Cone^\prime$.  Define
\begin{equation*}
	\overline{(\Cone^\prime)^{>0}} = \Cone^\prime \cap \left( \Zpos^6 \times \Zpos\right),
	\quad\quad	(\Cone^\prime)^{<0} = \Cone^\prime \cap \left( \Zpos^6 \times \Z_{<0}\right).
\end{equation*}
Then $\Cone^\prime = \overline{(\Cone^\prime)^{>0}} \sqcup (\Cone^\prime)^{<0}$ is a partition.  

First, we show $c'_1, \dots, c'_6, c'_7, c'_8$ spans $\Cone^\prime$.  We see that
\begin{equation*}	\label{spanning-equation-1}
\tag{$\dagger$}
	\overline{(\Cone^\prime)^{>0}} = \mathrm{Span}_{\Zpos}\left(c'_1, \dots, c'_6\right) + \Zpos \cdot c'_7 \quad\left(= \Zpos^6 \times \Zpos \right).
\end{equation*}	
If $c' \in (\Cone^\prime)^{<0}$, then its last coordinate is $x \leq -1$.  Since $-x > 0$ and $-x \leq \mathrm{min}(r_{12}, r_{22}, r_{32})$, 
\begin{equation*}
	c' = \left(r_{11}, -x+r'_{12}, r_{21}, -x+r'_{22}, r_{31}, -x+r'_{32}, -x \cdot -1\right)
\end{equation*}
for some $r_{11}, r'_{12}, r_{21}, r'_{22}, r_{31}, r'_{32} \in \Zpos$ and $-x \in \Z_{>0}$.  That is, 
\begin{equation*}
	c' = r_{11} c'_1 + r'_{12} c'_2 + r_{21} c'_3 + r'_{22} c'_4
	+ r_{31} c'_5 + r'_{32} c'_6 + \left(-x\right) c'_8
	\quad
	\in \mathrm{Span}_{\Zpos}\left(c'_1, \dots, c'_6\right) + \Z_{>0} \cdot c'_8.
\end{equation*}
  Thus, 
\begin{equation*}	\label{spanning-equation-2}
\tag{$\dagger\dagger$}
	(\Cone^\prime)^{<0} = \mathrm{Span}_{\Zpos}\left(c'_1, \dots, c'_6\right) + \Z_{>0} \cdot c'_8,
\end{equation*}
where the $\supseteq$ containment follows since $\mathrm{Span}_{\Zpos}\left(c'_1, \dots, c'_6\right) + \Z_{>0} \cdot c'_8 \subset \Zpos^6 \times \Z_{<0}$.  

Next, we show $c'_1, \dots, c'_6, c'_7, c'_8$ are weakly independent over $\Zpos$.  Indeed, if $n_1 c'_1 + \cdots + n_7 c'_7 + n_8 c'_8 = 0$, then $n_1 = n_3 = n_5 = 0$ and $n_2 + n_8, n_4 + n_8, n_6 + n_8, n_7 - n_8 = 0$.  Since all $n_i \in \Zpos$, it follows that $n_2 = n_4 = n_6 = n_8 = 0$, and so $n_7 = n_8 = 0$, as desired.  

We gather that $c'_1, \dots, c'_6, c'_7, c'_8$ form a weak basis of $\Cone^\prime$.  

Next, we show $c'_1, \dots, c'_6, c'_7$ are strongly independent over $\Zpos$.  This is equivalent to being linearly independent over $\Q$, which follows from the definitions.  Similarly, it follows from the definitions that $c'_1, \dots, c'_6, c'_8$ are strongly independent over $\Zpos$.  

We now define a $\Zpos$-linear bijection $\varphi \colon \KTcone{\triang} \to \Cone^\prime$. Its inverse will be the desired $\Zpos$-linear bijection $\psi = \varphi^{-1} \colon \Cone^\prime \to \KTcone{\triang}$.  Let $c$ be a cone point in $\KTcone{\triang}$, written as in Equation \eqref{eq:integer-point}.  Put
\begin{equation*}
\begin{split}
	x &=  \left(a_{11} - a_{12} + a_{21} - a_{22} + a_{31} - a_{32}\right)/3
	\\	&=  r_{13} - r_{12} = r_{23} - r_{22} = r_{33} - r_{32}
	\\	
	&\geq -r_{12} \text{ and} -r_{22} \text{ and} -r_{32}
\end{split}
\end{equation*}
where the rhombus numbers $r_{ij}$ are in $\Zpos$ since $c \in \KTcone{\triang}$; see Figure \ref{fig:dotted-triangle-def-of-x} (we think of $x$ as the tropical Fock-Goncharov $\mathcal{X}$-coordinate for the triangle).  Thus,
\begin{equation*}
	x \geq \mathrm{max}(-r_{12}, -r_{22}, -r_{32}) = - \mathrm{min}(r_{12}, r_{22}, r_{32}).
\end{equation*}  
Therefore, recalling $\Cone^\prime \subset \Zpos^6 \times \Z$ (Equation \eqref{eq:second-cone-def}), we may define the function $\varphi \colon \KTcone{\triang} \to \Cone^\prime$~by 
\begin{equation*}
	\varphi(c) = (r_{11}, r_{12}, r_{21}, r_{22}, r_{31}, r_{32}, x).  
\end{equation*}

It follows from the  definition that $\varphi \colon \KTcone{\triang} \to \Cone^\prime$ is $\Zpos$-linear.  One checks that $\varphi(c_i) = c'_i$.  Since the $c'_i$ span $\Cone^\prime$, we have $\varphi$ is surjective.  In particular, by Equations \eqref{eq:def-of-in-out-partition},\eqref{eq:def-of-in-out-partition2},\eqref{spanning-equation-1},\eqref{spanning-equation-2},
\begin{equation*}
	\varphi\left(\overline{(\KTcone{\triang})^\mathrm{in}}\right) = \overline{(\Cone^\prime)^{>0}}
	\quad  \text{and}  \quad
	\varphi\left((\KTcone{\triang})^\mathrm{out}\right) = (\Cone^\prime)^{<0}.
\end{equation*}

The formula for $\varphi$ extends to define a $\Q$-linear isomorphism $\widetilde{\varphi} \colon \Q^7 \to \Q^7$, and its inverse is the desired $\Q$-linear isomorphism $\widetilde{\psi} = (\widetilde{\varphi})^{-1} \colon \Q^7 \to \Q^7$.  Indeed, the bijectivity of $\widetilde{\varphi}$ follows by computing the values on the standard column basis of $\Q^7$, giving the invertible matrix
\begin{equation*}
	\widetilde{\varphi}(\vec{e}_1, \vec{e}_2, \vec{e}_3, \vec{e}_4, \vec{e}_5, \vec{e}_6, \vec{e}_7) = \frac{1}{3}
	\left( \begin{smallmatrix}
		0&0&0&1&1&0&-1	\\
		-1&0&0&0&-1&1&1	\\
		1&0&0&0&0&1&-1	\\
		-1&1&-1&0&0&0&1	\\
		0&1&1&0&0&0&-1	\\
		0&0&-1&1&-1&0&1	\\
		1&-1&1&-1&1&-1&0	\\
	\end{smallmatrix} \right).  
\end{equation*}

So $\widetilde{\psi} = (\widetilde{\varphi})^{-1}$ is defined.  Since $\widetilde{\varphi}$ is an injection, so is its restriction $\varphi \colon \KTcone{\triang} \to \Cone^\prime$.  Also, since, as we argued above, $\varphi$ is a surjection, we gather $\varphi$ is a bijection.  Thus, $\psi = \varphi^{-1} : \Cone^\prime \to \KTcone{\triang}$ is defined.  This completes the proof of the claim, thereby establishing the proposition.
\end{proof}

\begin{figure}[htb]
	\centering
	\includegraphics[width=\textwidth]{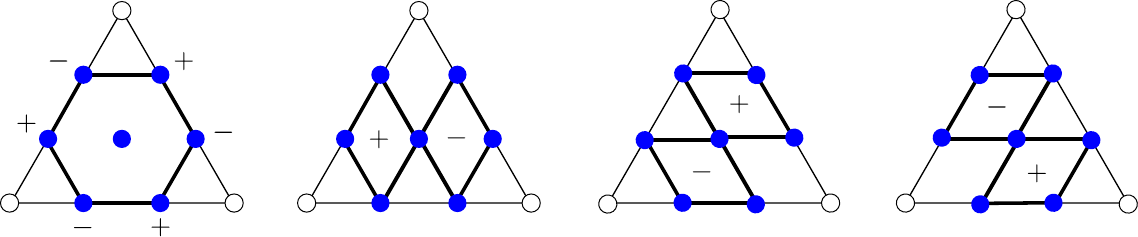}
	\caption{Four ways to view the tropical Fock-Goncharov $\mathcal{X}$-coordinate}
	\label{fig:dotted-triangle-def-of-x}
\end{figure}

		\subsection{Global Knutson-Tao cone}
		\label{ssec:global-knutson-tao-cone}

Given the dotted ideal triangulation $\idealtriang$ on the surface $\surf$, an element $c$ of $\Z^N$ corresponds to a function $\{\text{dots on } \idealtriang\} \to \Z$; see \S \ref{ssec:dotted-ideal-triangulations}.  If $\triang$ is a dotted triangle of $\idealtriang$, then an element $c$ of $\Z^N$ induces a function $\{ \text{dots on } \triang \} \to \Z$, which likewise corresponds to an element $c_\triang$ of $\Z^7$. 

\begin{definition}
\label{def:global-cone-point}
	The \textit{global Knutson-Tao positive cone}, or just \textit{Knutson-Tao cone} or \textit{global cone}, $\KTcone{\idealtriang} \subset \Zpos^N$ is defined by
\begin{equation*}
	\KTcone{\idealtriang} =
	\left\{
		c \in \Z^N;
		\quad  c_\triang \text{ is in } \KTcone{\triang} 
		 \text{ for all triangles } \triang \text{ of } \idealtriang
	\right\}.  
\end{equation*}
\end{definition} 

It follows from Corollary \ref{cor:first-corollary} that $\KTcone{\idealtriang} \subset \Zpos^N$ is indeed a positive cone.  

In \S \ref{sec:global-coordinates-for-non-elliptic-webs}, we defined the global coordinate function $\themap{\idealtriang}^\FG : [\webbasis{\surf}] \to \Zpos^N$; see Definition \ref{def:fock-goncharov-global-coordinate-function}.  Since the image $\themap{\triang}^\FG([\webbasis{\triang}]) \subset \KTcone{\triang}$ (which is, in fact, an equality by Corollary \ref{cor:first-corollary}), it follows by the construction of $\themap{\idealtriang}^\FG$ that the image $\themap{\idealtriang}^\FG([\webbasis{\surf}]) \subset \KTcone{\idealtriang}$; recall, for instance, Figure \ref{fig:coordinates-example}.  

\begin{proposition}
\label{prop:surjectivity}
	Moreover, we have (see  {\upshape\S \ref{sec:main-theorem}-\ref{sec:proof-of-main-lemma}} for a proof)
\begin{equation*}
	\themap{\idealtriang}^\FG([\webbasis{\surf}]) = \KTcone{\idealtriang}.  
\end{equation*}
\end{proposition}

		\section{Main result: global coordinates}
		\label{sec:main-theorem}

We summarize what we have done so far.  Consider a punctured surface $\surf$ with empty boundary; see \S \ref{sec:global-webs}.  Let $[\webbasis{\surf}]$ denote the collection of parallel-equivalence classes of global non-elliptic webs on $\surf$.  Assume that $\surf$ is equipped with an ideal triangulation $\idealtriang$.  For $N = - 8 \chi(\surf)$, in \S \ref{sec:global-coordinates-for-non-elliptic-webs} we defined the Fock-Goncharov global coordinate function $\themap{\idealtriang}^\FG : [\webbasis{\surf}] \to \Zpos^N$, depending on the choice of the ideal triangulation $\idealtriang$.  Proposition \ref{prop:injectivity}, which still needs to be proved, says that the mapping $\themap{\idealtriang}^\FG$ is injective.  In \S \ref{sec:knutson-tao-positive-integer-cone}, we defined the global Knutson-Tao positive cone $\KTcone{\idealtriang} \subset \Zpos^N$, which also depends on the ideal triangulation $\idealtriang$.  By construction, the image $\themap{\idealtriang}^\FG([\webbasis{\surf}]) \subset \KTcone{\idealtriang}$.  According to Proposition \ref{prop:surjectivity}, which also still needs to be proved, $\themap{\idealtriang}^\FG$ maps $[\webbasis{\surf}]$ onto $\KTcone{\idealtriang}$.  Therefore, assuming Propositions \ref{prop:injectivity} and \ref{prop:surjectivity}, we have proved:
	
\begin{theorem}
\label{thm:main-theorem-1}
	The Fock-Goncharov global coordinate function
\begin{equation*}
	\themap{\idealtriang}^\FG : [\webbasis{\surf}]
	\overset{\sim}{\longrightarrow} \KTcone{\idealtriang} \subset \mathbb{Z}_{\geq 0}^N
\end{equation*}
is a bijection of sets.  \qed
\end{theorem}

\begin{remark}
In \S \ref{sec:webs-on-surfaces-with-boundary}, we generalize Theorem \ref{thm:main-theorem-1} to the setting of surfaces-with-boundary $\surfbord$.  
\end{remark}

		\subsection{Inverse mapping}
		\label{ssec:inverse-mapping}

Our strategy for proving Propositions \ref{prop:injectivity} and \ref{prop:surjectivity} (equivalently, Theorem \ref{thm:main-theorem-1}) is to construct an explicit inverse mapping
\begin{equation*}
	\invmap{\idealtriang}^\FG:
	\KTcone{\idealtriang}
	\longrightarrow
	[\webbasis{\surf}]
\end{equation*}
namely a function that is both a left and a right inverse for the function $\themap{\idealtriang}^\FG$.  The definition of the mapping $\invmap{\idealtriang}^\FG$ is relatively straightforward, and it will be automatic that it is an inverse for $\themap{\idealtriang}^\FG$.  The more challenging part will be to show that $\invmap{\idealtriang}^\FG$ is well-defined.

		\subsection{Inverse mapping: ladder gluing construction}
		\label{ssec:ladder-gluing-construction}

Recall that for a triangle $\triang$ we denote by $\webbasis{\triang}$ the collection of rung-less essential local webs $W_\triang$ in $\triang$; see Definition \ref{def:corner-ambiguity}.  We will once again make use of the split ideal triangulation $\splitidealtriang$; see \S \ref{ssec:split-ideal-triangulations}.

\begin{definition}
\label{def:compatible}
	A collection $\{ W_\triang \}_{\triang \in \splitidealtriang }$ of local webs $W_\triang \in \webbasis{\triang}$, varying over the triangles $\triang$ of $\splitidealtriang$, is \textit{compatible} if for each biangle $\biang$, with boundary edges $E^\prime$ and $E^{\prime\prime}$, sitting between two triangles $\triang^\prime$ and $\triang^{\prime\prime}$, respectively, the number of out-strands (resp. in-strands) of $W_{\triang^\prime}$ on $E^\prime$ is equal to the number of in-strands (resp. out-strands) of $W_{\triang^{\prime\prime}}$ on $E^{\prime\prime}$.  
\end{definition}

For example, see the third row of Figure \ref{fig:coordinates-example2}, an example on a once-punctured torus.  

To a compatible collection $\{ W_\triang \}_{\triang \in \splitidealtriang }$ of local webs, we will associate a global web $W$ on $\surf$ that \underline{need not be non-elliptic} and that is in good position with respect to $\splitidealtriang$; recall Definition \ref{def:good-position}.  The global web $W$ is well-defined up to ambient isotopy of $\surf$ respecting $\splitidealtriang$.  

\textit{Construction of $W$.}  Consider a biangle $\biang$ sitting between two triangles $\triang^\prime$ and $\triang^{\prime\prime}$.  The local webs $W_{\triang^\prime}$ and $W_{\triang^{\prime\prime}}$ determine strand sets $S^\prime$ and $S^{\prime\prime}$ on the boundary edges $E^\prime$ and $E^{\prime\prime}$, respectively.  By the compatibility property, the strand-set pair $S = (S^\prime, S^{\prime\prime})$ is symmetric; see Definition \ref{def:ladder-web}.  Let $W_\biang = W_\biang(S)$ be the induced ladder-web in $\biang$; see Definition \ref{def:actual-def-of-ladder-web}.  

Define $W$ to be the global web obtained by gluing together the local webs $\{ W_\triang \}_{\triang \in \splitidealtriang}$ and $\{ W_\biang \}_{\biang \in \splitidealtriang}$ in the obvious way; see the fourth row and the left side of the fifth row of Figure~\ref{fig:coordinates-example2}.  

\begin{definition}
We say that the global web $W$ has been obtained from the compatible collection $\{ W_\triang \}_{\triang \in \splitidealtriang }$ of local webs by applying the \textit{ladder gluing construction}.  
\end{definition}

The following statement is immediate.

\begin{lemma}
\label{lem:basics-of-ladder-construction}
	A global web $W$ obtained via the ladder gluing construction is in good position with respect to $\splitidealtriang$.  Conversely, if $W$ is a global web in good position, then $W$ can be recovered as the result of applying the ladder gluing construction to $\{ W_\triang = W \cap \triang \}_{\triang \in \splitidealtriang}$.  \qed
\end{lemma}

If the global web $W$ is obtained via the ladder gluing construction, then $W$ could be (1) non-elliptic, for example see the left side of the fifth row of Figure \ref{fig:coordinates-example2}, or (2) elliptic, for example see the fourth row of Figure \ref{fig:coordinates-example2}.  

\begin{figure}[htb]
	\centering
	\includegraphics[scale=.54]{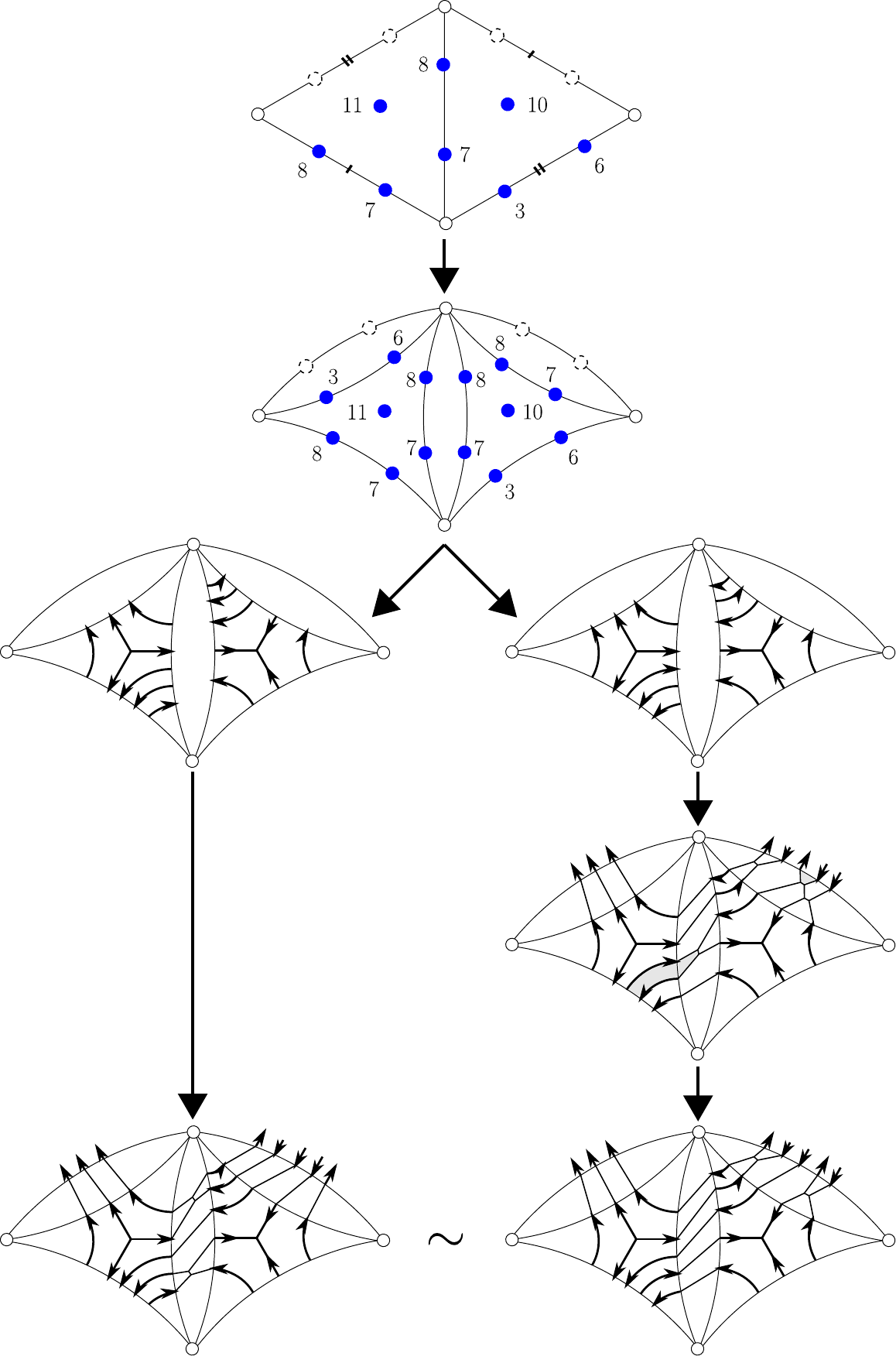}
	\caption{Ladder gluing construction (on the once-punctured torus).  Shown are two different ways of assigning the local webs, differing by permutations of corner arcs.  On the left, the result of the gluing is a non-elliptic web.  On the right, the result is an elliptic web, which has to be resolved by removing a square before becoming a non-elliptic web.  The two non-elliptic webs obtained in this way are equivalent.}
	\label{fig:coordinates-example2}
\end{figure}

		\subsection{Inverse mapping: resolving an elliptic web}
		\label{ssec:inverse-mapping-correcting-an-elliptic-web}

Recall the notion of a local parallel-move; see Figure \ref{fig:local-parallel-move-no-coordinates}.  Note that if $\{ W_\triang^\prime \}_{\triang \in \splitidealtriang}$ is a compatible collection of local webs, and if $W_\triang$ is related to $W^\prime_\triang$ by a sequence of local parallel-moves, then $\{ W_\triang\}_{\triang \in \splitidealtriang}$ is also compatible.  

\begin{lemma}
\label{lem:fixing-an-elliptic-web}
	Given a compatible collection $\{ W^\prime_\triang \}_{\triang \in \splitidealtriang}$ of local webs, there exist local webs $\{ W_\triang \}_{\triang \in \splitidealtriang}$ such that $W_\triang$ is related to $W^\prime_\triang$ by a sequence of local parallel-moves, and the global web $W$ obtained by applying the ladder gluing construction to $\{ W_\triang \}_{\triang \in \splitidealtriang}$ is non-elliptic. 
\end{lemma}

\begin{proof}
	Suppose that the global web $W^\prime$ obtained by applying the ladder gluing construction to the local webs $\{ W^\prime_\triang \}_{\triang \in \splitidealtriang}$ is elliptic.  
	
\textit{Step 1.}  We show that the elliptic global web $W^\prime$ has no disk- or bigon-faces.  If there were a disk- or bigon-face, then it could not lie completely in a triangle $\triang$ or biangle $\biang$ of $\splitidealtriang$, for this would violate that the  local web restriction $W^\prime_\triang$ or $W^\prime_\biang$ is essential (in particular, non-elliptic) by Lemma \ref{lem:basics-of-ladder-construction}.  Consequently, there is a cap- or fork-face lying in some $\triang$ or $\biang$, contradicting that the  local web restriction $W^\prime_\triang$ or $W^\prime_\biang$ is essential (in particular, taut).
	
\textit{Step 2.}  We consider the possible positions of square-faces relative to the split ideal triangulation $\splitidealtriang$.  We claim that a square-face can only appear as demonstrated at the top of Figure \ref{fig:resolving-an-elliptic-web}, namely having two H-faces in two (possibly identical) biangles $\biang$ and, in between, having opposite sides traveling parallel through the intermediate triangles $\triang$ and biangles $\biang$.  Indeed, otherwise there would be a square-, cap-, or fork-face, similar to Step 1.  
	
\textit{Step 3.}  We remove a square-face. Since the square-faces are positioned in this way, given a fixed square-face there is a well-defined \textit{state} into which the square-face can be resolved, illustrated in Figure \ref{fig:resolving-an-elliptic-web}.  The resulting global web $W_1$ is in good position with respect to $\splitidealtriang$.  Also, $W_1$ is less complex than  $W^\prime$, where the \textit{complexity} of a global web in good position is measured by the total number of vertices lying in the union $\cup_\biang \biang$ of all of the biangles $\biang$.  Note that resolving a square-face decreases the complexity by 4.  

The effect of resolving a square-face is to perform, in each triangle $\triang$, some number (possibly zero) of local parallel-moves, replacing the original local webs $\{ W^\prime_\triang \}_{\triang \in \splitidealtriang}$ with new local webs $\{ (W_1)_\triang \}_{\triang \in \splitidealtriang}$ such that $(W_1)_\triang$ is equivalent to $W^\prime_\triang$ up to corner-ambiguity; see Figure~\ref{fig:resolving-an-elliptic-web}.

\textit{Step 4.}  By a complexity argument, we can repeat the previous step until we obtain a sequence $W^\prime = W_0, W_1, W_2, \dots, W_n = W$ of global webs in good position such that $\{ (W_{i+1})_\triang \}_{\triang \in \splitidealtriang}$ is related to $\{ (W_{i})_\triang \}_{\triang \in \splitidealtriang}$ by  a sequence of local parallel-moves, and such that $W$ has no square-faces.  By Lemma \ref{lem:basics-of-ladder-construction}, $W$ is recovered by applying the ladder gluing construction to $\{ W_\triang \}_{\triang \in \splitidealtriang}$. By Step 1, $W$ has no disk- or bigon-faces.  Thus, $W$ is non-elliptic.
\end{proof}  

We refer to the algorithm used in the proof of Lemma \ref{lem:fixing-an-elliptic-web} as the \textit{square removing algorithm}.  For example, see the fourth row and the right side of the fifth row of Figure \ref{fig:coordinates-example2}.  

Note that the algorithm removes the square-faces at random, thus the local webs $\{ W_\triang \}_{\triang \in \splitidealtriang}$ satisfying the conclusion of Lemma \ref{lem:fixing-an-elliptic-web} are not necessarily unique.  For example, see Figure~\ref{fig:elliptic-web-non-example}.

\begin{figure}[htb]
	\centering
	\includegraphics[width=.38\textwidth]{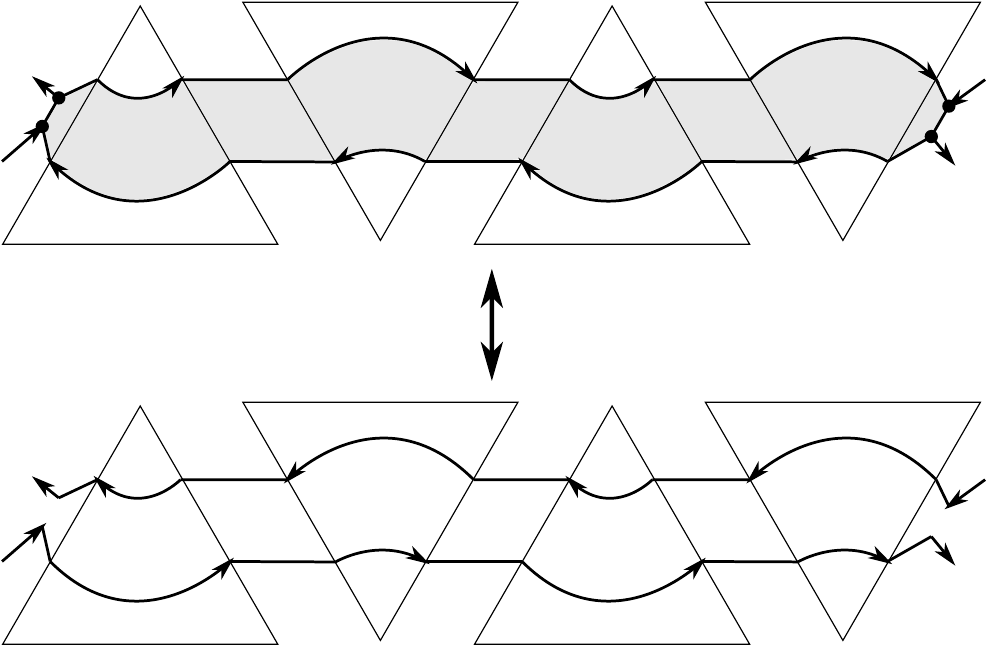}
	\caption{Resolving a square-face}
	\label{fig:resolving-an-elliptic-web}
\end{figure}

\begin{figure}[htb]
	\centering
	\includegraphics[width=.48\textwidth]{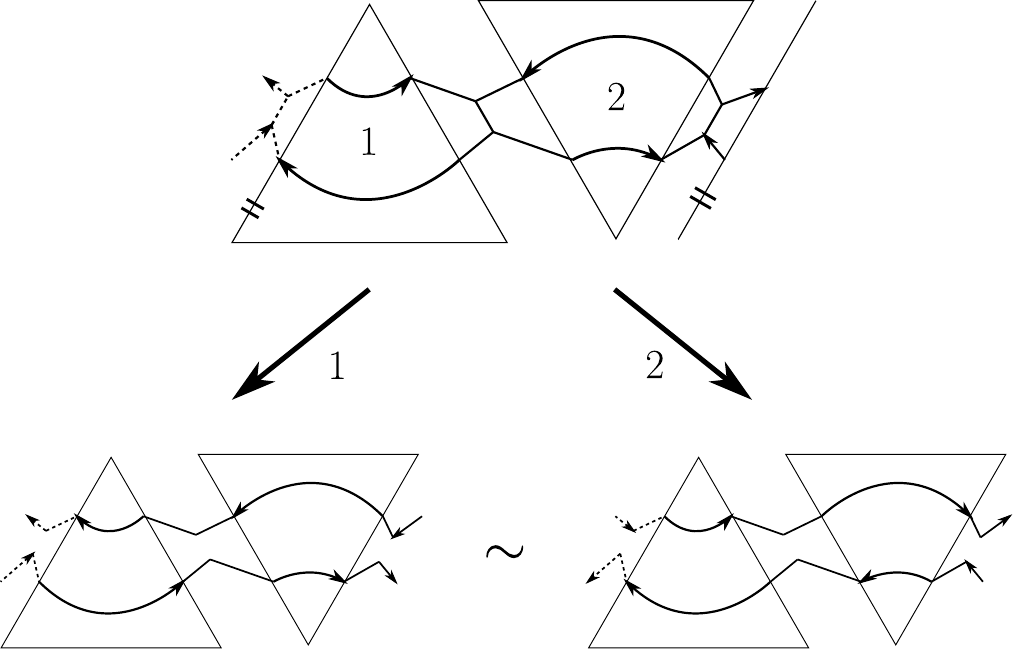}
	\caption{Elliptic web resulting from the ladder gluing construction (top), and
	two different applications of the square removing algorithm, yielding different, but parallel-equivalent, non-elliptic webs (bottom)}
	\label{fig:elliptic-web-non-example}
\end{figure}

		\subsection{Inverse mapping: definition}
		\label{ssec:inverse-mapping-def}

Let $c$ be a point in the global cone $\KTcone{\idealtriang}$; see Definition \ref{def:global-cone-point}.  Our goal is to associate to $c$ a parallel-equivalence class $\invmap{\idealtriang}^\FG(c) \in [\webbasis{\surf}]$ of global non-elliptic webs on $\surf$.  Equivalently, we want to associate to $c$ a non-elliptic web $\invmaptilde{\idealtriang}^\FG(c)$ on $\surf$ well-defined up to parallel-equivalence; see Definition \ref{def:parallel-equivalent-webs}.  Recall that we identify the triangles $\triang$ of the ideal triangulation $\idealtriang$ with the triangles $\triang$ of the split ideal triangulation $\splitidealtriang$.  

\textit{Construction of $\invmaptilde{\idealtriang}^\FG(c)$.}  The global cone point $c$ determines a local cone point $c_\triang$ in the local cone $\KTcone{\triang}$ for each triangle $\triang$ of $\idealtriang$; see just before Definition \ref{def:global-cone-point}.  By the triangle identifications between $\idealtriang$ and $\splitidealtriang$, the local cone point $c_\triang \in \KTcone{\triang}$ is assigned to each triangle $\triang$ of $\splitidealtriang$; see the first and second rows of Figure \ref{fig:coordinates-example2}.  
	
Note, by construction, corresponding edge-coordinates located across a biangle $\biang$ take the same value.  More precisely, if $\biang$ sits between two triangles $\triang^\prime$ and $\triang^{\prime\prime}$ and if the boundary edges of $\biang$ are $E^\prime$ and $E^{\prime\prime}$, respectively, then the coordinate $a^L_{E^\prime}$ (resp. $a^R_{E^\prime}$) lying on the left-edge-dot (resp. right-edge-dot) as viewed from $\triang^\prime$ agrees with the coordinate $a^R_{E^{\prime\prime}}$ (resp. $a^L_{E^{\prime\prime}}$) lying on the right-edge-dot (resp. left-edge-dot) as viewed from $\triang^{\prime\prime}$; see Figure \ref{fig:coordinates-example2}.

By Corollaries \ref{cor:first-corollary} and \ref{cor:second-corollary}, for each local cone point $c_\triang \in \KTcone{\triang}$ assigned to a triangle $\triang$ of $\splitidealtriang$, there exists a unique corner-ambiguity class $[W_\triang]$ of local webs $W_\triang$ in $\webbasis{\triang}$ such that $\themap{\triang}^\FG(W_\triang) = c_\triang$ for any representative $W_\triang$ of $[W_\triang]$.  

We now make a \underline{choice} of such a representative $W_\triang$ for each $\triang$.  Two different choices $W_\triang$ and $W^\prime_\triang$ of local webs representing $[W_\triang] = [W^\prime_\triang]$ are, by definition, related by  local parallel-moves; see the third row of Figure \ref{fig:coordinates-example2}.  

Since corresponding edge-coordinates across biangles agree, the collection $\{ W_\triang \}_{\triang \in \splitidealtriang}$ of local webs is compatible (Definition \ref{def:compatible}).  This follows by Figure \ref{fig:triangle-hilbert-basis}.  (There is also a general argument, by properties (2) and (3) in Definition \ref{def:local-coordinate-function}, which uses the fact that if $W_\triang \in \webbasis{\triang}$, then the opposite web $W^\mathrm{op}_\triang$ obtained by reversing all of the orientations of $W_\triang$ is also in $\webbasis{\triang}$).  

By Lemma \ref{lem:fixing-an-elliptic-web}, this choice of a compatible collection $\{ W_\triang \}_{\triang \in \splitidealtriang}$ of local webs can be made (in a non-unique way) such that the global web $W$ on $\surf$ obtained by applying the ladder gluing construction to $\{ W_\triang \}_{\triang \in \splitidealtriang}$ is non-elliptic.  Finally, we define $\invmaptilde{\idealtriang}^\FG(c) = W$.  In order for the global web $\invmaptilde{\idealtriang}^\FG(c)$ to be well-defined up to parallel-equivalence, we require:  

\begin{mainlemma}
\label{lem:main-lemma}
	Assume that each of $\{W_\triang \}_{\triang \in \splitidealtriang}$ and $\{W^\prime_\triang\}_{\triang \in \splitidealtriang}$ is a compatible collection of rung-less essential webs in the $\webbasis{\triang}$, satisfying
	\begin{enumerate}
		\item  for each triangle $\triang$, the local webs $W_\triang$ and $W^\prime_\triang$ are equivalent up to corner-ambiguity;
		\item  both global webs $W$ and $W^\prime$, obtained from the compatible collections $\{ W_\triang \}_{\triang \in \splitidealtriang}$ and $\{ W^\prime_\triang \}_{\triang \in \splitidealtriang}$, respectively, by applying the ladder gluing construction, are non-elliptic.
	\end{enumerate}
	Then, the non-elliptic webs $W$ and $W^\prime$ represent the same parallel-equivalence class in $[\webbasis{\surf}]$.  
\end{mainlemma} 

\begin{definition}
	The \textit{inverse mapping}
\begin{equation*}
	\invmap{\idealtriang}^\FG : \KTcone{\idealtriang} \longrightarrow [\webbasis{\surf}]
\end{equation*}
	is defined by sending a cone point $c$ in the global Knutson-Tao cone $\KTcone{\idealtriang}$ to the parallel-equivalence class in $[\webbasis{\surf}]$ of the global non-elliptic web $\invmaptilde{\idealtriang}^\FG(c)$ on $\surf$.  
\end{definition}

\begin{proof}[Proof of Propositions \ref{prop:injectivity} and \ref{prop:surjectivity}]
	Assuming Main Lemma \ref{lem:main-lemma} to be true, it follows immediately from the constructions that the well-defined mapping $\invmap{\idealtriang}^\FG : \KTcone{\idealtriang} \to [\webbasis{\surf}]$  is the set-functional inverse of the Fock-Goncharov global coordinate function $\themap{\idealtriang}^\FG : [\webbasis{\surf}] \to \KTcone{\idealtriang}$.  
\end{proof}

In summary, we have reduced the proof of Theorem \ref{thm:main-theorem-1} to proving the main lemma.

		\section{Proof of the main lemma}
		\label{sec:proof-of-main-lemma}

In this section, we prove Main Lemma \ref{lem:main-lemma}.  In particular, we provide an explicit algorithm taking one web to the other by a sequence of modified H-moves and global parallel-moves.  

The strategy of the proof is simple, whereas its implementation is more complicated due to the combinatorics.  The key idea is to think of a web $W$ not as a graph, but as a multi-curve $\left< W \right>$, which we call a \textit{web picture}; see Figure \ref{fig:BEV-multicurve-example12}.  We have already previewed web pictures at the local level, in Definitions \ref{def:biangle-local-picture} and \ref{def:triangle-local-picture}  (see also the second paragraph of \S \ref{ssec:local-coordinates-from-Fock-Goncharov-theory}).  

If $W$ and $W^\prime$ are two non-elliptic webs as in Main Lemma \ref{lem:main-lemma}, we show that their associated multi-curves $\left< W \right>$ and $\left< W^\prime \right>$ satisfy a fellow-travel property; see Lemma \ref{lem:life-neighbors-lemma}.  As a consequence of this Fellow Traveler Lemma, the intersection points $\mathscr{P} \subset \left< W \right>$ are in natural bijection with those $\mathscr{P}^\prime \subset \left< W^\prime \right>$; here, the non-elliptic hypothesis is necessary.  To finish, we can use modified H-moves (Figures \ref{fig:modified-H-move-no-coordinates} and \ref{fig:odified-H-move-BEV-curve11}) to push around these intersection points in both webs until they are in the same configuration, establishing that $W$ and $W^\prime$ are equivalent.

		\subsection{Preparation:  web pictures on the surface}
		\label{ssec:non-elliptic-webs-and-global-pictures-in-the-surface}

 For a web $W$ on $\surf$ in good position with respect to the split ideal triangulation $\splitidealtriang$, the restrictions $W_\biang = W \cap \biang$ and $W_\triang = W \cap \triang$ in the biangles $\biang$ and triangles $\triang$ of $\splitidealtriang$ are essential and rung-less essential local webs, respectively.  By Definitions \ref{def:biangle-local-picture} and \ref{def:triangle-local-picture}, we may consider the corresponding local pictures $\left< W_\biang \right>$ and $\left< W_\triang \right>$, which are in particular immersed multi-curves in the biangle $\biang$ and the holed triangle $\triang^0$, respectively; see Definition \ref{def:multi-curve} and Figures \ref{fig:decomposing-essential-web-bigon} and \ref{fig:disconnected-rung-less-essential-web}.  

\begin{definition}
\label{def:global-picture}
	The \textit{holed surface} $\surf^0$ is the surface $\surf$ minus one open disk per triangle $\triang$ of $\splitidealtriang$.  The \textit{global picture} $\left< W \right>$ corresponding to a web $W$ in good position with respect to $\splitidealtriang$ is the multi-curve on the holed surface $\surf^0$ obtained by gluing together in the obvious way the collection of local pictures $\{ \left< W_\biang \right> \}_{\biang \in \splitidealtriang}$ and $\{ \left< W_\triang \right> \}_{\triang \in \splitidealtriang}$ associated to the biangles $\biang$ and triangles $\triang$ of $\splitidealtriang$, well-defined up to ambient isotopy of $\surf^0$ respecting $\splitidealtriang$.  See Figure \ref{fig:BEV-multicurve-example12}.
\end{definition}

Figure \ref{fig:odified-H-move-BEV-curve11} depicts how a modified H-move between webs $W$ and $W^\prime$ in good position looks when viewed from the perspective of the global pictures $\left< W \right>$ and $\left< W^\prime \right>$; see Figure \ref{fig:modified-H-move-no-coordinates}.  

Note the global picture $\left< W \right>$ has no U-turns on any edge of $\splitidealtriang$, meaning there are no bigons formed between a component $\gamma$ of $\left< W \right>$ and $\splitidealtriang$.  We call this the \textit{no-switchbacks property}.  

\begin{definition}
\label{def:based-global-picture}
A \textit{based multi-curve} $(\Gamma, \{ x_0^j \})$ on the holed surface $\surf^0$ is a multi-curve $\Gamma = \{ \gamma_i \}$ equipped with a base point $x_0^j \in \gamma_j$ for each loop component $\gamma_j$ of $\Gamma$, such that the base points $x_0^j$ do not lie on any edges of the split ideal triangulation $\splitidealtriang$; see Definition \ref{def:multi-curve}.  
\end{definition}

\begin{figure}[htb]
	\centering
	\includegraphics[width=.72\textwidth]{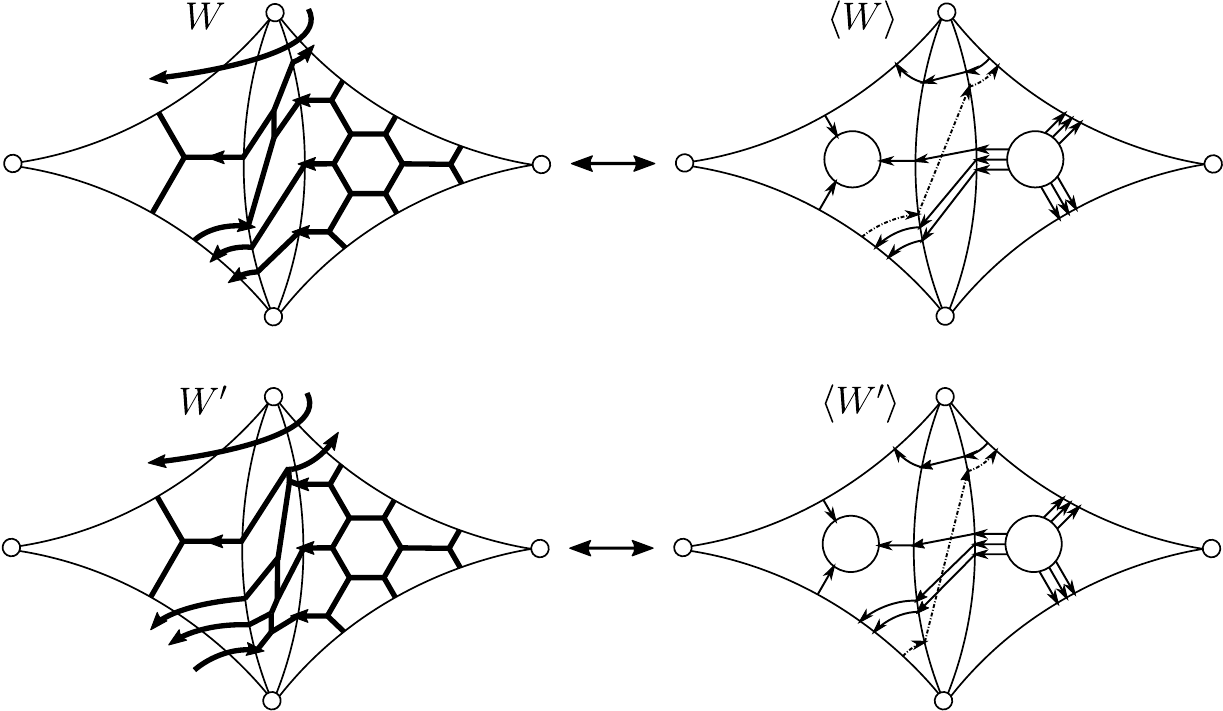}
	\caption{(Parts of) two webs $W$ and $W^\prime$ in good position on the surface, and their corresponding global pictures $\left< W \right>$ and $\left< W^\prime \right>$ on the holed surface.  Note that, over triangles, $W$ and  $W^\prime$ differ by a permutation of corner arcs.}
	\label{fig:BEV-multicurve-example12}
\end{figure}

\begin{figure}[htb]
	\centering
	\includegraphics[scale=.52]{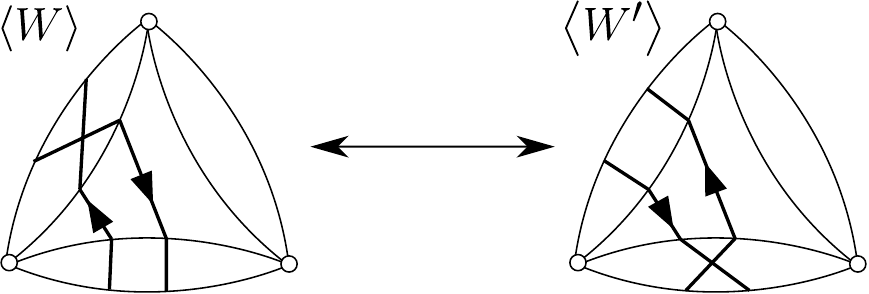}
	\caption{Modified H-move from the perspective of web pictures}
	\label{fig:odified-H-move-BEV-curve11}
\end{figure}

		\subsection{Preparation: sequences}
		\label{ssec:section20.5}

A \textit{convex subset} $I \subset \Z$ of the integers is a subset such that if $n, m \in I$ are integers, then all the integers between $n$ and $m$ are in $I$.
	
	A \textit{sequence} $(a_i)_{i \in I}$ valued in a set $\mathscr{A}$ is a function $I \to \mathscr{A}$, $i \mapsto a_i$, where $I \subset \Z$ is a convex subset of the integers.  
	
Given a sequence $(a_i)_{i \in I}$, a \textit{subsequence} $(a_{i_k})_{k \in K}$ is the sequence $K \to \mathscr{A}$ determined by a convex subset $K \subset \Z$  together with an order-preserving injective function $K \to I$, $k \mapsto i_k$.  
	
Given a sequence $(a_i)_{i \in I}$, a \textit{convex subsequence} $(a_{i_k})_{k \in K}$ is a subsequence such that the image $I^\prime$ of $K$ in $I$ under the function $K \to I$ is a convex subset of $\Z$.  
	
Given two sequences $(a_i)_{i \in I}$ and $(b_j)_{j \in J}$ taking values in the same set, a \textit{common subsequence} $\{ (a_{i_k})_{k \in K}, (b_{j_k})_{k \in K} \}$ is a pair of subsequences having the same indexing set $K$, such that $a_{i_k} = b_{j_k}$ for all $k \in K$. 
	
A \textit{convex common subsequence} $\{ (a_{i_k})_{k \in K}, (b_{j_k})_{k \in K} \}$ is a common subsequence such that both subsequences $(a_{i_k})_{k \in K}$ and $(b_{j_k})_{k \in K}$ are convex.  
	
A \textit{maximal convex common subsequence} $\{ (a_{i_k})_{k \in K}, (b_{j_k})_{k \in K} \}$ is a convex common subsequence, such that there does not exist:  $K \subsetneq K^\prime$ and a convex common subsequence $\{ (a_{i^\prime_k})_{k \in K^\prime}, (b_{j^\prime_k})_{k \in K^\prime} \}$, satisfying  $i^\prime_k = i_k$ and $j^\prime_k = j_k$ for all $k \in K$.

		\subsection{Preparation: edge-sequences and the Fellow-Traveler Lemma }
		\label{ssec:section22.5}

Let $W$ be a web on $\surf$ in good position with respect to $\splitidealtriang$ such that its global picture $(\left< W \right>, \{x_0^j\})$ is based.  
	
	Let $\gamma$ be a loop or arc in $\left< W \right>$.  Associated to the component $\gamma$ is an \textit{edge-sequence} $(E_i)_{i \in I}$ where $E_i$ is an edge of the split ideal triangulation $\splitidealtriang$.  More precisely, the sequence $(E_i)_{i \in I}$ describes the $i$-th edge crossed by $\gamma$ listed in order according to $\gamma$'s orientation.  In the case where $\gamma$ is an arc, we put $I=\{0,1,\dots,n\} \subset \Z$, and the edge-sequence is well-defined.   In the case where $\gamma$ is a loop with base point $x_0$, we put $I = \Z$, and the edge-sequence is well-defined by sending $0$ to the first edge $E_0$ encountered by $\gamma$ after passing $x_0$.  

We also associate an \textit{inverse edge-sequence} $(E^{-1}_i)_{i \in I^{-1}}$ to the inverse curve $\gamma^{-1}$, defined as follows.  In the case of an arc put $I^{-1} = \{-n, \dots, 1, 0\}$, and in the case of a loop put $I^{-1} = \Z$.  Then the inverse edge-sequence is defined by $E^{-1}_{i} = E_{-i}$ for all $i \in I^{-1}$.  

Another name for a loop or arc $\gamma$ in the global picture $\left< W \right>$ is a \textit{traveler}.  Another name for an inverse curve $\gamma^{-1}$ is a \textit{past-traveler}.  The edge-sequence $( E_i)_{i \in I}$ associated to a traveler $\gamma$ is called its \textit{route}, and the edge-sequence $( E^{-1}_i)_{i \in I^{-1}}$ associated to a past-traveler $\gamma^{-1}$ is called its \textit{past-route}; see Figure \ref{fig:foute-and-past-route}.  Two travelers $\gamma$ in $\left< W \right>$ and $\gamma^\prime$ in $\left< W^\prime \right>$ are called \textit{fellow-travelers} if they have the same routes $(E_i)_{i \in I} = (E^\prime_i)_{i \in I^\prime}$, $I = I^\prime$.  In particular, if $\gamma$ is a loop (resp. arc), then $\gamma^\prime$ is also a loop (resp. arc of the same length).  

The following statement is the key to proving the main lemma.  

\begin{figure}[b]
	\centering
	\includegraphics[scale=.49]{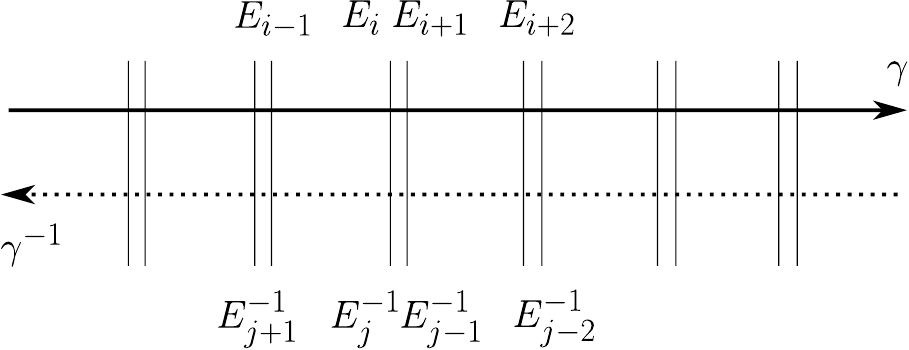}
	\caption{Route and past-route}
	\label{fig:foute-and-past-route}
\end{figure}

\begin{lemma}[Fellow-Traveler Lemma]
\label{lem:life-neighbors-lemma}
	Fix compatible local webs $\{W_\triang \}_{\triang \in \splitidealtriang}$ and $\{W^\prime_\triang\}_{\triang \in \splitidealtriang}$ in the $\webbasis{\triang}$ satisfying hypothesis (1) of Main Lemma {\upshape\ref{lem:main-lemma}}, and let $W$ and $W^\prime$ be the induced global webs obtained by the ladder gluing construction.  Then, there exists a natural one-to-one correspondence
\begin{equation*}
	\gamma \longleftrightarrow \gamma^\prime = \varphi(\gamma)
\end{equation*}
between the collection of travelers $\gamma$ in the global picture $\left< W \right>$ and the collection of travelers $\gamma^\prime = \varphi(\gamma)$ in $\left< W^\prime \right>$, and there exists a choice of base points $ x_0$ and $x^{\prime}_0$ for the loops $\gamma$ and $\gamma^\prime$ in  $\left< W \right>$ and $\left< W^\prime \right>$, respectively, such that $\gamma$ and $\gamma^\prime=\varphi(\gamma)$ are fellow-travelers for all travelers~$\gamma$.  
\end{lemma}

For an example of the Fellow-Traveler Lemma on the once punctured torus, see Figure \ref{fig:example-of-fellow-traveler-lemma}.  

\begin{figure}[t]
	\centering
	\includegraphics[width=\textwidth]{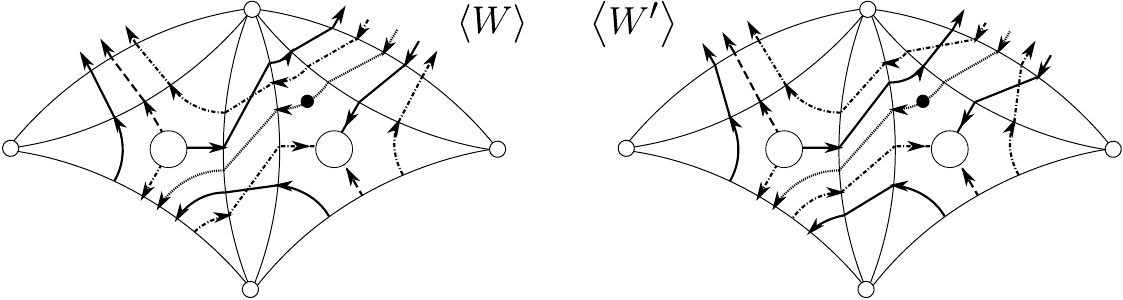}
	\caption{Fellow-Traveler Lemma}
	\label{fig:example-of-fellow-traveler-lemma}
\end{figure}

\begin{proof}[Proof of Lemma \ref{lem:life-neighbors-lemma}]
	Let $E$ be an edge of $\splitidealtriang$.  This is associated to a unique triangle $\triang$ of $\splitidealtriang$ containing $E$ in its boundary.  Let $S^{(E)\mathrm{out}}=(s_i^{(E)\mathrm{out}})_{i=1,\dots,n_E^\mathrm{out}}$ (resp. $S^{\prime(E)\mathrm{out}}=(s_i^{\prime(E)\mathrm{out}})_{i=1,\dots,n_E^{\prime\mathrm{out}}}$) denote the sequence of out-strands of the global picture $\left< W \right>$ (resp. $\left< W^\prime \right>$) lying on the edge $E$, ordered, say, from left to right as viewed from $\triang$.   By hypothesis (1) of the main lemma,  $n_E^\mathrm{out}=n_E^{\prime\mathrm{out}}$.  Let $\gamma_i^{(E)}$ denote the unique traveler in $\left< W \right>$ containing the strand $s_i^{(E)\mathrm{out}}$.  Similarly, define travelers $\gamma_i^{\prime(E)}$ with respect to $\left< W^\prime \right>$.  The mapping $\varphi$ is defined by
\begin{equation*}
	\varphi \left( \gamma_i^{(E)} \right) = \gamma_i^{\prime(E)}
	\quad\quad
	\left( i=1,2,\dots,n_E^\mathrm{out}=n_E^{\prime\mathrm{out}} \right).
\end{equation*}
Note every traveler $\gamma$ in $\left< W \right>$ (resp. $\gamma^\prime$ in $\left< W^\prime \right>$) is of the form $\gamma_i^{(E)}$ (resp. $\gamma_i^{\prime(E)}$) for some $E$.  

To establish that $\varphi$ is well-defined, we show that $\gamma_{i_1}^{(E_1)} = \gamma_{i_2}^{(E_2)}$ implies $\gamma_{i_1}^{\prime(E_1)} = \gamma_{i_2}^{\prime(E_2)}$.  This property follows immediately from:

\begin{claim}
\label{claim:fellow-traveler}
	For some $k \in \{ 1, 2, \dots, n_E^\mathrm{out}=n_E^{\prime\mathrm{out}} \}$, let $s_k^{(E)\mathrm{out}} \in S^{(E)\mathrm{out}}$ and $s_k^{\prime(E)\mathrm{out}} \in S^{\prime(E)\mathrm{out}}$ be out-strands of $\left< W \right>$ and $\left< W^\prime \right>$, respectively, lying on an edge $E$ of a triangle $\triang$ of $\splitidealtriang$.  Note that each of these strands, according to its orientation, enters via the edge $E$ into a biangle $\biang$, exits via an edge $E_2$ into a triangle $\triang_2$, and then either
\begin{enumerate}
	\item  turns left in $\triang_2$, ending as a strand $s$ or $s^\prime$, respectively, lying on an edge $E_3$;
	\item turns right in $\triang_2$, ending as a strand $s$ or $s^\prime$, respectively, lying on an edge $E_3$;
	\item  terminates in a honeycomb $H_n$.
\end{enumerate}
The claim is that if the forward motion of the strand $s_k^{(E)\mathrm{out}}$ is described by item $(i)$ above for $i \in \{ 1, 2, 3 \}$, then the forward motion of the strand $s_k^{\prime(E)\mathrm{out}}$ is also described by item $(i)$.  Consequently, in cases (1) or (2), there exists some $k_3 \in \{ 1, 2, \dots, n_{E_3}^\mathrm{out}=n_{E_3}^{\prime\mathrm{out}} \}$ such that 
\begin{equation*}
	s = s_{k_3}^{(E_3)\mathrm{out}} \in S^{(E_3)\mathrm{out}}
	\quad\quad  \text{and}  \quad\quad
	s^\prime = s_{k_3}^{\prime(E_3)\mathrm{out}} \in S^{\prime(E_3)\mathrm{out}}.
\end{equation*}
\end{claim}

The claim is true since, by hypothesis, on each corner of each triangle, $\left< W \right>$ and $\left< W^\prime \right>$ have the same number of clockwise-oriented (resp. counterclockwise-oriented) corner arcs, together with the fact that only oppositely-oriented arcs cross in the biangles; see Figure \ref{fig:fellow-traveler-lemma-proof}.  

Having established that $\varphi$ is well-defined, it follows by the definition that $\varphi$ is a bijection.  Another consequence of Claim \ref{claim:fellow-traveler} is that if $\gamma$ is an arc, then $\gamma^\prime = \varphi(\gamma)$ is an arc such that $\gamma$ and $\gamma^\prime$ are fellow-travelers.  Also, if $\gamma=\gamma_i^{(E)}$ is a loop, then $\gamma^\prime = \varphi(\gamma)=\gamma_i^{\prime(E)}$ is a loop.  Choosing base points $x_0$ and $x_0^\prime$ on the out-strands $s_i^{(E)\mathrm{out}}$ and $s_i^{\prime (E)\mathrm{out}}$, respectively, just before, say, the strands cross the edge $E$ makes the loops $\gamma$ and $\gamma^\prime$ into fellow-travelers.  
\end{proof}

\begin{figure}[htb]
	\centering
	\includegraphics[width=.49\textwidth]{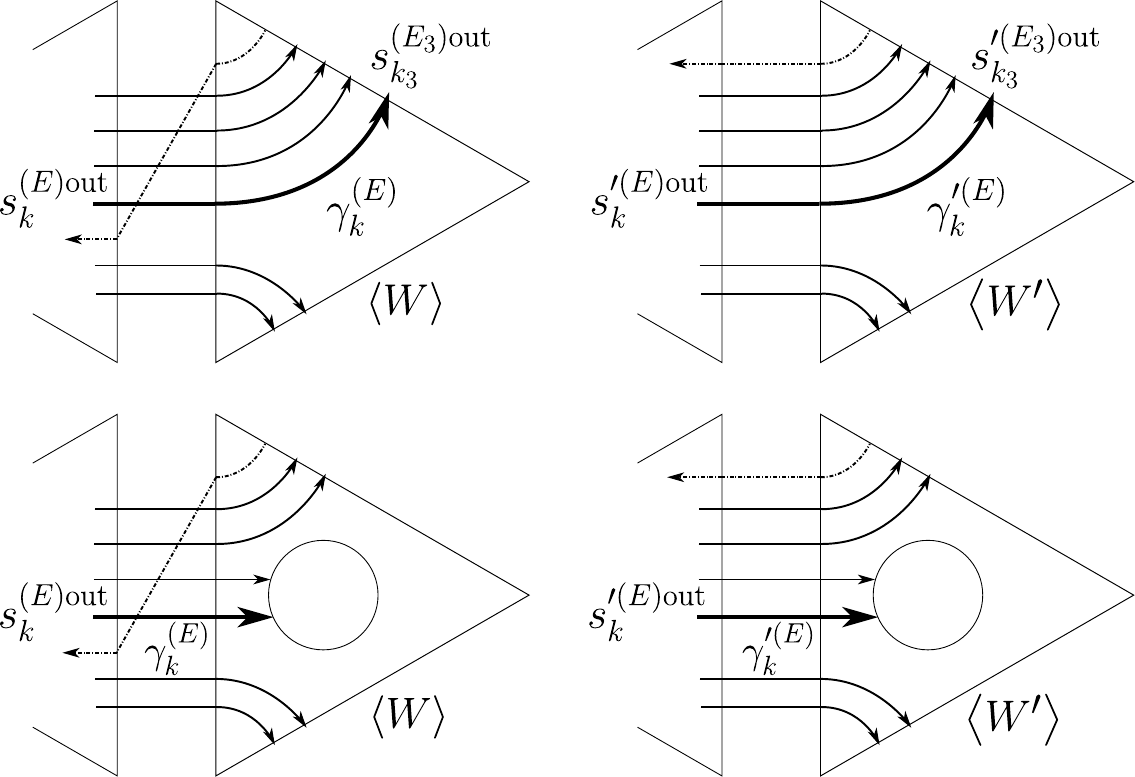}
	\caption{Cases (1) (top) and (3) (bottom) in Claim \ref{claim:fellow-traveler}}
	\label{fig:fellow-traveler-lemma-proof}
\end{figure}

		\subsection{Preparation: shared-routes}
		\label{ssec:section25}

As in the previous sub-section, let $W$ be a web on $\surf$ in good position with respect to $\splitidealtriang$ such that its global picture $(\left< W \right>, \{x_0^j\})$ is based.  
		
Let $\gamma$ be a traveler in $\left< W \right>$ having route $(E_i)_{i \in I}$.  For some $i \in I$ indexing an edge $E_i$, by definition of the route there is a corresponding point $y_i$ of $\gamma$ lying on $E_i$.  Consider the associated \textit{segment} $\overline{\gamma}_{i}$ of $\gamma$ lying between the points $y_i$ and $y_{i+1}$.  Similarly, define segments $\overline{(\gamma^{-1})}_i$ associated to the past-traveler $\gamma^{-1}$ with respect to its past-route $(E^{-1}_i)_{i \in I^{-1}}$.  

\begin{definition}
	Let $\gamma_1$, $\gamma_2$ be travelers in $\left< W \right>$ and $\gamma_1^{-1}$, $\gamma_2^{-1}$ the corresponding past-travelers, with routes $(E^1_i)_{i \in I}$ and $(E^2_j)_{j \in J}$ and past-routes $((E^1)^{-1}_i)_{ \in I^{-1}}$ and $((E^2)^{-1}_j)_{ \in J^{-1}}$.  

An \textit{oppositely-oriented shared-route}, or just \textit{shared-route}, $SR$ for the ordered pair $(\gamma_1, \gamma_2)$ of travelers is a maximal convex common subsequence (\S \ref{ssec:section20.5}) $SR= \{  (E^1_{i_k})_{k \in K}, ((E^2)^{-1}_{j_k})_{k \in K} \}$ for the route $(E^1_i)_{i \in I}$ of $\gamma_1$ and the past-route $((E^2)^{-1}_j)_{ \in J^{-1}}$ of $\gamma_2^{-1}$.
	
A shared-route is \textit{open} (resp. \textit{closed}) if its domain $K$ is not equal to (resp. equal to) $\Z$.  
	
A shared-route is \textit{crossing} if there exists an index $k \in K$ such that the associated segments $\overline{(\gamma_1)}_{i_k}$ and $\overline{(\gamma_2)^{-1}_{j_k}}$ intersect, say at a point $p_k$.  We call $p_k$ an \textit{intersection point} of the crossing shared-route.  Note that an intersection point must lie inside a biangle $\biang$ of $\splitidealtriang$.  A shared-route is \textit{non-crossing} if it has no intersection points.  
\end{definition}

\begin{figure}[htb]
	\centering
	\includegraphics[scale=.46]{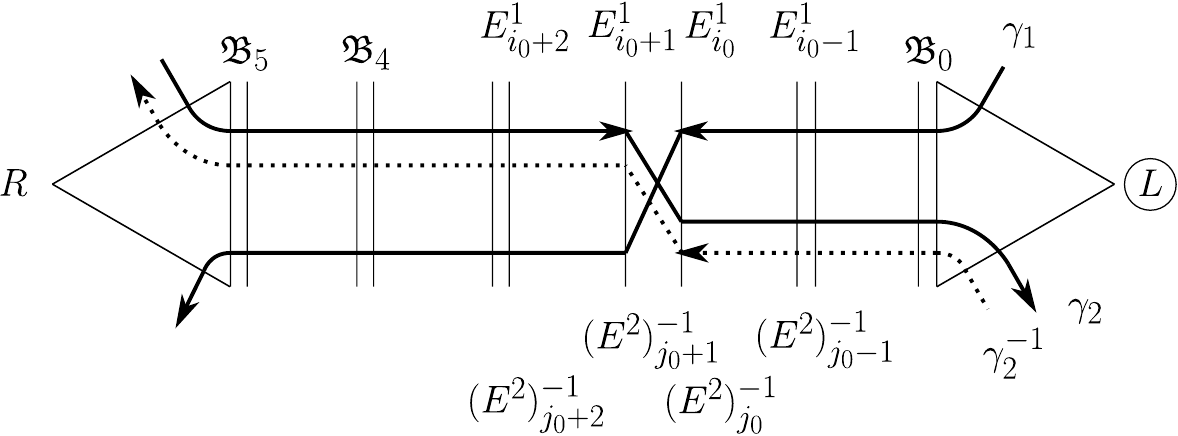}
	\caption{Crossing shared-route}
	\label{fig:intersecting-crossing-journey}
\end{figure}

\begin{figure}[htb]
     \centering
     \begin{subfigure}{0.49\textwidth}
         \centering
         \includegraphics[scale=.4]{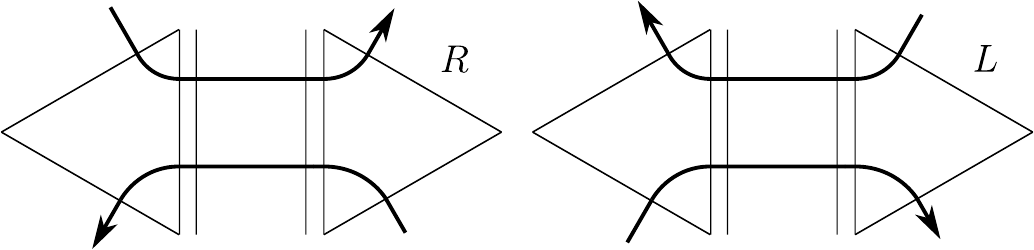}
         \caption{Open}
         \label{subfig:open-non-crossing-shared-route}
     \end{subfigure}     
\hfill
     \begin{subfigure}{0.49\textwidth}
         \centering
         \includegraphics[scale=.45]{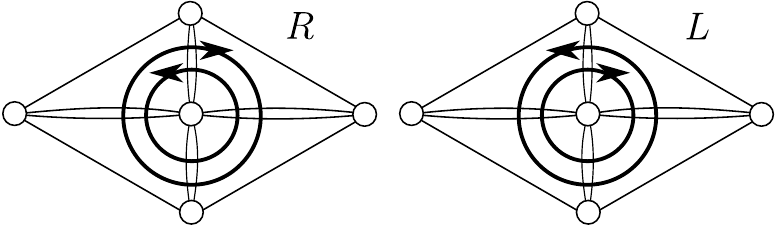}
         \caption{Closed}
         \label{subfig:closed-non-crossing-shared-route}
     \end{subfigure}
        \caption{Non-crossing shared-routes}
        \label{fig:non-crossing-shared-routes}
\end{figure}

For some examples, see Figures \ref{fig:intersecting-crossing-journey} and \ref{fig:non-crossing-shared-routes}.  Our pictures for shared-routes are only schematics, since the actual shared-routes on $\surf^0$ might cross the same edge multiple times.  That is, there might exist $k \neq k^\prime$ such that $E^1_{i_k} = (E^2)^{-1}_{j_k} = E^1_{i_{k^\prime}} = (E^2)^{-1}_{j_{k^\prime}}$.  Alternatively, one could think of these pictures at the level of the universal cover $\widetilde{\surf^0}$.  Note that  travelers in open shared-routes may end in honeycombs (Figure \ref{fig:fellow-traveler-lemma-proof}), but this will not affect our arguments.  

\begin{lemma}
\label{lem:unique-intersection-point}
	Assume in addition that $W$ is non-elliptic.   Then any shared-route $SR$ has at most one intersection point $p$.  In particular, a crossing shared-route is necessarily open.  
\end{lemma}

\begin{proof}
	The second statement follows from the first since otherwise the oriented holed surface $\surf^0$ would contain a  M\"{o}bius strip.  

	Suppose, for an ordered pair $(\gamma_1, \gamma_2)$ of travelers, there were a crossing shared-route $ \{  (E^1_{i_k})_{k \in K}, ((E^2)^{-1}_{j_k})_{k \in K} \}$ that has more than one intersection point.  There are only finitely-many intersection points, denoted $p_{k_1}, p_{k_2}, \dots, p_{k_m}$ with $k_i < k_{i+1}$.  The intersection points $p_{k_1}$ and $p_{k_{2}}$ form the tips of an \textit{immersed bigon} $B$, which we formalize as the convex common subsequence $ B = \{   (E^1_{i_k})_{k_1 \leq k \leq k_2+1}, ((E^2)^{-1}_{j_k})_{k_1 \leq k \leq k_2+1} \}$; see the bottom of Figure \ref{fig:no-immersed-bigons}.  Alternatively, we think of $B$ as bounded by the segments of $\gamma_1$ and $\gamma_2$ between $p_{k_1}$ and $p_{k_2}$.  
	
	\begin{figure}[b]
	\centering
	\includegraphics[scale=.7]{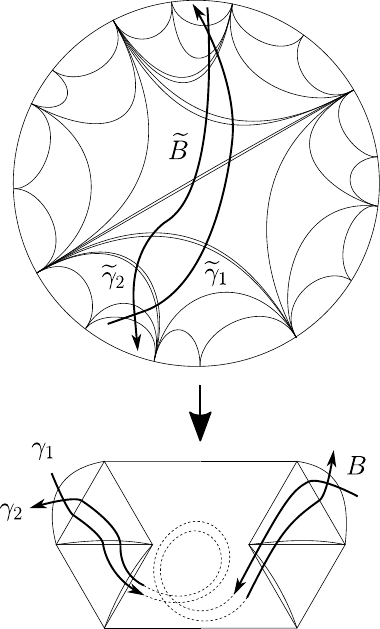}
	\caption{Immersed bigons do not exist: 1 of 2}
	\label{fig:no-immersed-bigons}	
\end{figure}

Let $\pi$ be the projection map from the universal cover $\widetilde{\surf^0}$ to the holed surface $\surf^0$.  Equip $\widetilde{\surf^0}$ with the lifted split ideal triangulation $\widetilde{\splitidealtriang} = \pi^{-1}(\splitidealtriang)$.  For a traveler $\gamma$, consider one of its lifts $\widetilde{\gamma}$ in  $\widetilde{\surf^0}$.  By the no-switchbacks property (\S \ref{ssec:non-elliptic-webs-and-global-pictures-in-the-surface}), and the fact that the dual graph of $\widetilde{\splitidealtriang}$ in $\widetilde{\surf^0}$ is a tree, the lifted curve $\widetilde{\gamma}$ does not cross the same edge $\widetilde{E}$ in the universal cover $\widetilde{\surf^0}$ more than once.  Therefore, the immersed bigon $B$ lifts to an  embedded topological bigon $\widetilde{B}$ in $\widetilde{\surf^0}$, bounded by segments of lifts $\widetilde{\gamma}_1$ and $\widetilde{\gamma_2}$ of the curves $\gamma_1$ and $\gamma_2$; see Figure \ref{fig:no-immersed-bigons}.  
	
The preimage $\widetilde{W} = \pi^{-1}(W)$ of the web $W$ is an (infinite) web in $\widetilde{\surf^0}$.  Moreover, $\widetilde{W}$ is in good position with respect to $\widetilde{\splitidealtriang}$.  Since $W$ is non-elliptic, so is $\widetilde{W}$ (compare the proof of Lemma \ref{lem:fixing-an-elliptic-web}).    Let $\left< \widetilde{W} \right>$ be the global picture associated to $\widetilde{W}$.  Note that the lifted curves $\widetilde{\gamma}_1$ and $\widetilde{\gamma}_2$ are in $\left< \widetilde{W} \right>$.  Observe that it is possible for $\mathrm{int}(\widetilde{B}) \cap \left< \widetilde{W} \right> \neq \emptyset$ to be non-empty; see Figure \ref{fig:no-immersed-bigons-2}.  However, by the no-switchbacks property, there are no closed curves of $\left< \widetilde{W} \right>$ in this interior.  

	The rest of the proof is similar to the proof of Proposition \ref{prop:ladder-webs}; see Figure \ref{subfig:internal-squares}.  Here, the web orientation is important.  Specifically, since only (locally) oppositely-oriented (with respect to biangles) curves in the global picture $\left< \widetilde{W} \right>$ can intersect, it follows by the no-switchbacks property that if a curve $\widetilde{\gamma}$ enters the embedded bigon $\widetilde{B}$ via a boundary edge $\widetilde{E}$, then $\widetilde{\gamma}$ must leave through $\widetilde{E}$ as well.  Consequently, there exists an inner-most embedded bigon $\widetilde{B}^\prime \subset \widetilde{B}$ whose interior does not intersect $\left< \widetilde{W} \right>$; see Figure \ref{fig:no-immersed-bigons-2}.  But then $\widetilde{B}^\prime$ corresponds to a square-face $\widetilde{D}$ in the lifted non-elliptic web $\widetilde{W}$, which is a contradiction.  
\end{proof}

\begin{figure}[t]
	\centering
	\includegraphics[scale=.9]{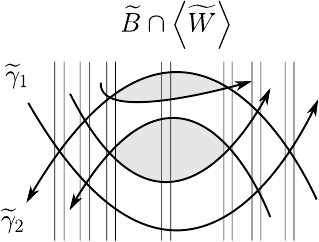}
	\caption{Immersed bigons do not exist: 2 of 2}
	\label{fig:no-immersed-bigons-2}	
\end{figure}

\begin{lemma}
\label{lem:closed-non-crossing-shared-routes-are-embedded}
	If the web $W$ is non-elliptic, then there are no intersection points of $\left< W \right>$ along any closed non-crossing shared-route (as opposed to open non-crossing shared-routes).  In particular, each closed non-crossing shared-route $SR$ is embedded, namely its  travelers $\gamma_1$ and $\gamma_2$ bound an embedded annulus $A \subset \surf$; see Figure {\upshape\ref{subfig:closed-non-crossing-shared-route}}.  
\end{lemma}

\begin{proof}
	If there were an intersection point of $\left< W \right>$ along either traveler, then an argument similar to that depicted in Figure \ref{fig:no-immersed-bigons-2} implies there would exist an immersed bigon in $\left< W \right>$.
\end{proof}

\begin{definition}
\label{def:source-end}
Consider an open shared-route $SR$ for an ordered pair $(\gamma_1, \gamma_2)$ of travelers in $\left< W \right>$.
We say that the \textit{source-end} $\E$ of the open shared-route $SR$ is the unique end $\E$ of $SR$ such that the traveler $\gamma_1$ enters  the shared-route $SR$ through the end  $\E$.

Assuming $W$ is non-elliptic, we say that the unique intersection point $p$ in a crossing shared-route $SR$, which is necessarily open by Lemma \ref{lem:unique-intersection-point}, \textit{lies in the $i$-th shared-route-biangle $\biang_i$}, denoted $p \in_{SR} \biang_i$, $i \geq 0$, if $\gamma_1$ crosses $\gamma_2^{-1}$ (at the point $p$) inside the $i$-th biangle through which $\gamma_1$ travels after entering $SR$ through the source-end $\E$.
\end{definition}

For example, in Figure \ref{fig:intersecting-crossing-journey}, the source-end $\E$ of $SR$ is the end labeled $L$.  

Also, in Figure \ref{fig:intersecting-crossing-journey}, $p$ is in the shared-route-biangle $p \in_{SR} \biang_2$.  Note that there is a unique index $i$ such that $p \in_{SR} \biang_i$.  This definition is specially designed to circumvent the situation where $\biang_i$ and $\biang_j$ represent the same biangle $\biang$ on the surface for different indices $i \neq j$.  For example, in Figure \ref{fig:intersecting-crossing-journey}, even if, say, $\biang_5$ represented the same biangle $\biang$ as $\biang_2$, we would say $p \in_{SR} \biang_2$ and $p \notin_{SR} \biang_5$.  Alternatively, one could think of this distinction at the level of the universal cover.

		\subsection{Preparation:  oriented shared-routes}
		\label{ssec:oriented-highways}

As previously, let $W$ be a web on $\surf$ in good position with respect to $\splitidealtriang$ such that its global picture $(\left< W \right>, \{x_0^j\})$ is based.

\begin{definition}
\label{def:orientationsofnoncrossingsharedroutes}
	We say that a non-crossing shared-route $SR$ for an ordered pair $(\gamma_1, \gamma_2)$ of travelers in $\left< W \right>$ is \textit{left-oriented} (resp. \textit{right-oriented}) if for either of the travelers $\gamma_1$ or $\gamma_2$, call it $\gamma$, the other traveler appears on the left (resp. right) of $\gamma$ with respect to $\gamma$'s orientation; see Figure \ref{fig:non-crossing-shared-routes}.  
	
	The web $W$ is \textit{closed-left-oriented} (resp. \textit{closed-right-oriented}) if all of $\left< W \right>$'s closed non-crossing shared-routes are left-oriented (resp. right-oriented); see Figure \ref{subfig:closed-non-crossing-shared-route}.
\end{definition}

	Note, by Lemma \ref{lem:closed-non-crossing-shared-routes-are-embedded}, a non-elliptic web $W$ can always be replaced with a closed-left-oriented or closed-right-oriented non-elliptic web by performing global parallel-moves (Definition \ref{def:parallel-equivalent-webs}); see Figure~\ref{fig:parallel-move}.  

	We also want to define a notion of orientation for crossing shared-routes.  Unlike for non-crossing shared-routes, this will depend  on the ordering of the pair $(\gamma_1, \gamma_2)$.  Since we will be dealing with non-elliptic webs, by Lemma \ref{lem:unique-intersection-point} it suffices to think about open shared-routes.
	
\begin{definition}
\label{def:orientationofcrossingsharedroutefornonelliptic}

We say that an end $\E$ of an open shared-route $SR$ is \textit{left-oriented} or \textit{right-oriented} in the same way as in Definition \ref{def:orientationsofnoncrossingsharedroutes} for non-crossing shared routes.  
	
Assuming $W$ is non-elliptic, a crossing shared-route $SR$ for an ordered pair $(\gamma_1, \gamma_2)$ of travelers in $\left< W \right>$, which is necessarily open by Lemma \ref{lem:unique-intersection-point}, 
 is \textit{left-oriented} (resp. \textit{right-oriented}) if its source-end $\E$ (Definition \ref{def:source-end}) is left-oriented (resp. right-oriented).  
	
\end{definition}
	
	For example, the crossing shared-route shown in Figure \ref{fig:intersecting-crossing-journey} is left-oriented.

		\subsection{Proof of the main lemma:  intersection points}
		\label{ssec:intersection-points}

We now begin the formal proof of Main Lemma \ref{lem:main-lemma}.  	Fix  local webs $\{W_\triang \}_{\triang \in \splitidealtriang}$ and $\{W^\prime_\triang\}_{\triang \in \splitidealtriang}$ in the $\webbasis{\triang}$ satisfying the hypotheses of the main lemma, and let $W$ and $W^\prime$ be the induced non-elliptic global webs obtained by the ladder gluing construction.  By applying global parallel-moves, we may assume that both $W$ and $W^\prime$ are closed-left-oriented, say (Definition \ref{def:orientationsofnoncrossingsharedroutes}).  Assume that the global pictures $(\left<W\right>, \{x_0^j\})$ and $(\left<W^\prime\right>, \{ x_0^{\prime j} \})$ are based, and that the base points $x_0^j$ and $x_0^{\prime j}$ satisfy the conclusion of the Fellow-Traveler Lemma \ref{lem:life-neighbors-lemma}.  Throughout, for each traveler $\gamma$ in $\left< W \right>$ we denote by $\gamma^\prime$ the corresponding traveler in $\left< W^\prime \right>$ as provided by the Fellow-Traveler Lemma.  

Let $\mathscr{P}$ (resp. $\mathscr{P}^\prime$) denote the set of intersection points $p$ of all travelers in $\left< W \right>$ (resp. $\left< W^\prime \right>$).

\begin{corollary}
\label{cor:same-number-of-intersection-points}
	There is a natural bijection $\varphi : \mathscr{P} \overset{\sim}{\to} \mathscr{P}^\prime$.  We write $p^\prime = \varphi(p)$.  
\end{corollary}

For the proof, we will need the following notion.

\begin{definition}
\label{def:generated-crossing-shared-route}
	Let $p \in \mathscr{P}$.  We define \textit{the left-oriented crossing shared-route generated by $p$}, denoted $SR(p)$, to be the unique left-oriented crossing shared-route (Definition \ref{def:orientationofcrossingsharedroutefornonelliptic}) in $\left< W \right>$ whose intersection point is $p$.  Note, in particular, that the left-orientation condition determines the order $(\gamma_1, \gamma_2)$ of the involved travelers.  (Technically speaking, we choose $K$ starting at $0$, and then
the shared-route $SR(p) = \{  (E^1_{i_k})_{k \in K}, ((E^2)^{-1}_{j_k})_{k \in K} \}$ is only uniquely determined after choosing the two indices $i_0$ and $j_0$ assigned by $0 \in K$; this ambiguity only occurs when the shared-route has--part of--a loop traveler.)
\end{definition}	

\begin{proof}[Proof of Corollary \ref{cor:same-number-of-intersection-points}]
	Consider the left-oriented crossing shared-route $SR(p)$ in $\left< W \right>$ with travelers $(\gamma_1, \gamma_2)$ generated by the intersection point $p$.  By the Fellow-Traveler Lemma, there is a corresponding shared-route $SR^\prime$ in $\left< W^\prime \right>$ with the travelers $(\gamma^\prime_1, \gamma^\prime_2)$, which must also be open; see Figures \ref{fig:intersecting-crossing-journey} and \ref{subfig:open-non-crossing-shared-route}.  Moreover, the ends $\E^\prime$ of $SR^\prime$ have orientations (Definition \ref{def:orientationofcrossingsharedroutefornonelliptic}) matching those of the ends $\E$ of $SR(p)$.  It follows that $SR^\prime$ is crossing.  Its unique intersection point $p^\prime$ is the desired image of $p$; see Figure \ref{fig:corresponding-intersection-points}.  (Note for later that since $SR(p)$ is left-oriented, so is $SR^\prime$, thus $SR^\prime = SR^\prime(p^\prime)$.)  
\end{proof}

Recall that a crossing shared-route $SR$ for the ordered pair $(\gamma_1, \gamma_2)$ comes with an ordering of the shared-route-biangles $\biang_i$ appearing along $\gamma_1$'s route, starting from the source-end $\E$; see Definition \ref{def:source-end}.  If $p$ and $p^\prime$ are intersection points as in Corollary \ref{cor:same-number-of-intersection-points} and its proof, then the left-oriented crossing shared-routes $SR(p)$ and $SR^\prime(p^\prime)$ have the same associated sequence of shared-route-biangles $\biang_i$.  However, if $p \in_{SR(p)} \biang_i$ and $p^\prime \in_{SR^\prime(p^\prime)} \biang_j$ (see again Definition \ref{def:source-end}), it need not be true that $i = j$; see Figure \ref{fig:corresponding-intersection-points}. 

\begin{definition}
\label{def:lying-in-the-same-biangle}
	We say that two corresponding intersection points $p$ and $p^\prime$, as in Corollary \ref{cor:same-number-of-intersection-points}, \textit{lie in the same shared-route-biangle} if there is an index $i$ such that $p \in_{SR(p)} \biang_i \ni_{SR^\prime(p^\prime)} p^\prime$, where the sequence of shared-route-biangles $\{ \biang_i \}$ is defined with respect to the left-oriented crossing shared-routes $SR(p)$ and $SR^\prime(p^\prime)$ generated by $p$ and $p^\prime$, respectively.  
\end{definition}

For example, in Figure \ref{fig:corresponding-intersection-points}, even if it were true that $\biang_0$ and $\biang_2$ represented the same biangle $\biang$ on the surface, we would not say that $p$ and $p^\prime$ lie in the same shared-route-biangle.  

\begin{lemma}
\label{lem:moving-into-same-bigon}
	There is a sequence of modified H-moves (Figure {\upshape\ref{fig:odified-H-move-BEV-curve11}}) applicable to the web $W$ and  a sequence of modified H-moves applicable to $W^\prime$, after which the bijection $\mathscr{P} \leftrightarrow \mathscr{P}^\prime$ from Corollary {\upshape\ref{cor:same-number-of-intersection-points}} satisfies the property that each intersection point $p$ in the global picture $\left< W \right>$ and its corresponding intersection point $p^\prime$ in $\left< W^\prime \right>$ lie in the same shared-route-biangle $\biang_i$.
\end{lemma}
	
Before giving a proof (\S \ref{ssec:section26}), we reduce the proof of the main lemma to that of Lemma \ref{lem:moving-into-same-bigon}.

\begin{figure}[htb]
	\centering
	\includegraphics[scale=.41]{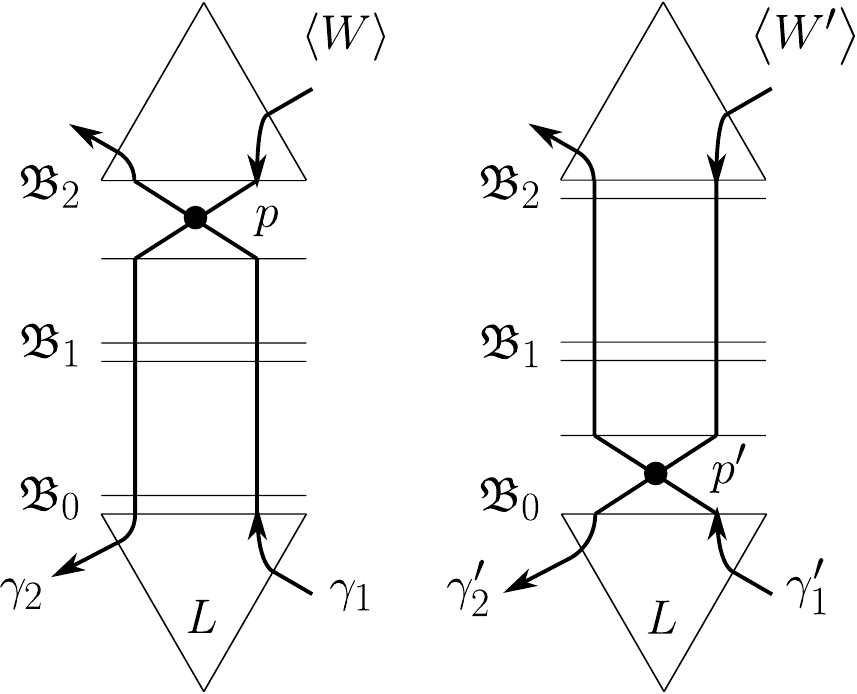}
	\caption{Natural one-to-one correspondence between intersection points}
	\label{fig:corresponding-intersection-points}
\end{figure}

		\subsection{Proof of the main lemma:  finishing the argument}
		\label{ssec:finishing-the-argument}

Assuming corresponding intersection points lie in the same shared-route-biangle, we claim that we are done, $W = W^\prime$.  
	
By the proof of the Fellow-Traveler Lemma, not only is there a natural bijection of travelers $\gamma \leftrightarrow \gamma^\prime$, moreover for each edge $E$ of $\splitidealtriang$ there is a natural bijection of oriented strands $s \leftrightarrow s^\prime$ of $\left< W \right>$ and $\left< W^\prime \right>$, respectively, on $E$. Namely, the $k$-th out-strand (resp. in-strand) $s$, measured from left to right, say, with respect to $\triang$, is matched with the $k$-th out-strand (resp. in-strand) $s^\prime$. 
This satisfies that $s$ lies in $\gamma$ if and only if $s^\prime$ lies in $\gamma^\prime$.
	
Fix an edge $E$ adjacent to a triangle $\triang$.  Let $S = (s_i)$ (resp. $S^\prime=(s^*_i)$) be the full sequence of oriented strands for $\left< W \right>$ (resp. $\left< W^\prime \right>$) on the edge $E$, measured from left to right.  In particular, both in- and out-strands occur in $S$ (resp. $S^\prime$). 

\begin{lemma}
\label{lem:same-edge-sequences}
	Assuming corresponding intersection points lie in the same shared-route-biangle, we have that  $S = S^\prime$, for every edge $E$ of $\splitidealtriang$; see Definition {\upshape\ref{def:lying-in-the-same-biangle}}.  (That is, $s^*_i=s^\prime_i$ for all $i$.)
\end{lemma}

\begin{proof}
	It suffices to prove the following statement.
\begin{claim}
\label{claim:last-argument-claim}
	If $s^\mathrm{out}$ is an out-strand of $S$, and if $s^\mathrm{in}$ is an in-strand of $S$, then
\begin{equation*}
	s^\mathrm{out} \text{ lies to the left of } s^\mathrm{in}
	\quad  \Longleftrightarrow  \quad
	s^{\prime \mathrm{out}} \text{ lies to the left of } s^{\prime \mathrm{in}}.
\end{equation*}
\end{claim}

\begin{figure}[t]
	\centering
	\includegraphics[scale=.43]{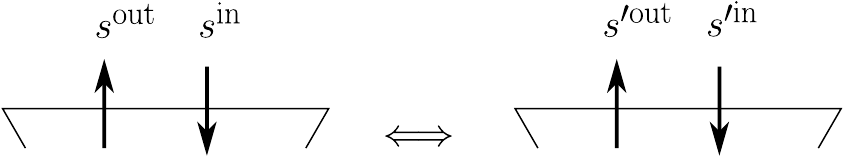}
	\caption{Identical oriented strand-sequences on each edge $E$}
	\label{fig:same-edge-sequences}
\end{figure}

See Figure \ref{fig:same-edge-sequences}.  To prove the forward direction of the claim, suppose otherwise, that is suppose $s^{\prime \mathrm{out}}$ lies to the right of $s^{\prime \mathrm{in}}$.  Let $SR$ (resp. $SR^\prime$) be a shared-route containing $s^\mathrm{out}$ and $s^\mathrm{in}$ (resp. $s^{\prime\mathrm{out}}$ and $s^{\prime\mathrm{in}}$) (there are two possibilities for each, determined by the order of the pair of involved travelers).  By the Fellow-Traveler Lemma, $SR$ is crossing (resp. open/closed non-crossing) if and only if $SR^\prime$ is crossing (resp. open/closed non-crossing).  

Suppose $SR$ and $SR^\prime$ are crossing.  Then we may assume that $SR=SR(p)$ and $SR^\prime=SR^\prime(p^\prime)$ have been chosen as the left-oriented crossing shared-routes generated by their unique intersection points $p$ and $p^\prime$, respectively; see Definition \ref{def:generated-crossing-shared-route}.  By hypothesis, $p$ and $p^\prime$ lie in the same shared-route-biangle, call it $\biang_i$, that is $p \in_{SR(p)} \biang_i \ni_{SR^\prime(p^\prime)} p^\prime$.  
Letting $\biang_j$ denote the shared-route-biangle containing the edge $E$ with the strands $s^\mathrm{out}$ and $s^\mathrm{in}$ (and $s^{\prime\mathrm{in}}$ and $s^{\prime\mathrm{out}}$), let us say that \textit{the strands are on the close (resp. far) side} if they are on the first (resp. last) edge of  $\biang_j$ hit while traveling from the source-end of the shared-route.  Similarly, it makes sense to say that \textit{the crossing comes before (resp. after) the strands} with respect to the source-end.  Our first observation is that since the source-end is left-oriented, and since $s^\mathrm{out}$ lies to the left of $s^\mathrm{in}$, it cannot be true that the crossing comes after the strands.  So let Case 1 (resp. Case 2) be the case that the crossing comes before the strands and that the strands are on the close (resp. far) side.  Since $s^{\prime\mathrm{in}}$ lies to the left of $s^{\prime\mathrm{out}}$ (by contradiction hypothesis), both Case 1 and Case 2 lead to a contradiction, namely that both ends of the crossing shared-route $SR^\prime$ are left-oriented (Definition \ref{def:orientationofcrossingsharedroutefornonelliptic}); see Figure \ref{fig:proof-by-contradiction-last-argument}.

\begin{figure}[htb]
	\centering
	\begin{subfigure}{.49\textwidth}
		\centering
		\includegraphics[scale=.35]{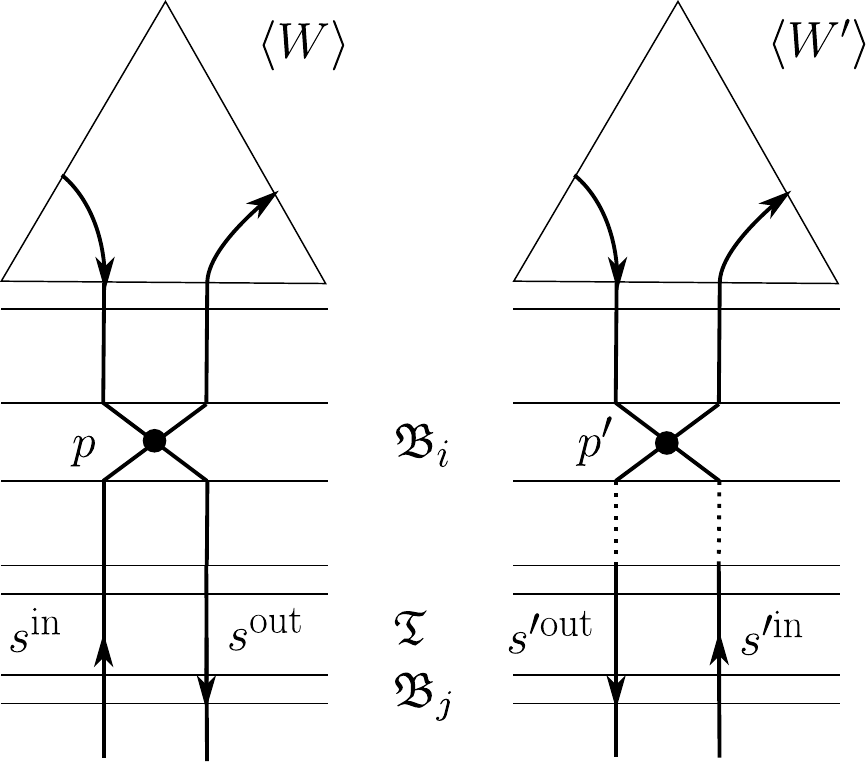}
		\caption{Case 1: crossing before the strands, and strands on the close side of the biangle $\biang_j$}
	\end{subfigure}
\hfill
	\begin{subfigure}{.49\textwidth}
		\centering
		\includegraphics[scale=.35]{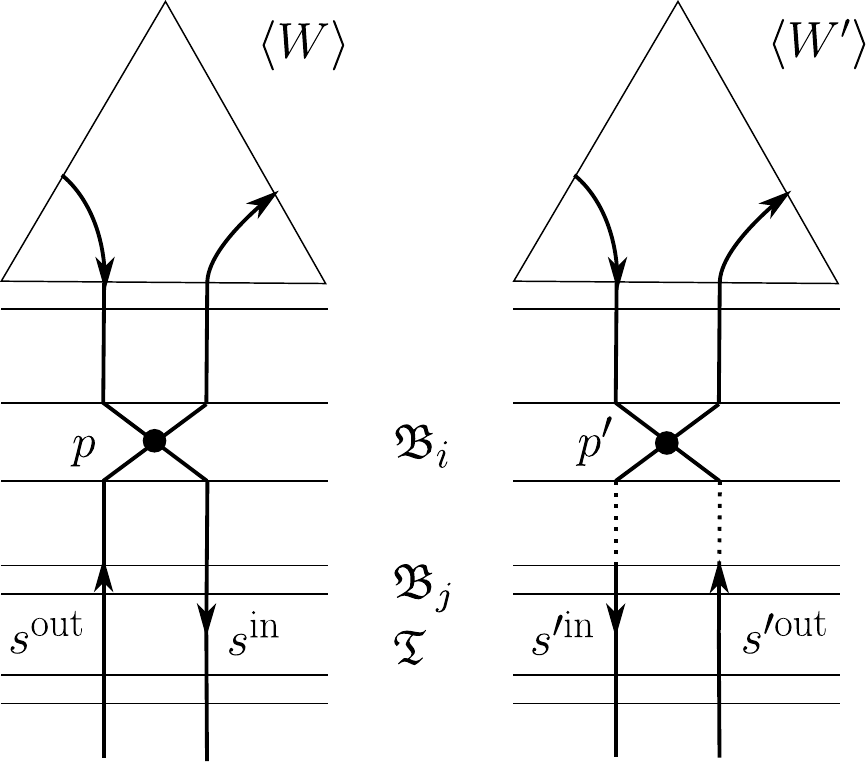}
		\caption{Case 2:  crossing before the strands, and strands on the far side of the biangle $\biang_j$}
	\end{subfigure}
	\caption{Proof of Claim \ref{claim:last-argument-claim}, by contradiction}
	\label{fig:proof-by-contradiction-last-argument}
\end{figure}

Similarly, if $SR$ and $SR^\prime$ are non-crossing, the contradiction is that one of the shared-routes is left-oriented, and the other is right-oriented; see Definition \ref{def:orientationsofnoncrossingsharedroutes}.  Indeed, in the open case (Figure \ref{subfig:open-non-crossing-shared-route}), this violates their matching end orientations (by the Fellow-Traveler Lemma), and in the closed case (Figure \ref{subfig:closed-non-crossing-shared-route}), this violates that both $W$ and $W^\prime$ are closed-left-oriented; see the beginning of \S \ref{ssec:intersection-points}.  

The backward direction of the claim is proved by symmetry.  
\end{proof}

\begin{proof}[Proof of Main Lemma \ref{lem:main-lemma}]
	By Lemma \ref{lem:moving-into-same-bigon}, we may assume that corresponding intersection points lie in the same shared-route-biangle.  By hypothesis, the webs $W$ and $W^\prime$ may differ over triangles $\triang$ by permutations of corner arcs.  However, we gather from Lemma \ref{lem:same-edge-sequences} that they in fact have the same orderings of corner arcs in each triangle $\triang$.  Also, since the ladder-webs in the biangles $\biang$ are uniquely determined by their boundary-edge sequences, it follows that $W$ and $W^\prime$ have the same ladder-web in each biangle $\biang$; see Proposition \ref{prop:ladder-webs}.  
\end{proof}

		\subsection{Proof of the main lemma: proof of Lemma \ref{lem:moving-into-same-bigon}}
		\label{ssec:section26}

We have reduced the proof of the main lemma to proving Lemma \ref{lem:moving-into-same-bigon}.  We begin by laying some groundwork.  

Let $\biang$ be a biangle, and let $\mathscr{P}_\biang = \mathscr{P} \cap \biang$ be the set of intersection points of $\left< W \right>$ in $\biang$.  Let $E$ be a boundary edge of the biangle $\biang$,  and let $\biang_1$ and $\biang_2$ be the two biangles opposite $\biang$ across the triangle $\triang$ adjacent to the edge $E$; see Figure \ref{fig:nested-pyramids}.  

\begin{definition}
Let $p \in \mathscr{P}_\biang$ be an intersection point in $\biang$ of two travelers $\gamma_1$ and $\gamma_2$ in $\left< W \right>$.  We denote by $\overline{\gamma}_1(p, E)$ the half-segment of $\gamma_1$ connecting $p$ to $E$.  Define similarly $\overline{\gamma}_2(p, E)$.  The \textit{pyramid $\Delta(p, E)$ bounded by $p$ and $E$} is the triangular subset of the biangle $\biang$ bordered by the boundary edge $E$ and the two half-segments $\overline{\gamma}_1(p, E)$ and $\overline{\gamma}_2(p, E)$; see Figure \ref{fig:nested-pyramids}.  

Let $P \subset \mathscr{P}_\biang$ be a subset of intersection points.  We call $P$ \textit{saturated with respect to $E$} if 
\begin{equation*}
	\mathscr{P}_\biang \cap \left( \bigcup_{p \in P} \Delta(p, E) \right) = P.
\end{equation*}  
In other words, there are no intersection points in the pyramids $\Delta(p, E)$, $p\in P$, that are not already in $P$.  

An intersection point $p \in \mathscr{P}_\biang$ is \textit{movable with respect to $E$} if, after crossing $E$, the half-segments $\overline{\gamma}_1(p, E)$ and $\overline{\gamma}_2(p, E)$ extend parallel to each other across the adjacent triangle $\triang$, thus landing in the same opposite biangle, either $\biang_1$ or $\biang_2$; see Figure \ref{fig:nested-pyramids}, where on the left, six points are movable, in the middle, four points are movable, and on the right, none are movable.  We say a subset $P \subset \mathscr{P}_\biang$ is \textit{movable with respect to $E$} if each $p \in P$ is movable.  
\end{definition}

\begin{figure}[b]
	\centering
	\includegraphics[width=.9\textwidth]{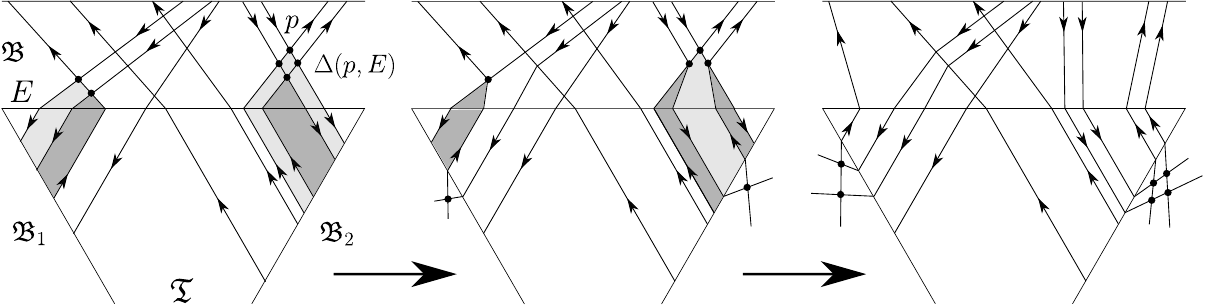}
	\caption{Pushing a saturated movable subset $P$ into adjacent biangles.  (Two rounds of pushes are required to go from the second to third picture.)}
	\label{fig:nested-pyramids}
\end{figure}

\begin{claim}
\label{claim:minefield-sub-step}
	Let $P \subset \mathscr{P}_\biang$ be a subset of intersection points that is saturated and movable with respect to $E$.  Then, there exists a sequence $W = W_0$, $W_1$, $\dots, W_n$ of webs and a sequence $P_{-1} = \emptyset \subsetneq P_0 \subsetneq P_1 \subsetneq \dots \subsetneq P_{n-1}=P \subset \mathscr{P}_\biang$ of intersection points of $\left< W \right>$ in the biangle $\biang$, such that $W_{i+1}$ is obtained from $W_i$ by a finite number of modified H-moves (Figure {\upshape\ref{fig:odified-H-move-BEV-curve11}}) in such a way that the points $P_i - P_{i-1}$ are carried into the two biangles $\biang_1 \cup \biang_2$ and no other intersection points are moved.  After this process is complete, $P$ has been moved into $\biang_1 \cup \biang_2$ and all the other intersection points $\mathscr{P} - P$ remain un-moved in their original biangles.  
\end{claim}

Claim \ref{claim:minefield-sub-step} will first be used in the proof of Claim \ref{claim:minefield-main-step}.  We now prepare to prove Claim \ref{claim:minefield-sub-step}.  

We say that $p \in \mathscr{P}_\biang$ is \textit{immediately movable with respect to $E$} if it is movable and there are no other intersection points in the pyramid $\Delta(p, E)$, that is $\Delta(p, E) \cap \mathscr{P}_\biang = \{ p \}$.  Equivalently, $\mathrm{Int}(\Delta(p, E)) \cap \left< W \right> = \emptyset$, hence a modified H-move can be applied to carry  $p$ across the edge $E$, across the adjacent triangle $\triang$, and into one of the opposite biangles $\biang_1$ or $\biang_2$; see Figure \ref{fig:nested-pyramids}, where on the left and in the middle, two and three points are immediately movable.

The following statement is evident from the ladder-web structure in the biangle $\biang$.  

\begin{fact}[Nested pyramids]
\label{fact:movable-points}
	If $q \in \mathscr{P}_\biang \cap \Delta(p, E)$
 is an intersection point in the pyramid $\Delta(p, E)$, then $\Delta(q, E) \subset \Delta(p, E)$.  Consequently, if $p$ is movable, then so is $q$.  Therefore, if $p$ is movable, then there exists an inner-most $q$ in $\Delta(p, E)$ that is immediately movable.  \qed
\end{fact}

\begin{proof}[Proof of Claim \ref{claim:minefield-sub-step}]
By induction, assume $W_i$ and $P_{i-1}$ are given.  At this stage, the intersection points $P_{i-1}$ have been moved into $\biang_1 \cup \biang_2$, and the intersection points $P - P_{i-1} \neq \emptyset$ are still in  $\biang$.  Note that, since $P$ is saturated in $\left< W \right>$, $P - P_{i-1}$ is saturated in $\left< W_i \right>$, that is 
\begin{equation*}
	P - P_{i-1} = \mathscr{P}_\biang^{(i)} \cap \left( \cup_{p \in P - P_{i-1}} \Delta^{(i)}(p, E) \right)
	\quad  \subset \left< W_i \right>.
\end{equation*}
	Since by hypothesis each $p \in P - P_{i-1}$ is movable, by Fact \ref{fact:movable-points} the subset
\begin{equation*}
	Q_i = \{  q \in \mathscr{P}^{(i)}_\biang \cap ( \cup_{p \in P - P_{i-1}} \Delta^{(i)}(p, E)) ; \quad q \text{ is immediately movable} \} \neq \emptyset,
\end{equation*}
is non-empty.  In particular, $Q_i \subset P - P_{i-1}$.  We can thus apply modified H-moves to $W_i$ to move the intersection points $Q_i$ from the biangle $\biang$ into the two biangles $\biang_1 \cup \biang_2$, yielding the new web $W_{i+1}$.  Putting $P_{i} = P_{i-1} \cup Q_i$ finishes the induction step; see Figure \ref{fig:nested-pyramids}, where $n=3$.    
\end{proof}

The following statement is  immediate from  Fact \ref{fact:movable-points}.  

\begin{fact}[Saturation of a subset of intersection points]
\label{fact:suf-cond-for-saturated}
	For any subset $Q \subset \mathscr{P}_\biang$, the set
\begin{equation*}
	P = \mathscr{P}_\biang \cap \left( \bigcup_{q \in Q} \Delta(q, E) \right) \quad \subset \mathscr{P}_\biang,
\end{equation*}
is saturated with respect to $E$.  \qed
\end{fact}

We continue moving toward the proof of Lemma \ref{lem:moving-into-same-bigon}.  We are now dealing with two webs $W$ and $W^\prime$.  Let $E$, $\triang$, $\biang$, $\biang_1$, $\biang_2$ be as before.  We begin by setting some notation.  

Given a subset $P \subset \mathscr{P}_\biang$, put
\begin{equation*}
	P^\prime(P, E) = \mathscr{P}^\prime_\biang \cap \left( \bigcup_{ \{ p \in P; \quad p \text{ and } p^\prime \text{ lie in the same shared-route-biangle} \} } \Delta(p^\prime, E)
	\right)
	\quad  \subset \mathscr{P}^\prime_\biang.
\end{equation*}
In other words, $P^\prime(P, E)$ consists of the points in $\mathscr{P}^\prime_\biang$ lying in the pyramids $\Delta(p^\prime, E)$ generated by those intersection points $p^\prime$ in $\mathscr{P}^\prime_\biang$ whose corresponding intersection point $p$ lies in the same shared-route-biangle as $p^\prime$ and satisfies $p \in P$.  Symmetrically, given a subset $P^\prime \subset \mathscr{P}^\prime_\biang$, put
\begin{equation*}
	P(P^\prime, E) = \mathscr{P}_\biang \cap \left( \bigcup_{ \{ p^\prime \in P^\prime; \quad p^\prime \text{ and } p \text{ lie in the same shared-route-biangle} \} } \Delta(p, E)
	\right) 
	\quad  \subset  \mathscr{P}_\biang.
\end{equation*}
Note that (1) the above shared-route-biangles necessarily coincide with the biangle $\biang$, and (2) generally, either of the sets $P^\prime(P, E)$ or $P(P^\prime, E)$ may be empty.  

\begin{fact}
\label{fact:movable-then-movable}
	The union of movable sets is movable.  Let $P \subset \mathscr{P}_\biang$ (resp. $P^\prime \subset \mathscr{P}^\prime_\biang$) be movable with respect to $E$.  Then $P^\prime(P, E) \subset \mathscr{P}^\prime_\biang$ (resp. $P(P^\prime, E) \subset \mathscr{P}_\biang$) is movable with respect to $E$.  
\end{fact}

\begin{proof}
	The first statement is obvious.  For the second, if $p \in P$ is movable and if $p^\prime$ lies in the same shared-route-biangle as $p$, then, by the Fellow-Traveler Lemma \ref{lem:life-neighbors-lemma}, $p^\prime$ is movable.  By Fact \ref{fact:movable-points}, $\mathscr{P}^\prime_\biang \cap \Delta(p^\prime, E) \subset \mathscr{P}^\prime_\biang$ is movable.  By the first statement, $P^\prime(P, E)$ is movable.  
\end{proof}

We are now prepared to prove Lemma \ref{lem:moving-into-same-bigon}, which we re-state here for convenience.  

\begin{customlem}{\ref{lem:moving-into-same-bigon}}
	There is a sequence of modified H-moves (Figure {\upshape\ref{fig:odified-H-move-BEV-curve11}}) applicable to the web $W$ and  a sequence of modified H-moves applicable to $W^\prime$, after which the bijection $\mathscr{P} \leftrightarrow \mathscr{P}^\prime$ from Corollary {\upshape\ref{cor:same-number-of-intersection-points}} satisfies the property that each intersection point $p$ in the global picture $\left< W \right>$ and its corresponding intersection point $p^\prime$ in $\left< W^\prime \right>$ lie in the same shared-route-biangle.
\end{customlem}

\begin{proof}  
	  \textit{Step 1.} Let $N$ equal the cardinality $N = |\mathscr{P}| = |\mathscr{P}^\prime|$.  Define
\begin{equation*}
	N(W, W^\prime) = 
	\left|  \{ p \in \mathscr{P} ; \quad  p \text{ and } p^\prime \text{ lie in the same shared-route-biangle} \}  \right|
	\quad  \in \Z_{\geq 0}.
\end{equation*}
If $N(W, W^\prime) = N$, then we are done.  So assume $N(W, W^\prime) < N$.  

The strategy is simple.  If two intersection points $p \in \mathscr{P}_\biang$ and $p^\prime \in \mathscr{P}_{\biang^\prime}$ do not lie in the same shared-route-biangle, then we choose sufficiently large saturated movable sets $p \in P \subset \mathscr{P}_\biang$ and $p^\prime \in P^\prime \subset \mathscr{P}_{\biang^\prime}$ such that pushing $P$ and $P^\prime$ into adjacent biangles via Claim \ref{claim:minefield-sub-step} does not decrease $N(W, W^\prime)$.  This can be done in a controlled way so that eventually $N(W, W^\prime)$ increases.  

\textit{Step 2.}  Let $E$, $\triang$, $\biang$, $\biang_1$, $\biang_2$ be as above.

\begin{claim}
\label{claim:minefield-main-step}
	Let $p_0 \in \mathscr{P}_\biang$ be movable with respect to $E$.  Then, there exist subsets $p_0 \in P(p_0) \subset \mathscr{P}_\biang$ and $P^\prime(p_0) \subset \mathscr{P}^\prime_\biang$, and webs $W_1$ and $W^\prime_1$ obtained by applying finitely many modified H-moves to $W$ and  $W^\prime$, respectively, such that: in $\left< W_1 \right>$ and $\left< W_1^\prime \right>$ the subsets $P(p_0)$ and $P^\prime(p_0)$ have been moved into $\biang_1 \cup \biang_2$;  also $\mathscr{P} - P(p_0)$ and $\mathscr{P}^\prime - P^\prime(p_0)$ are un-moved; and, 
	\begin{equation*}
\tag{$\ast$}
\label{eq:biangles-intersections-nondecreasing}
	N \geq N(W_1, W^\prime_1) \geq N(W, W^\prime)
	\quad  \in  \Z_{\geq 0}.
\end{equation*}
\end{claim}

We prove the claim.  Our main task is to define two subsets $p_0 \in P(p_0) \subset \mathscr{P}_\biang$ and $P^\prime(p_0) \subset \mathscr{P}^\prime_\biang$ that are saturated and movable with respect to $E$, satisfying the property that
\begin{gather*}
	p \in P(p_0), \quad 
	p \text{ and } p^\prime \text{ lie in the same shared-route-biangle}
\\\tag{$\ast \ast$}\label{eq:biangles-symmetric-condition}	\Longleftrightarrow
\\	p^\prime \in P^\prime(p_0), \quad p^\prime \text{ and } p \text{ lie in the same shared-route-biangle}.
\end{gather*}
We do this simultaneously by a ping-pong procedure.  

Put $P_1 = \mathscr{P}_\biang \cap \Delta(p_0, E)$ and $P^\prime_1 = P^\prime(P_1, E) \subset \mathscr{P}^\prime_\biang$.  Having defined $P_i \subset \mathscr{P}_\biang$ and $P^\prime_i \subset \mathscr{P}^\prime_\biang$, put $P_{i+1} = P_i \cup P(P^\prime_i, E)$ and $P^\prime_{i+1} = P^\prime_i \cup P^\prime(P_{i+1}, E)$.  This defines two nested infinite sequences $P_1 \subset P_2 \subset \cdots \subset \mathscr{P}_\biang$ and $P^\prime_1 \subset P^\prime_2 \subset \cdots \subset \mathscr{P}^\prime_\biang$.  Since $\mathscr{P}_\biang$ and $\mathscr{P}^\prime_\biang$ are finite, these sequences stabilize: $P_i = P_{i+1}$ and $P^\prime_i = P^\prime_{i+1}$ for all $i \geq i_0$. Set $P(p_0) = P_{i_0} \ni p_0$ and $P^\prime(p_0) = P^\prime_{i_0}$.  

Note that, by construction, there exists $Q \subset \mathscr{P}_\biang$ and $Q^\prime \subset \mathscr{P}^\prime_\biang$ such that
\begin{equation*}
	P(p_0) = \mathscr{P}_\biang \cap \left( \bigcup_{q \in Q} \Delta(q, E) \right) 
	\quad  \text{and}  \quad
	P^\prime(p_0) = \mathscr{P}^\prime_\biang \cap \left( \bigcup_{q^\prime \in Q^\prime} \Delta(q^\prime, E) \right).  
\end{equation*}
By Fact \ref{fact:suf-cond-for-saturated}, $P(p_0)$ and $P^\prime(p_0)$ are saturated with respect to $E$.  

Observe also that since $p_0 \in \mathscr{P}_\biang$ is movable by hypothesis, $P_1 = \mathscr{P}_\biang \cap \Delta(p_0, E)$ is movable by Fact \ref{fact:movable-points}, hence $P(p_0) \subset \mathscr{P}_\biang$ and $P^\prime(p_0) \subset \mathscr{P}^\prime_\biang$ are movable by Fact \ref{fact:movable-then-movable}.  

To check Equation \eqref{eq:biangles-symmetric-condition}, by symmetry it suffices to check one direction.  Assume $p \in P(p_0)$ and that $p$ and $p^\prime$ lie in the same shared-route-biangle.  Let $i$ be such that $p \in P_i$.  Then
\begin{equation*}
	p^\prime \in \mathscr{P}^\prime_\biang \cap \Delta(p^\prime, E) \subset P^\prime(P_i, E) \subset P^\prime_{i} \subset P^\prime(p_0).
\end{equation*}

To prove Equation \eqref{eq:biangles-intersections-nondecreasing}, we use Claim \ref{claim:minefield-sub-step} to move the saturated and movable sets $P(p_0)$ and $P^\prime(p_0)$, and only these sets, into the opposite biangles $\biang_1 \cup \biang_2$ via finitely many modified H-moves applied to $W$ and $W^\prime$, yielding the desired webs $W_1$ and $W^\prime_1$ (note what we are here calling $W_1$ was called $W_n$ in the statement of Claim \ref{claim:minefield-sub-step}).  If $p \in P(p_0)$ moves into the biangle $\biang_1$ (resp. $\biang_2$), and if $p^\prime$ lies in the same shared-route-biangle as $p$, so that $p^\prime \in P^\prime(p_0)$ by Equation \eqref{eq:biangles-symmetric-condition}, then by the Fellow-Traveler Lemma \ref{lem:life-neighbors-lemma} $p^\prime$ also moves into the biangle $\biang_1$ (resp. $\biang_2$), and similarly if the roles of $p$ and $p^\prime$ are reversed.  

\textit{Step 3.}  To finish the proof, assume that $p$ and $p^\prime$ do not lie in the same shared-route-biangle.  Then it makes sense to talk about which of $p$ or $p^\prime$ is \textit{farther away} from the source-end $\E$ or $\E^\prime$ of the left-oriented crossing shared-route $SR(p)$ or $SR^\prime(p^\prime)$ which it generates, respectively.  More precisely, if $p \in_{SR(p)} \biang_i$ and $p^\prime \in _{SR^\prime(p^\prime)} \biang_j$, $i, j \geq 0$, then $i \neq j$ and $p$ being farther away is equivalent to $i > j$.  

Assume $p$ is farther away, so $i > j$.    By Claim \ref{claim:minefield-main-step}, we can push $p$ one step closer to the source-end $\E$, that is we can push $p$ into $\biang_{i-1}$.  For this step, $p^\prime$ either (1) stays in $\biang_j$, (2) is pushed into $\biang_{j-1}$, or (3) is pushed into $\biang_{j+1}$; see Figure \ref{fig:algorithm-moving-into-the-same-biangle}. Notice since no two adjacent edges in a shared-route can represent the same edge in the split ideal triangulation $\splitidealtriang$ (by the no-switchbacks property), case (3) can only happen if $j < i-1$.  Also, again by Claim \ref{claim:minefield-main-step}, as a result of this step the number $N(W, W^\prime)$ only increases or stays the same.  

Since the indices $i$,$j$ are bounded below, after multiple applications of this step eventually $p$ and $p^\prime$ fall into the same shared-route-biangle, at which point $N(W, W^\prime)$ strictly increases.  Repeating this procedure for each pair $p$ and $p^\prime$ completes the proof of Lemma \ref{lem:moving-into-same-bigon}.  
\end{proof}

\begin{figure}[htb]
	\centering
	\includegraphics[width=\textwidth]{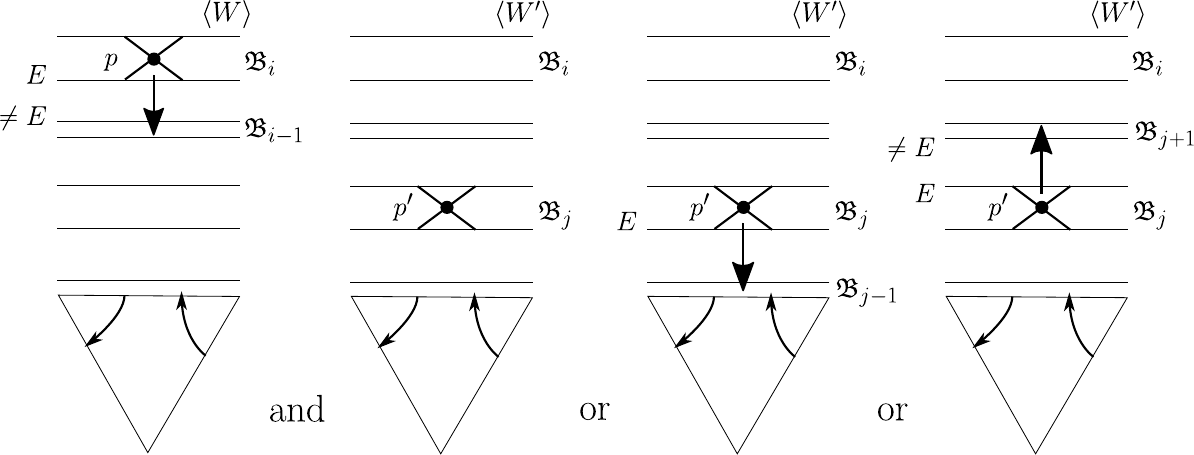}
	\caption{Moving intersection points into the same shared-route-biangle}
	\label{fig:algorithm-moving-into-the-same-biangle}
\end{figure}

		\section{Webs on surfaces-with-boundary}
		\label{sec:webs-on-surfaces-with-boundary}

We generalize Theorem \ref{thm:main-theorem-1} to the case of surfaces-with-boundary $\surfbord$.  More precisely, we give two distinct, but complementary, versions of the result.  The first version, where we think of the surface $\surfbord$ as generalizing punctured surfaces $\surf$, originates in the geometry and topology of $\mathrm{SL}_3(\mathbb{C})$-character varieties.  The second version, where we think of the surface $\surfbord$ as generalizing ideal polygons $\poly_k$, originates in the representation theory of the Lie group $\mathrm{SL}_3(\mathbb{C})$.  The proof of either statement is essentially the same as in the empty boundary case.

		\subsection{Essential webs}
		\label{ssec:topological-setting--boundarysetting}

		\subsubsection{Surfaces-with-boundary}
		\label{sssec:surfaces-with-boundary}

Our surfaces, now denoted $\surfbord=\overline{\surf}-P$, are obtained by removing a finite set $P$ of punctures from a compact oriented surface $\overline{\surf}$.  We require that there is at least one puncture, that each boundary component of $\overline{\surf}$ contains a puncture, and that the resulting punctured surface $\surfbord$ admits an ideal triangulation $\idealtriang$; this last property is equivalent to the Euler characteristic condition $\chi(\surfbord) < d/2$, where $d$ is the number of components of $\partial \surfbord$.  The boundary edges of $\surfbord$ count as edges in an ideal triangulation $\idealtriang$.  

Once again, for simplicity, we assume that $\idealtriang$ does not contain any self-folded triangles; however, our results should extend to this setting essentially without change.  

The split ideal triangulation $\splitidealtriang$ associated to an ideal triangulation $\idealtriang$ is defined as in \S \ref{ssec:split-ideal-triangulations}.  In particular, the boundary edges of $\surfbord$ are split as well.

		\subsubsection{Essential webs}
		\label{sssec:webs--boundarysetting}

A \textit{global web}, or just \textit{web}, $W$ on the surface $\surfbord$ is defined as in Definition \ref{def:local-web-on-a-surface}, except ``$\mathfrak{D}_k$'' is replaced by ``$\surfbord$'', and ``local'' is replaced by ``global''. 

The internal and external faces of a web $W$ on $\surfbord$ are defined as in Definition \ref{def:external-web-face-closed-surface}, except with the appropriate replacements as above.  As usual, a web $W$ on $\surfbord$ is non-elliptic if all of its internal faces have at least six sides; compare Definition \ref{def:non-elliptic-web-with-boundary}.  

An \textit{essential} web $W$ on $\surfbord$ is defined as in Definition \ref{def:essential-web-with-boundary}, where in addition the arc $\alpha$ needs to be isotopic (respecting boundary) in $\surfbord$ to the segment $\overline{E}$.  

The good position of a web $W$ with respect to a split ideal triangulation $\splitidealtriang$ is defined exactly as in Definition \ref{def:good-position}, without change.

		\subsection{Rung-less essential webs; first version of the boundary result}
		\label{sssec:rung-less-essential-webs-first-version-of-the-result}

		\subsubsection{Rung-less essential webs}
		\label{sssec:rung-less-essential-webs}

As usual, a \textit{rung-less} web $W$ on the surface $\surfbord$ is a web that does not have any H-faces; compare Definition \ref{def:rung-less-webs}.  

The \textit{parallel-equivalence class of a rung-less web} $W$ is defined as in Definition \ref{def:parallel-equivalent-webs}, except we have to include another global parallel-move exchanging two arcs that together with segments in $\partial \surfbord$ form the boundary of an embedded rectangle $R$ in the surface $\surfbord$; for instance, this would be the case in Figure \ref{fig:parallel-move} had we not identified the top and bottom edges of the surface.  

The property of being essential is preserved by parallel-equivalence.  The collection of parallel-equivalence classes of rung-less essential webs is denoted by $[\webbasis{\surfbord}]$.  (Note, by definition, the empty class $[W]=[\emptyset]$ is in $[\webbasis{\surfbord}]$.)

		\subsubsection{Knutson-Tao cone associated to an ideal triangulation}
		\label{sssec:knutson-tao-cone-associated-to-a-triangulation--boundarysetting}

To an ideal triangulation $\idealtriang$ of the surface $\surfbord$ we associate a dotted ideal triangulation, also denoted $\idealtriang$, as in \S \ref{ssec:dotted-ideal-triangulations}.  In particular, there are dots located on the boundary edges of $\surfbord$, as for example in Figure \ref{fig:dotted-triangle} where $\surfbord = \triang$ is an ideal triangle.  The number $N$ of dots in the dotted triangulation $\idealtriang$ can be computed as $N = 2 * \#\left\{\text{edges } E \text{ of } \idealtriang \right\} + \#\left\{\text{triangles } \triang \text{ of } \idealtriang \right\}$.  (Since each ideal triangulation $\lambda$ has $-3\chi(\surfbord)+2d$ edges and $-2\chi(\surfbord)+d$ triangles, note $N$ is independent of~$\lambda$.)

To the dotted triangulation $\idealtriang$ we associate the \textit{Knutson-Tao cone} $\KTcone{\idealtriang} \subset \Zpos^N$, as in \S \ref{ssec:global-knutson-tao-cone}.

		\subsubsection{Coordinates for rung-less essential webs}
		\label{sssec:coordinates-for-rung-less-essential-webs--boundarysetting}

The minimal position of a rung-less web $W$ with respect to an ideal triangulation $\idealtriang$ is defined as in Definition \ref{def:minimal position}.  Then, Proposition \ref{prop:minimal-position} holds word for word, except ``non-elliptic'' is replaced by ``rung-less essential''.  		

Modified H-moves take rung-less essential webs in good position to webs of the same type.  Proposition \ref{prop:good-position} holds verbatim, except ``non-elliptic'' is replaced by ``rung-less essential''.  

Given an ideal triangulation $\idealtriang$, we define the Fock-Goncharov global coordinate function
	$\themap{\idealtriang}^\mathrm{FG} : [\webbasis{\surfbord}] \to \KTcone{\idealtriang}$
 as in \S \ref{ssec:global-coordinates-from-local-coordinate-functions}; see Definition \ref{def:fock-goncharov-global-coordinate-function}.  
 
 \begin{theorem}[First boundary result]
\label{thm:main-theorem-v1}
	The Fock-Goncharov global coordinate function
\begin{equation*}
	\themap{\idealtriang}^\FG : [\webbasis{\surfbord}]
	\overset{\sim}{\longrightarrow} \KTcone{\idealtriang} \subset \mathbb{Z}_{\geq 0}^N
\end{equation*}
is a bijection of sets, identifying parallel-equivalence classes of rung-less essential webs on the surface $\surfbord$ with points of the Knutson-Tao cone associated to the ideal triangulation $\idealtriang$.
\end{theorem}

\begin{proof}
	As in the proof of Theorem \ref{thm:main-theorem-1}, the strategy is to construct an explicit inverse
\begin{equation*}
	\invmap{\idealtriang}^\mathrm{FG} : \KTcone{\idealtriang}
	\longrightarrow [\webbasis{\surfbord}].  
\end{equation*}
The mapping $\invmap{\idealtriang}^\mathrm{FG}$ is defined via the ladder gluing construction followed by removing internal elliptic faces, as explained in \S \ref{ssec:ladder-gluing-construction}-\ref{ssec:inverse-mapping-correcting-an-elliptic-web}.  Because of the rung-less condition, we also need to remove external H-faces, which can be done at the cost of  swapping two strands of the web lying on the boundary $\partial \surfbord$.  For two examples of this procedure, see Figures \ref{fig:ladder-construction-boundary-1} and \ref{fig:ladder-construction-boundary-1-twoways} (compare Figures \ref{fig:coordinates-example2} and \ref{fig:elliptic-web-non-example}).  As before, the resulting rung-less essential web is not unique in general.  

In order to deal with this ambiguity, we need the analogue of Main Lemma \ref{lem:main-lemma}, saying that two rung-less essential webs resulting from the ladder gluing construction are parallel-equivalent.  The proof of the main lemma is essentially unchanged from \S \ref{sec:proof-of-main-lemma}.  To say a word about it, the proof of Corollary \ref{cor:same-number-of-intersection-points} requires the fact that there are no crossing shared-routes terminating on the boundary $\partial \surfbord$.  This follows from the rung-less condition.  
\end{proof}

\begin{remark}\label{rem:kim}
	Theorem \ref{thm:main-theorem-v1} is closely related to \cite[Proposition 1.12]{KimArxiv20}.  
\end{remark}

\begin{figure}[htb]
	\centering
	\includegraphics[width=.6\textwidth]{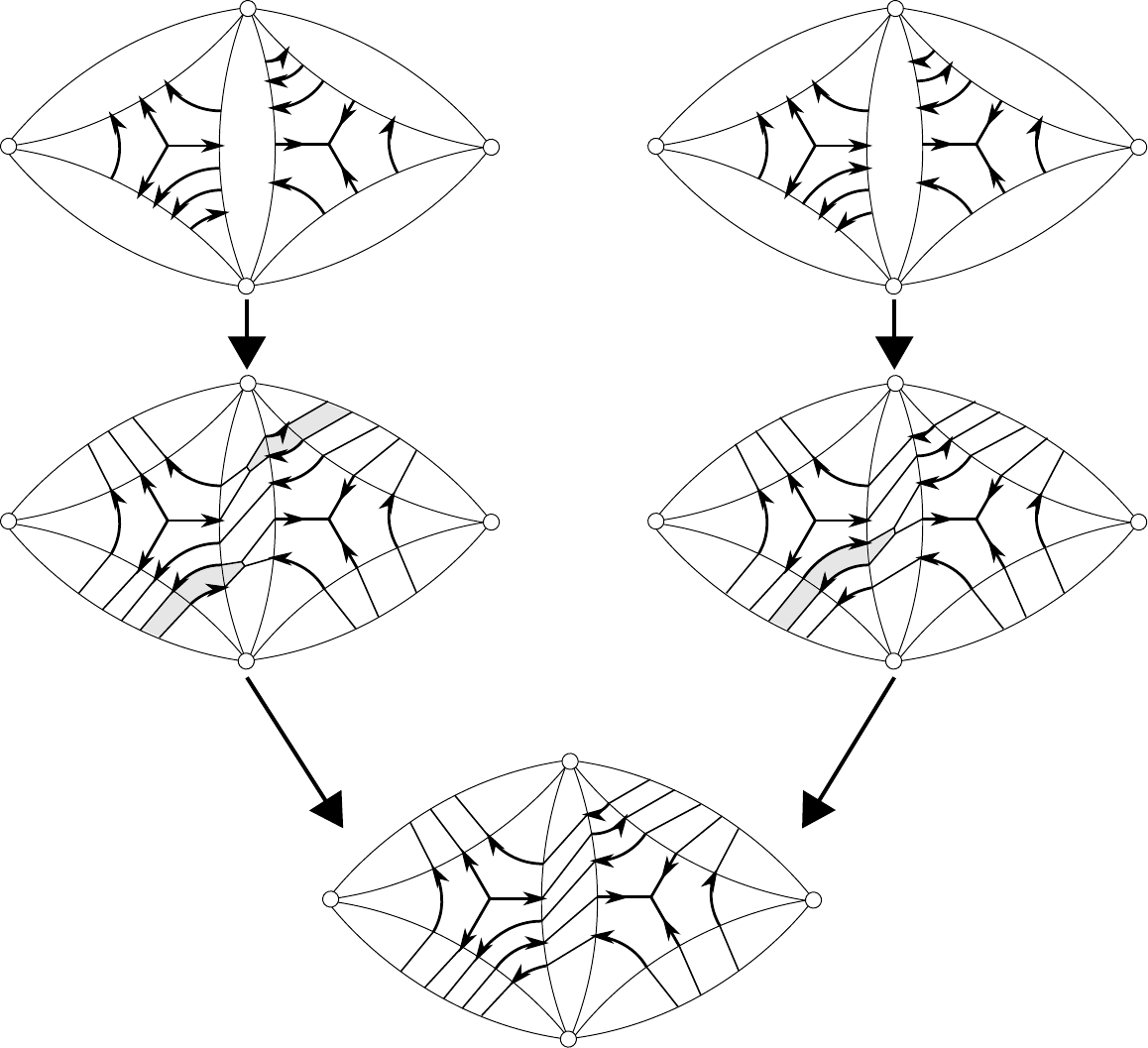}
	\caption{Ladder gluing construction for rung-less essential webs: 1 of 2 (on the ideal square)}
	\label{fig:ladder-construction-boundary-1}
\end{figure}

\begin{figure}[htb]
	\centering
	\includegraphics[width=.46\textwidth]{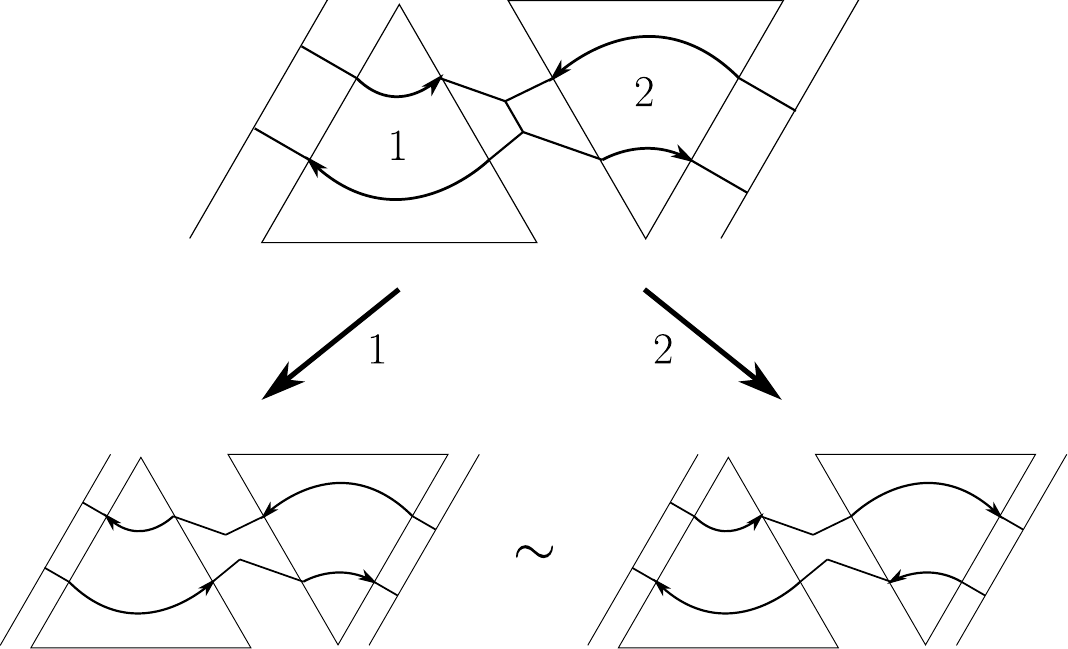}
	\caption{Ladder gluing construction for rung-less essential webs: 2 of 2}
	\label{fig:ladder-construction-boundary-1-twoways}
\end{figure}

		\subsection{Application: geometry and topology of \texorpdfstring{$\mathrm{SL}_3(\mathbb{C})$}{SL3}-character varieties}
		\label{ssec:application-geometry-and-topology-of-sl3c-character-varieties}

As a consequence of the first version of the result, Theorem \ref{thm:main-theorem-v1}, we give an alternative geometric proof of the Sikora-Westbury theorem \cite[Theorem 9.5]{SikoraAlgGeomTop07} (see also \cite[Proposition 4]{MR4359515frohmansikora}), whose original proof in \cite{SikoraAlgGeomTop07} assumes the Diamond Lemma from non-commutative algebra.

\begin{corollary}[Application of the first boundary result]
\label{cor:application-of-the-first-version}
	The collection $[\webbasis{\surfbord}]$ of parallel-equivalence classes of rung-less essential webs on the surface $\surfbord$ indexes a natural linear basis for the algebra $\mathbb{C}[\mathscr{R}_{\mathrm{SL}_3(\mathbb{C})}(\surfbord)]$ of regular functions on the $\mathrm{SL}_3(\mathbb{C})$-character variety.  
\end{corollary}

Here, the character variety $\mathscr{R}_{\mathrm{SL}_n(\mathbb{C})}(\surfbord)$,  for general $n$, was discussed in the introduction for surfaces $\surfbord = \surf$ with empty boundary.  When $\partial \surfbord \neq \emptyset$, there is not a mainstream definition for $\mathscr{R}_{\mathrm{SL}_n(\mathbb{C})}(\surfbord)$.  Possible models may be found in \cite{FominPNAS14, GoncharovInvent15, MR4493620costantinole, korinman2019arxiv, MR4359515frohmansikora, MR4609753higgins}.  In order for Corollary \ref{cor:application-of-the-first-version} to be well-posed, we make use of a purely topological model, via \textit{skein algebras}, which has the advantage of admitting a natural deformation quantization.  

\begin{definition}
		Following Frohman-Sikora \cite[\S 1,12]{MR4359515frohmansikora}, for a surface-with-boundary $\surfbord$ we \underline{define} the \textit{algebra $\mathbb{C}[\mathscr{R}_{\mathrm{SL}_3(\mathbb{C})}(\surfbord)]$ of regular functions on the $\mathrm{SL}_3(\mathbb{C})$-character variety} to be the commutative (reduced skein) algebra $\mathscr{S}^{1}(\surfbord)$ of  \cite[\S 3]{MR4359515frohmansikora}, where we have taken the specialization $q=a=1$ of their deformation parameters.  
\end{definition}

In particular, a web $W$ on $\surfbord$ represents an element of $\mathscr{S}^1(\surfbord)$.  

\begin{remark}
	When $\surfbord = \surf$, Sikora \cite{SikoraTrans01} proved that the trace functions $\mathrm{Tr}_W$ on the character variety $\mathscr{R}_{\mathrm{SL}_3(\mathbb{C})}(\surf)$ furnish a natural isomorphism $\mathscr{S}^1(\surf) \cong \mathbb{C}[\mathscr{R}_{\mathrm{SL}_3(\mathbb{C})}(\surf)]$.  
\end{remark}

\begin{proof}[Proof of Corollary \ref{cor:application-of-the-first-version}]
	We prove a more general statement.  Let $\mathscr{S}^{q,a}(\surfbord)$ be the \textit{reduced $\mathrm{SL}_3$-skein algebra} of \cite[\S 3]{MR4359515frohmansikora} for deformation parameters $q, a \in \mathbb{C} - \left\{ 0 \right\}$.  So, $\mathscr{S}^1(\surfbord) = \mathscr{S}^{1,1}(\surfbord)$.  
	
	Higgins \cite{MR4609753higgins} defined a \textit{$\mathrm{SL}_3$-stated skein algebra} $\mathscr{S}^q_\mathrm{st}(\surfbord)$ generalizing the $\mathrm{SL}_2$-stated skein algebra of \cite{MR4493620costantinole}.  More precisely, we define $\mathscr{S}^q_\mathrm{st}(\surfbord)$ to be Kim's \cite[\S 5]{KimArxiv20} adaptation of Higgins' stated skein algebra. 
	
	When $a = 1$, inclusion provides a natural algebra homomorphism $\iota : \mathscr{S}^{q, 1}(\surfbord) \to \mathscr{S}^q_\mathrm{st}(\surfbord)$ from the Frohman-Sikora reduced skein algebra to the Higgins stated skein algebra; in fact, the mapping $\iota$ is onto the sub-algebra $\mathscr{S}^q_{\mathrm{st}=\mathrm{top}}(\surfbord)$ generated by webs with all-top-states on the boundary $\partial \surfbord$.  Put $\mathscr{S}^q(\surfbord) := \mathscr{S}^{q, 1}(\surfbord)$.  In summary, $\iota : \mathscr{S}^q(\surfbord) \twoheadrightarrow \mathscr{S}^q_{\mathrm{st}=\mathrm{top}}(\surfbord) \subset \mathscr{S}^q_\mathrm{st}(\surfbord)$.  
	
	Using the boundary orientation of $\partial \surfbord$ induced by the orientation of $\surfbord$, a web $W$ on $\surfbord$ lifts to an element of the skein algebra $\mathscr{S}^q(\surfbord)$.  Moreover, parallel-equivalent rung-less webs $W \sim W^\prime$ determine the same element of $\mathscr{S}^q(\surfbord)$ (by \cite[Figure 6]{MR4359515frohmansikora} since $a=1$).  We prove $[\webbasis{\surfbord}]$ forms a basis for $\mathscr{S}^q(\surfbord)$.  It is immediate by construction that $[\webbasis{\surfbord}]$ is spanning.

We mimic the strategy of \cite[\S 8]{BonahonGT11} in the $\mathrm{SL}_2$-case.  Fix an ideal triangulation $\idealtriang$ of $\surfbord$.  Building on \cite{DouglasThesis20, DouglasArxiv21b, DouglasArxiv21}, Kim \cite{KimArxiv20} defined a \textit{$\mathrm{SL}_3$-quantum trace map}, which in particular is an algebra homomorphism $\mathrm{Tr}^q_\idealtriang : \mathscr{S}^q_\mathrm{st}(\surfbord) \to \mathscr{T}^q_\idealtriang$ from the stated skein algebra $\mathscr{S}^q_\mathrm{st}(\surfbord)$ to a \textit{quantum torus} $\mathscr{T}^q_\idealtriang$ depending on $\idealtriang$.  More precisely, $\mathscr{T}^q_\idealtriang = \mathbb{C}[Z_1^{\pm 1},Z_2^{\pm 1}, \dots, Z_N^{\pm 1}]^q$ is a non-commutative $q$-deformation of the algebra of Laurent polynomials in variables $Z_i$, which no longer commute but $q$-commute (according to a quiver drawn on the triangulated surface).  Here, $N$ is the number of coordinates in Theorem \ref{thm:main-theorem-v1}.  (When $q=1$, the variables $Z_i=X_i^{1/3}$ can be thought of as formal cube roots of the Fock-Goncharov coordinates $X_i$.)

	By \cite[Proposition 5.80]{KimArxiv20} (and \cite[Proposition 3.15]{kim2021mutation}), the quantum trace map $\mathrm{Tr}^q_\idealtriang$ satisfies the property that the polynomial $\mathrm{Tr}^q_\idealtriang(\iota(W))$, obtained by evaluating a rung-less essential web $W$ in $[\webbasis{\surfbord}]$, has a  \textit{highest term} $ Z_1^{a_1} Z_2^{a_2} \cdots Z_N^{a_N}$ (omitting the power of $q$ coefficient) whose exponents are  the coordinates $(a_1, a_2, \dots, a_N) = \themap{\idealtriang}^\mathrm{FG}(W) \in \Zpos^N$ of Theorem \ref{thm:main-theorem-v1}.  (Here, by highest term, we mean that if a monomial $Z_1^{a^\prime_1} Z_2^{a^\prime_2} \cdots Z_N^{a^\prime_N}$ also appears in $\mathrm{Tr}^q_\idealtriang(\iota(W))$, then $a^\prime_i \leq a_i$ for all $i = 1, 2, \dots, N$.)
	
	It follows that each $W$ is nonzero in $\mathscr{S}^q(\surfbord)$, that $\iota$ is injective on $[\webbasis{\surfbord}]$, that $\iota([\webbasis{\surfbord}])$ is linearly independent in $\mathscr{S}^q_\mathrm{st}(\surfbord)$, and lastly that $\iota : \mathscr{S}^q(\surfbord) \overset{\sim}{\to} \mathscr{S}^q_{\mathrm{st}=\mathrm{top}}(\surfbord) \subset \mathscr{S}^q_\mathrm{st}(\surfbord)$ is an isomorphism.  In particular, we gather $[\webbasis{\surfbord}]$ is independent, hence a basis of $\mathscr{S}^q(\surfbord)$.
\end{proof}

		\subsection{Boundary-fixed essential webs; second version of the boundary result}
		\label{ssec:boundary-fixed-essential-webs-second-version-of-the-result}

		\subsubsection{Boundary-fixed essential webs}
		\label{sssec:boundary-fixed-essential-webs}

For a boundary edge $E$ of the surface $\surfbord$, a \textit{strand-set} $S_E$ is a (possibly empty) set $S_E = \{ s \}$ of disjoint oriented strands $s$ located on $E$ (compare Definition \ref{def:ladder-web}); see Figures \ref{fig:ladder-construction-boundary-2} and \ref{fig:ladder-construction-boundary-2-twoways}, where the strands are indicated by white-headed arrows.  A \textit{strand-set $S_{\partial \surfbord} = \{ S_E \}$ for $\surfbord$} is a collection of strand-sets $S_E$ varying over all $E \subset \partial \surfbord$.  

\begin{definition}
	A \textit{boundary-fixed web} $W$ with respect to a strand-set $S_{\partial \surfbord}$ for the surface $\surfbord$ is a web $W$ whose end-strands match the strand-set $S_{\partial \surfbord}$; see Figures \ref{fig:ladder-construction-boundary-2} and \ref{fig:ladder-construction-boundary-2-twoways}.  
\end{definition}

If $W$ is boundary-fixed for a strand-set $S_{\partial \surfbord}$, then $W$ is not boundary-fixed for any strand-set $S^\prime_{\partial \surfbord}$ obtained by swapping two oppositely oriented strands of $S_{\partial \surfbord}$ on a boundary edge.  

For boundary-fixed webs, global parallel-moves can only be performed across embedded annuli, in contrast to rung-less  webs (\S \ref{sssec:rung-less-essential-webs}).  Global parallel-moves preserve the property of being essential.  We denote by $[\webbasis{\surfbord}](S_{\partial \surfbord})$ the collection of \textit{parallel-equivalence classes of boundary-fixed essential webs} for the strand-set $S_{\partial \surfbord}$.  (Note, by definition,  $[W] = [\emptyset] \in [\webbasis{\surfbord}](S_{\partial \surfbord})$ if and only if $S_{\partial \surfbord} = \emptyset$.)

		\subsubsection{Boundary-fixed Knutson-Tao cone}
		\label{sssec:boundary-fixed-knutson-tao-cone}

By Figure \ref{fig:triangle-hilbert-basis} (recall also property (2) of Definition \ref{def:local-coordinate-function}), a strand-set $S_E$ determines two local coordinates on a boundary edge $E$ of $\surfbord$.  More generally, a strand-set $S_{\partial \surfbord}$ for the surface fixes $2 * \#\{\text{boundary edges }E\}$  coordinates on the boundary $\partial \surfbord$.  See Figure \ref{fig:ladder-construction-boundary-2} for an example, where the fixed coordinates are colored~red.

\begin{definition}
	The \textit{boundary-fixed Knutson-Tao cone} $\KTcone{\idealtriang}(S_{\partial \surfbord}) \subset \KTcone{\idealtriang} \subset \Zpos^N$ with respect to a strand-set $S_{\partial \surfbord}$ is the subset of $\KTcone{\idealtriang}$ (as defined in \S \ref{sssec:knutson-tao-cone-associated-to-a-triangulation--boundarysetting}) consisting of points whose boundary coordinates agree with those determined by $S_{\partial \surfbord}$.  
\end{definition}

Note that, in contrast to boundary-fixed webs (\S \ref{sssec:boundary-fixed-essential-webs}), the boundary-fixed Knutson-Tao cone $\KTcone{\idealtriang}(S_{\partial \surfbord}) \subset \KTcone{\idealtriang}$ is independent of permuting the boundary strands of $S_{\partial \surfbord}$.  This is because the coordinates on a boundary component only depend on the number of in- and out-strands, not on their ordering along the edge (see property (2) in Definition \ref{def:local-coordinate-function}).

		\subsubsection{Coordinates for boundary-fixed essential webs}
		\label{sssec:coordinates-for-boundary-fixed-essential-webs}
		
The minimal position of a boundary-fixed web $W$ with respect to an ideal triangulation $\idealtriang$ is defined as in Definition \ref{def:minimal position}.  Proposition \ref{prop:minimal-position} holds word for word, except ``non-elliptic'' is replaced by ``boundary-fixed essential''.  	

Modified H-moves take boundary-fixed essential webs in good position to webs of the same type.  Proposition \ref{prop:good-position} holds, except ``non-elliptic'' is replaced by ``boundary-fixed essential''.  	

Given a strand-set $S_{\partial \surfbord}$ and an ideal triangulation $\idealtriang$, we define the Fock-Goncharov global coordinate function
	$\themap{\idealtriang}^\mathrm{FG}(S_{\partial \surfbord}) : [\webbasis{\surfbord}](S_{\partial \surfbord}) \to \KTcone{\idealtriang}(S_{\partial \surfbord}) \subset \KTcone{\idealtriang}$
 as in \S \ref{ssec:global-coordinates-from-local-coordinate-functions}; see Definition \ref{def:fock-goncharov-global-coordinate-function}.  	

\begin{theorem}[Second boundary result]
\label{thm:main-theorem-2}
	The Fock-Goncharov global coordinate function
	\begin{equation*}
		\themap{\idealtriang}^\mathrm{FG}(S_{\partial \surfbord}) : [\webbasis{\surfbord}](S_{\partial \surfbord}) \overset{\sim}{\longrightarrow} \KTcone{\idealtriang}(S_{\partial \surfbord}) \subset \KTcone{\idealtriang} \subset \Zpos^N
	\end{equation*}
	with respect to the strand-set $S_{\partial \surfbord}$ is a bijection of sets,  identifying parallel-equivalence classes of boundary-fixed essential webs with points of the boundary-fixed Knutson-Tao cone.  
\end{theorem}

\begin{proof}
	As in the proof of Theorem \ref{thm:main-theorem-1}, the strategy is to construct an explicit inverse
	\begin{equation*}
		\invmap{\idealtriang}^\mathrm{FG}(S_{\partial \surfbord}) :  \KTcone{\idealtriang}(S_{\partial \surfbord})  \longrightarrow   [\webbasis{\surfbord}](S_{\partial \surfbord}). 
	\end{equation*}
The mapping $\invmap{\idealtriang}^\mathrm{FG}(S_{\partial \surfbord})$ is defined via the ladder gluing construction followed by removing internal elliptic faces, as explained in \S \ref{ssec:ladder-gluing-construction}-\ref{ssec:inverse-mapping-correcting-an-elliptic-web}.  In contrast to the rung-less setting (\S \ref{sssec:coordinates-for-rung-less-essential-webs--boundarysetting}), no new reductions are required.  For examples, see Figures \ref{fig:ladder-construction-boundary-2} and \ref{fig:ladder-construction-boundary-2-twoways}; compare the empty-boundary case, Figures \ref{fig:coordinates-example2} and \ref{fig:elliptic-web-non-example}, and the rung-less boundary case, Figures \ref{fig:ladder-construction-boundary-1} and \ref{fig:ladder-construction-boundary-1-twoways}.  

We also need the analogue of Main Lemma \ref{lem:main-lemma}, saying that two boundary-fixed essential webs resulting from the ladder gluing construction are parallel-equivalent.  The proof of the main lemma is essentially unchanged from \S \ref{sec:proof-of-main-lemma}.  To say a word about it, for the proof of Corollary \ref{cor:same-number-of-intersection-points}, if a shared-route for $W$ ending on the boundary $\partial \surfbord$ is crossing, then the corresponding shared-route for $W^\prime$ is also crossing, by the boundary-fixed condition.  
\end{proof}

\begin{corollary}
	Let two strand-sets $S_{\partial \surfbord}$ and $S^\prime_{\partial \surfbord}$ be the same up to permuting strands lying on the same boundary edge.  Then, there is a natural one-to-one correspondence
	\begin{equation*}
	\themap{\idealtriang}^\mathrm{FG}(S^\prime_{\partial \surfbord})^{-1} \circ \themap{\idealtriang}^\mathrm{FG}(S_{\partial \surfbord}) :
	[\webbasis{\surfbord}](S_{\partial \surfbord}) \overset{\sim}{\longrightarrow} \KTcone{\idealtriang}(S_{\partial \surfbord}) = 
	\KTcone{\idealtriang}(S^\prime_{\partial \surfbord}) \overset{\sim}{\longrightarrow}   [\webbasis{\surfbord}](S^\prime_{\partial \surfbord})
	\end{equation*}
	sending parallel-equivalence classes of boundary-fixed essential webs $W$ for $S_{\partial \surfbord}$ to parallel-equivalence classes of boundary-fixed essential webs $W^\prime$ for $S^\prime_{\partial \surfbord}$.  Here, \textit{natural} means that the resulting bijection $[\webbasis{\surfbord}](S_{\partial \surfbord}) \to [\webbasis{\surfbord}](S^\prime_{\partial \surfbord})$ is independent of the choice of  triangulation~$\idealtriang$.  
\end{corollary}

\begin{proof}
	Let $W$ be a boundary-fixed essential web for $S_{\partial \surfbord}$, and let ${\idealtriang_1}$ and ${\idealtriang_2}$ be two ideal triangulations.  We claim that there are webs $W_{\idealtriang_1}$ and $W_{{\idealtriang_2}}$ isotopic to $W$ and in good position for the split ideal triangulation $\splitidealtriang_1$ and $\splitidealtriang_2$, respectively, such that $W_{\idealtriang_1}$ and $W_{{\idealtriang_2}}$ have the same ladders in the boundary biangles $\biang$ facing $\partial \surfbord$.  
	
	Indeed, by pushing as many H's of $W$ as possible into the boundary biangles $\biang$, we may assume that the web $\widetilde{W}$ obtained from $W$ by chopping off the boundary biangles $\biang$ is rung-less essential.  By \S \ref{sssec:coordinates-for-rung-less-essential-webs--boundarysetting}, there exist webs $\widetilde{W}_{\idealtriang_1}$ and $\widetilde{W}_{{\idealtriang_2}}$ isotopic to $\widetilde{W}$ that are in good position for $\splitidealtriang_1$ and $\splitidealtriang_2$, respectively.  Let $W_{\idealtriang_1}$ and $W_{{\idealtriang_2}}$ be obtained by re-attaching the ladders from the cut-off boundary biangles $\biang$ to $\widetilde{W}_{\idealtriang_1}$ and $\widetilde{W}_{{\idealtriang_2}}$, respectively.  This proves the claim.  
	
	To finish, the map  $[\webbasis{\surfbord}](S_{\partial \surfbord}) \to [\webbasis{\surfbord}](S^\prime_{\partial \surfbord})$ for ${\idealtriang_1}$ is computed via the following steps:  (1)  erase the ladders of $W_{\idealtriang_1}$ from the boundary biangles $\biang$; (2)  replace $S_{\partial \surfbord}$ with $S^\prime_{\partial \surfbord}$ by permuting boundary strands;  (3)  insert the unique ladders into the boundary biangles $\biang$ matching this new boundary data;  (4)  eliminate elliptic faces.  Similarly for the map with respect to ${\idealtriang_2}$.  By the claim, the webs for ${\idealtriang_1}$ and ${\idealtriang_2}$ resulting after step (3) are isotopic.  
\end{proof}
		
\begin{figure}[t]
	\centering
	\includegraphics[width=.56\textwidth]{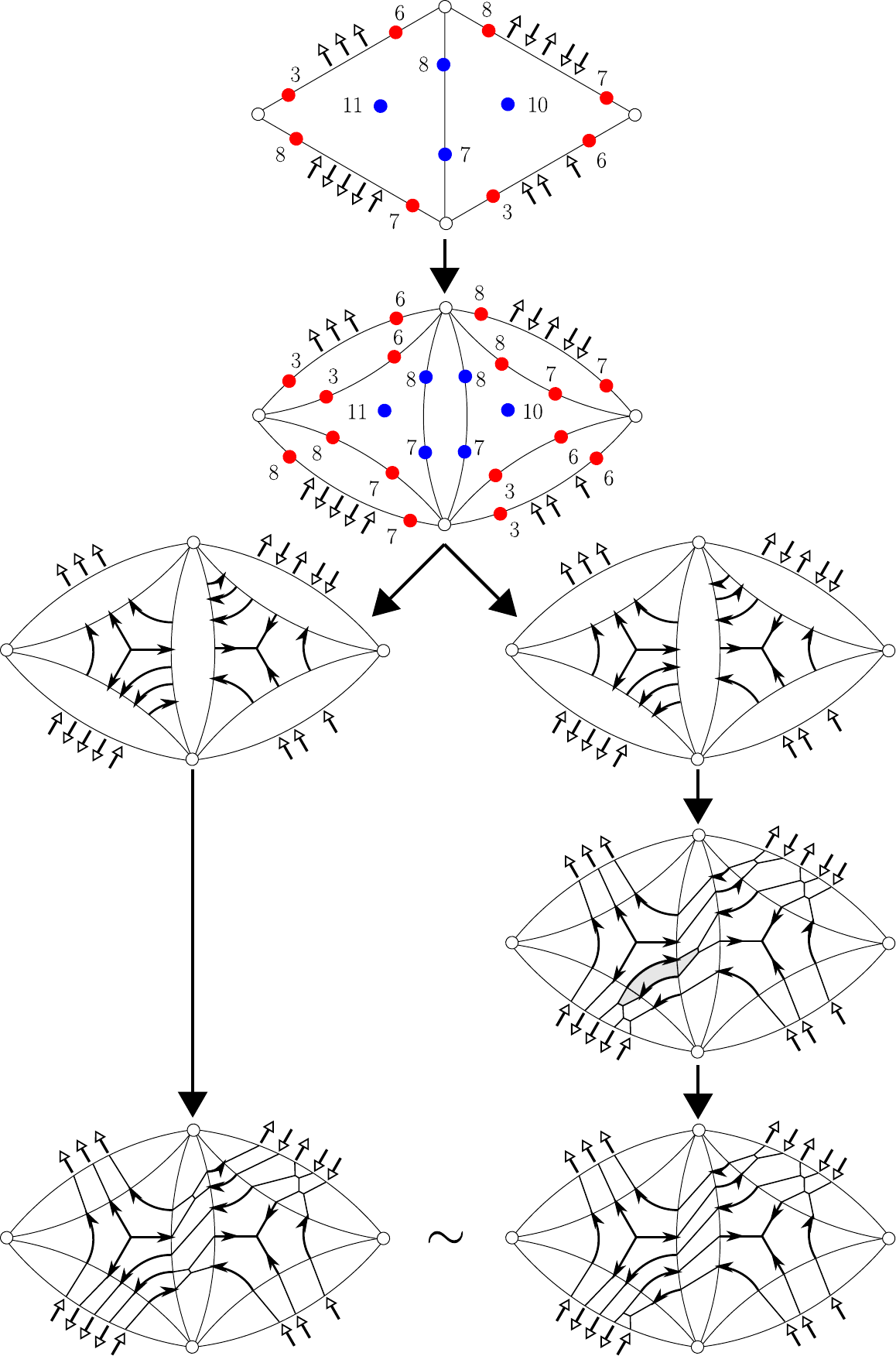}
	\caption{Ladder construction for boundary-fixed essential webs: 1 of 2 (on the ideal square)}
	\label{fig:ladder-construction-boundary-2}
\end{figure}

\begin{figure}[htb]
	\centering
	\includegraphics[width=.6\textwidth]{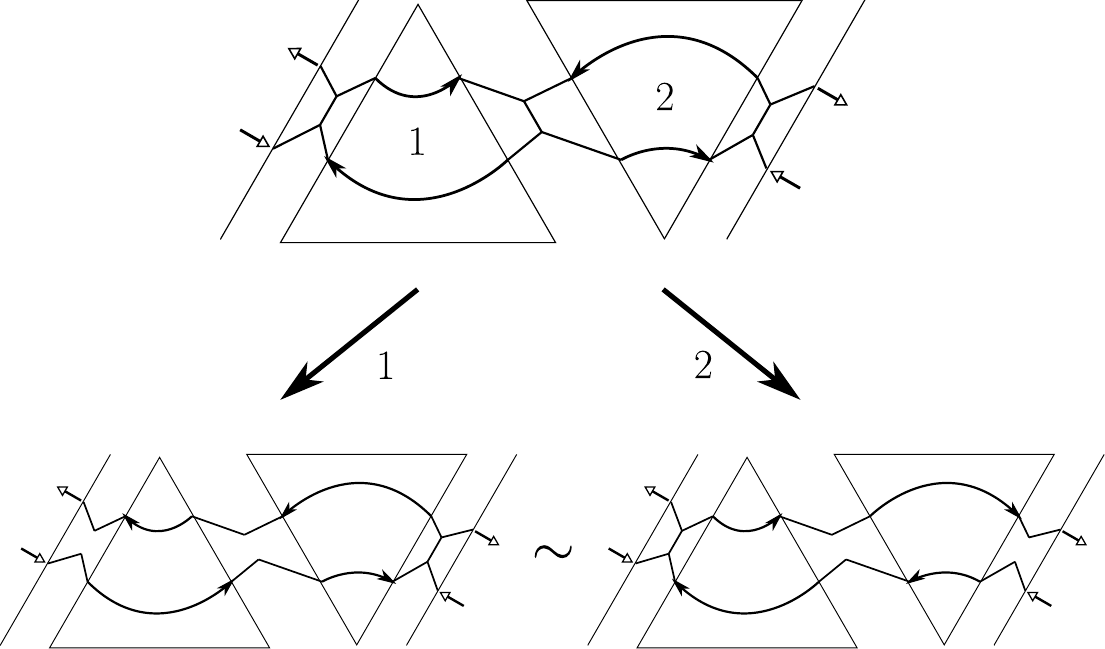}
	\caption{Ladder construction for boundary-fixed essential webs: 2 of 2}
	\label{fig:ladder-construction-boundary-2-twoways}
\end{figure}

		\subsection{Application: representation theory of the Lie group \texorpdfstring{$\mathrm{SL}_3(\mathbb{C})$}{SL3}}
		\label{ssec:application-representation-theory-of-the-lie-group-sl3c}
		
As a consequence of the second version of the result, Theorem \ref{thm:main-theorem-2}, we make a connection to Kuperberg's famous theorem relating webs in the disk to the representation theory of $\mathrm{SL}_3(\mathbb{C})$.  
		
The finite-dimensional irreducible representations of $\mathrm{SL}_3(\mathbb{C})$ are in one-to-one correspondence with ordered pairs $(n^\mathrm{in}, n^\mathrm{out}) \in \Zpos^2$. For example, we may say that $(1, 0)$ corresponds to the defining vector representation $V$ and $(0, 1)$ corresponds to its dual representation $V^*$.  
		
We assign to each strand-set $S_{\partial \surfbord}$ a tensor product $V(S_{\partial \surfbord}) = \otimes_E V_E$ of finite-dimensional irreducible representations $V_E$ of $\mathrm{SL}_3(\mathbb{C})$, varying over boundary edges $E$ of $\surfbord$, as follows.  If $n^\mathrm{in}_E$ (resp. $n^\mathrm{out}_E$) is the number of in-strands (resp. out-strands) of $S_E$, where $S_{\partial \surfbord} = \{ S_E \}_E$, then we define $V_E$ to be the irreducible representation corresponding to $(n^\mathrm{in}_E, n^\mathrm{out}_E)$.  
		
Note $V(S_{\partial \surfbord}) = V(S^\prime_{\partial \surfbord})$ if $S_{\partial \surfbord}$ and $S^\prime_{\partial \surfbord}$ are the same up to permuting strands on an $E$.  
		
For a representation $V$ of $\mathrm{SL}_3(\mathbb{C})$, let $V^{\mathrm{SL}_3(\mathbb{C})} \subset V$ be the sub-space of invariant vectors.    
		
\begin{theorem}[{\cite[Theorem 6.1]{KuperbergCommMathPhys96}}]
\label{thm:kuperberg-theorem}
	For $\surfbord = \poly_k$ ($k \geq 1$) the ideal polygon with $k$ boundary edges ({\upshape\S \ref{ssec:ideal-polygons}}), the collection $[\webbasis{\poly_k}](S_{\partial \poly_k})$ of classes $[W]$ of boundary-fixed essential webs with respect to a strand-set $S_{\partial \poly_k}$ indexes a linear basis for the invariant space $V(S_{\partial \poly_k})^{\mathrm{SL}_3(\mathbb{C})}$.  
\end{theorem}

Note that, for $\surfbord = \poly_k$, a parallel-equivalence class $[W] \in [\webbasis{\poly_k}](S_{\partial \poly_k})$ is an isotopy class.    

From Kuperberg's theorem, together with Theorem \ref{thm:main-theorem-2}, we immediately obtain:

\begin{corollary}[Application of the second boundary result]
\label{cor:sl3-clebsch-gordan}
	For $\surfbord = \poly_k$ ($k \geq 3$), a strand-set $S_{\partial \poly_k}$, and an ideal triangulation $\idealtriang$ of $\poly_k$, the boundary-fixed Knutson-Tao cone $\KTcone{\idealtriang}(S_{\partial \poly_k}) \subset \KTcone{\idealtriang} \subset \Zpos^N$ indexes a linear basis for the invariant space $V(S_{\partial \poly_k})^{\mathrm{SL}_3(\mathbb{C})}$.  \qed
\end{corollary}

\newpage		
\begin{remark}\label{rem:lastrem}      $  $
\begin{enumerate}
	\item
	This corollary is reminiscent of results about the Knutson-Tao hive model \cite{KnutsonJAmerMathsoc99, BuchEnseignMath00} for the general linear group $\mathrm{GL}_n(\mathbb{C})$, where the Littlewood-Richardson coefficients $c_{\lambda \mu}^\nu$ associated to highest weights $\lambda, \mu, \nu$ provide the multiplicities of irreducible representations $V_\nu$ in $V_\lambda \otimes V_\mu$.  Certain multiplicities $c_{\lambda \mu}^\nu$ can be computed as the number of solutions of the Knutson-Tao rhombus inequalities, without $n$-congruence conditions (see Remark \ref{rem:modulo-3-congruence-conditions}(\ref{subrem:GS1})), on the dotted $n$-triangle matching certain fixed boundary conditions determined by the weights $\lambda, \mu, \nu$.  (Possibly related, see \cite{MageeProcLondMathSoc20}.)
	\item\label{subrem:kup2}
	Corollary \ref{cor:sl3-clebsch-gordan} was the result of Kuperberg's theorem combined with Theorem \ref{thm:main-theorem-2}.  We would like to have gone in the other direction.  That is, we would like to give an alternative geometric proof of Kuperberg's theorem, as a consequence of Theorem \ref{thm:main-theorem-2} and Corollary \ref{cor:sl3-clebsch-gordan}.  Indeed, this was the spirit of Kuperberg's proof for the $\mathrm{SL}_2(\mathbb{C})$-version of his result \cite[Theorem 2.4]{KuperbergCommMathPhys96}, where the $\mathrm{SL}_2(\mathbb{C})$-analogue	of Corollary \ref{cor:sl3-clebsch-gordan} is a simple consequence of the Clebsch-Gordan theorem.    It is natural then to ask:
\end{enumerate}  
\end{remark}

\begin{question}
	Is there an alternative purely representation theoretic proof of Corollary \ref{cor:sl3-clebsch-gordan}?
\end{question}

\begin{question}
Is there a representation theoretic interpretation of Theorem \ref{thm:main-theorem-2} for any surface-with-boundary $\surfbord$, generalizing Kuperberg's theorem in the case $\surfbord = \poly_k$?  (We ask this question also for $\mathrm{SL}_2(\mathbb{C})$.)  Possible clues may lie in \cite{FominPNAS14, MR4493620costantinole, GoncharovArxiv19}.  
\end{question}

\bibliographystyle{alpha}
\bibliography{references.bib}
\end{document}